\documentclass[11pt]{amsart}
\usepackage{amsfonts, amsmath, amssymb, graphicx, epic, eepic, epsf}

\input xy
\xyoption{all}

\newtheorem{theorem}{Theorem}
\theoremstyle{plain}

\newtheorem{corollary}{Corollary}

\newtheorem{definition}{Definition}

\newtheorem{lemma}{Lemma}

\newtheorem{proposition}{Proposition}
\newtheorem{remark}{Remark}

\numberwithin{equation}{section}

\newcommand{\brak}[1]{\langle #1\rangle}
\DeclareMathOperator{\Hom}{Hom}
\DeclareMathOperator{\id}{Id}
\def\co{\colon\thinspace} 
\def\mf{\mathfrak}

\newskip\stdskip                      
\begin{document}
\title{The universal $\mf{sl}(2)$ cohomology via webs and foams}
\author{Carmen Caprau}
\address{Department of Mathematics, California State University, Fresno}
\email{ccaprau@csufresno.edu}
\date{}

\begin{abstract}
We construct the universal $\mf{sl}(2)$-tangle cohomology using an \linebreak approach with webs and dotted foams. This theory depends on two parameters, and for the case of links it is a categorification of the unnormalized Jones polynomial of the link. 
\end{abstract}

\maketitle

\section{\textbf{Introduction}}

Khovanov classified in~\cite{Kh2} all possible Frobenius systems of rank two that give rise to link homologies via his original construction in~\cite{Kh1}, and showed that there is a universal one corresponding to $\mathbb{Z}[X, a, h]/(X^2 - hX -a),$ where $a$ and $h$ are formal variables. Using Bar-Natan's~\cite{BN1} approach to local Khovanov homology and Khovanov's work in~\cite{Kh3}, the author showed in~\cite{me} how to construct a bigraded tangle cohomology theory depending on one parameter, via a setup with webs and foams---singular cobordisms---modulo a finite set of relations; see also~\cite{CC} for a longer, more detailed version of~\cite{me}. The construction corresponds to a Frobenius algebra structure defined on $\mathbb{Z}[i][X,a]/(X^2-a),$ and for the case of links it is a categorification of the quantum $\mf{sl}(2)$-link invariant, thus of the unnormalized Jones polynomial of the link. Adding the relation $a=0$ or $a=1$ yields an isomorphic version of the $\mf{sl}(2)$ Khovanov homology~\cite{Kh1} or Lee's theory~\cite{L}, respectively.

The advantage of working with webs and foams instead of classical $(1+1)$--dimensional cobordisms, and of considering the fourth root of unity $i$ in the ground ring is that the construction brings up a theory that satisfies functoriality property in a proper sense, that is, with no sign ambiguity. In particular, it resolves the sign indeterminacy in the functoriality property of the Khovanov homology (see Bar-Natan~\cite{BN1}, Jacobsson~\cite{J} or Khovanov~\cite{Kh4} for proofs of the functoriality of Khovanov's invariant). 

In the first half of this paper, we generalize the construction described in~\cite{me} to obtain the \textit{universal} $\mf{sl}(2)$-link cohomology---universal in the sense of~\cite{Kh2}---given by $\mathbb{Z}[i][X,a,h]/(X^2 - hX -a),$ where $a$ and $h$ are formal parameters. The invariant of a tangle is a complex of graded free $\mathbb{Z}[i][a, h]$-modules, up to cochain homotopy, and its cohomology is a bigraded tangle cohomology theory. The good news is that the generalized theory still satisfies the functoriality property with no sign indeterminacy. 

Besides generalizing the construction in~\cite{me} and thus obtaining a better invariant, we go further in the second half of the paper to show more insights about the new theory. Since any surface-link can be regarded as a cobordism between empty links, our construction yields an invariant of such surfaces. We prove that the invariant of a surface-knot or surface-link depends only on its genus. We also work over $\mathbb{C}$ and take $a$ and $h$ to be complex numbers, instead of formal parameters. Inspired by the work of Mackaay and Vaz~\cite{MV}, we show that there are two isomorphism classes of the invariant of a link,  depending on the number of distinct roots of the polynomial $f(X) = X^2 - hX -a.$

We remark that the universal $\mf{sl}(2)$-link cohomology is equivalent to a perturbed Khovanov-Rozansky cohomology~\cite{KhR1} for $n =2,$ in which the ``potential'' $\omega(x) = x^3$ is replaced by $\omega(x, a, h) = x^3 -\frac{3}{2} hx^2 -3ax.$ The new potential is homogeneous of degree 6, with $\deg (a) = 4, \deg (h) = 2$ and $\deg(x) = 2.$ The corresponding perturbed Khovanov-Rozansky theory was described in~\cite{CC1}.

\section{\textbf{The $\mf{sl}(2)$-link invariant via webs}}\label{sec:webs}

For our purpose, we are interested in working with the $\mf{sl}(2)$-link invariant via an approach with \textit{webs}, and for this, we consider the oriented state model for the Jones polynomial instead of the classical approach. A web with boundary $B$ is a planar graph $\Gamma$---properly embedded in a disk $\mathcal{D}^2$---with bivalent vertices near which the edges are either oriented ``in'' $\raisebox{-4pt}{\includegraphics[height=.17in]{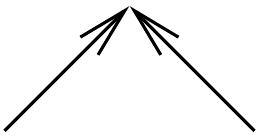}}$ or ``out'' $\raisebox{-4pt}{\includegraphics[height=.17in]{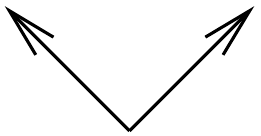}}$, and with univalent vertices that lie on the boundary of the disk $\mathcal{D}^2$. A closed web is a web with empty boundary ($B = \emptyset$). We also allow webs without vertices, which are oriented loops.

There is an ordering of the edges that join at a bivalent vertex, in the sense that each such vertex has a \textit{preferred edge}. If two edges oriented south-north share a bivalent vertex which is a `sink' or a `source', then the edge that goes in or goes out from the right, respectively, is the preferred edge of the corresponding bivalent vertex. If the two edges that share a bivalent vertex are oriented north-south, then one has to replace in the above definition the word ``right'' by ``left''. Two adjacent bivalent vertices are called \textit{of the same type} if the edge they share is either the preferred one or not, for both of them. For example, Figure~\ref{fig:types of vertices}(a) shows vertices of the same type; in the first drawing, the two vertices share their preferred edge, while in the second drawing, the preferred edges are on the sides of the picture. Otherwise, the vertices are called \textit{of different type}, as those given in Figure~\ref{fig:types of vertices}(b).

\begin{figure}[ht!]\begin{center}
\[(\text{a}) \begin{array}{c}\raisebox{-4pt}{\includegraphics[height=0.18in]{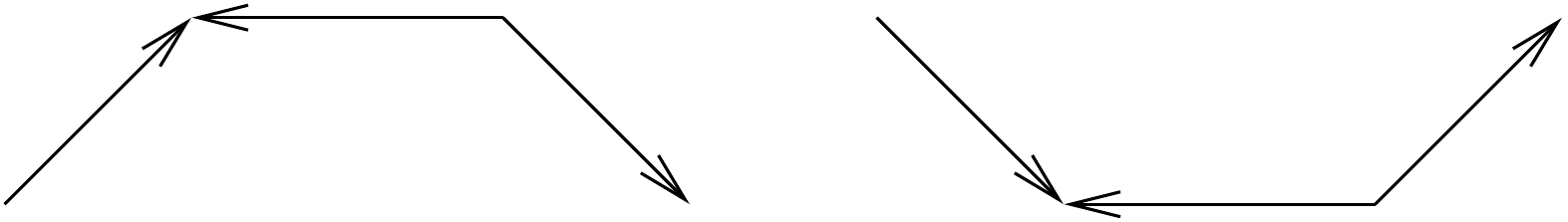}} \end{array} \qquad (\text{b})\begin{array}{c}\raisebox{-4pt}{\includegraphics[height=0.2in]{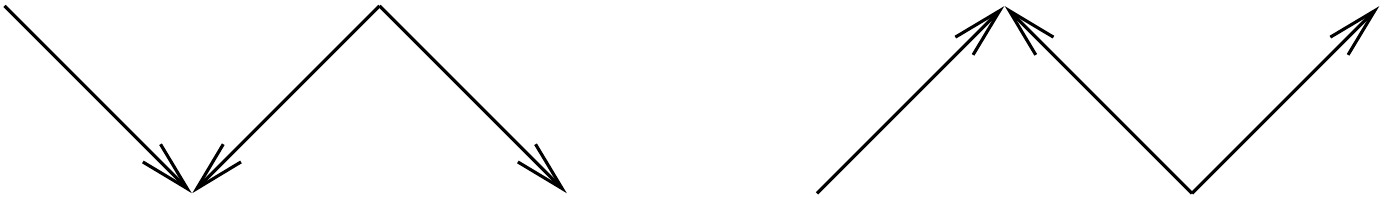}}\end{array}\]
\end{center}
\caption{Types of vertices}\label{fig:types of vertices}\end{figure}

Let $L$ be a link in $S^3.$ We fix a generic planar diagram $D$ of $L$ and replace each of its crossings by one of two planar pictures on the right:
\[ \raisebox{-15pt}{\includegraphics[height=0.5in]{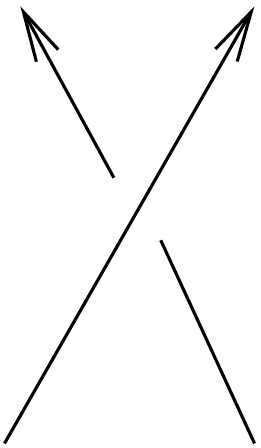}} \longrightarrow \raisebox{-15pt} {\includegraphics[height=0.5in]{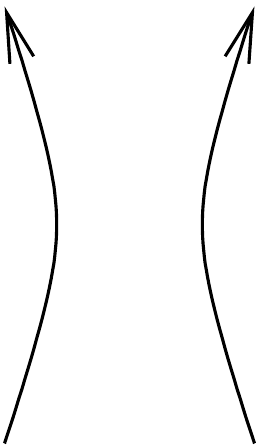}} \quad \text{and} \quad \raisebox{-15pt} {\includegraphics[height=0.5in]{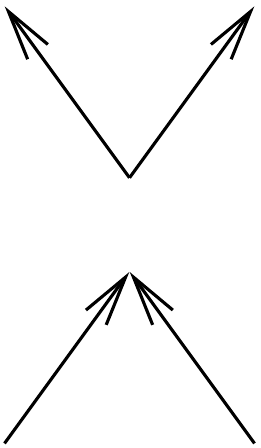}} \]
We call the resolution on the left the \textit{oriented resolution}, while the one on the right the \textit{singular resolution}. A diagram $\Gamma$ obtained by resolving all crossings of $D$ is a disjoint union of closed webs. Notice that for each singular resolution as depicted above, the preferred edges of the two bivalent vertices are on their right side. There is a unique way to associate a Laurent polynomial $\brak{\Gamma}$ to any closed web, so that it satisfies the web skein relations given in Figure~\ref{fig:web skein relations}.

\begin{figure}[ht!]
$$\xymatrix@R = 2mm{
\brak{\raisebox{-5pt}{\includegraphics[height=0.2in]{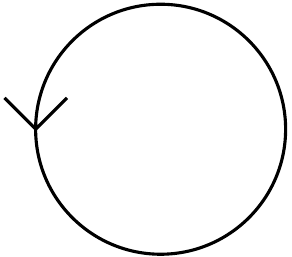}} \bigcup \Gamma} = (q + q^{-1}) \brak{\Gamma} =\brak{\raisebox{-5pt}{\includegraphics[height=0.2in]{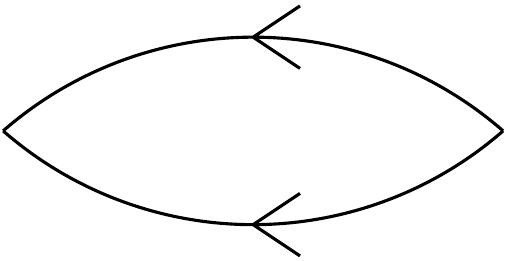}} \bigcup \Gamma} \\
\brak{\raisebox{-5pt}{\includegraphics[height=0.12in]{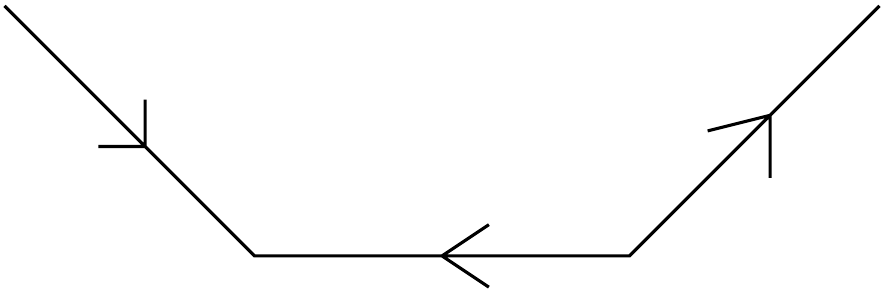}}} = 
\brak{\raisebox{-5pt}{\includegraphics[height=0.12in]{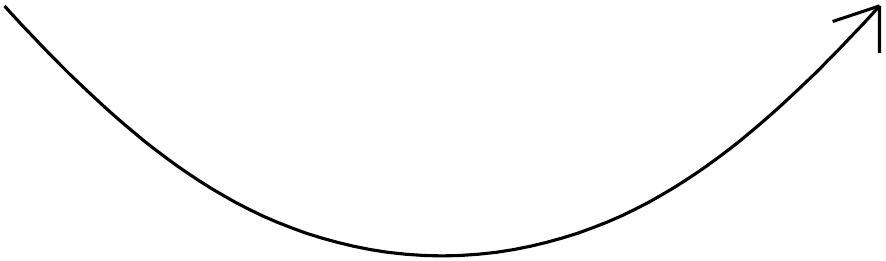}}}, \quad
\brak{\raisebox{-5pt}{\includegraphics[height=0.12in]{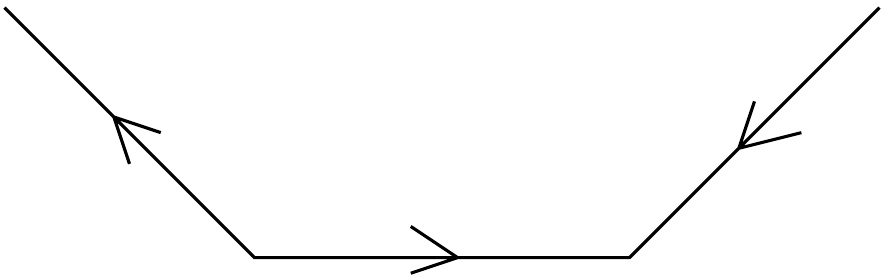}}} = 
\brak{\raisebox{-5pt}{\includegraphics[height=0.12in]{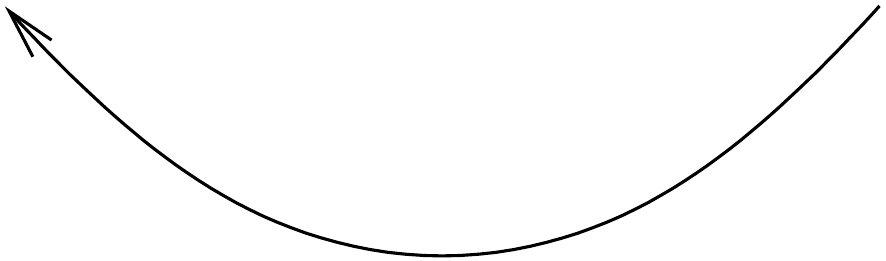}}}
}$$
\caption{Web skein relations}
\label{fig:web skein relations}
\end{figure}

We define $P_2(D) = \sum_{\Gamma} \pm q^{\alpha(\Gamma)}\brak{\Gamma},$ where the sum is over all resolutions of $D,$ and the exponents $\alpha(\Gamma)$ and the sign $\pm$ are determined by the relations in Figure~\ref{fig:decomposition of crossings}.  
\begin{figure}[ht!]
\begin{center}
$\xymatrix@R=2mm{
\raisebox{-13pt}{\includegraphics[height=0.4in]{poscrossing.pdf}} \,= \,q\,\,\,\,\,\raisebox{-13pt} {\includegraphics[height=0.4in]{orienres.pdf}} - q^2\,\,\,\,\,\raisebox{-13pt} {\includegraphics[height=0.4in]{singres.pdf}}\\
\raisebox{-13pt}{\includegraphics[height=0.4in]{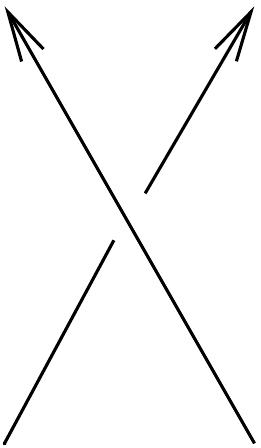}} \,= \,q^{-1}\,\raisebox{-13pt} {\includegraphics[height=0.4in]{orienres.pdf}} - q^{-2}\, \raisebox{-13pt} {\includegraphics[height=0.4in]{singres.pdf}}
}$\end{center}
\caption{Decomposition of crossings}
\label{fig:decomposition of crossings}
\end{figure}

If $D_1$ and $D_2$ are link diagrams that are related by a Reidemeister move, then $P_2(D_1) = P_2(D_2),$ hence $ P_2(L)= P_2(D)$ is an invariant of the oriented link $L.$ Excluding rightmost terms from the equations in Figure~\ref{fig:decomposition of crossings}, we obtain the well-known skein relation for computing the quantum $sl(2)$-link invariant:
$$q^2 \,\,\raisebox{-13pt}{\includegraphics[height=.4in]{negcrossing.pdf}} - q^{-2}\,\, \raisebox{-13pt}{\includegraphics[height=.4in]{poscrossing.pdf}} = (q - q^{-1}) \,\,\raisebox{-13pt}{\includegraphics[height=.4in]{orienres.pdf}}$$ 

\section{\textbf{The category of foams}}\label{sec:foams}

A \textit{foam} is a cobordism between two webs $\Gamma_0$ and $\Gamma_1$ with boundary $B,$ regarded up to boundary-preserving isotopy. More precisely, a foam is a  piecewise oriented  2-dimensional manifold $S$ with boundary $ \partial{S } = -\Gamma_1 \cup \Gamma_0 \cup B \times [0,1]$ and corners $B \times \{0\} \cup B \times \{1\},$ where the manifold $-\Gamma _1$ is $\Gamma_1$ with the opposite orientation. All foams are bounded within a cylinder, and the part of their boundary on the sides of the cylinder is the union of vertical straight lines. If $\Gamma_0$ and $\Gamma_1$ are closed webs, a foam from $\Gamma_0$ to $\Gamma_1$ is embedded in $\mathbb{R}^2 \times [0,1]$ and its boundary lies entirely in $\mathbb{R}^2 \times \{0,1\}.$ We read foams as morphisms from bottom to top by convention, and we compose them by placing one cobordism on top the other.

Foams have \textit{singular arcs} (and/or \textit{singular circles}) where orientations disagree. The two facets on the two sides of a singular arc have opposite orientation, and because of this, the two facets induce the same orientation on the singular arc. Specifically, the orientation of singular arcs is as in Figure~\ref{fig:saddle}, which shows examples of piecewise oriented saddles.

\begin{figure}[ht!]
\begin{center}
\includegraphics[height=.65in]{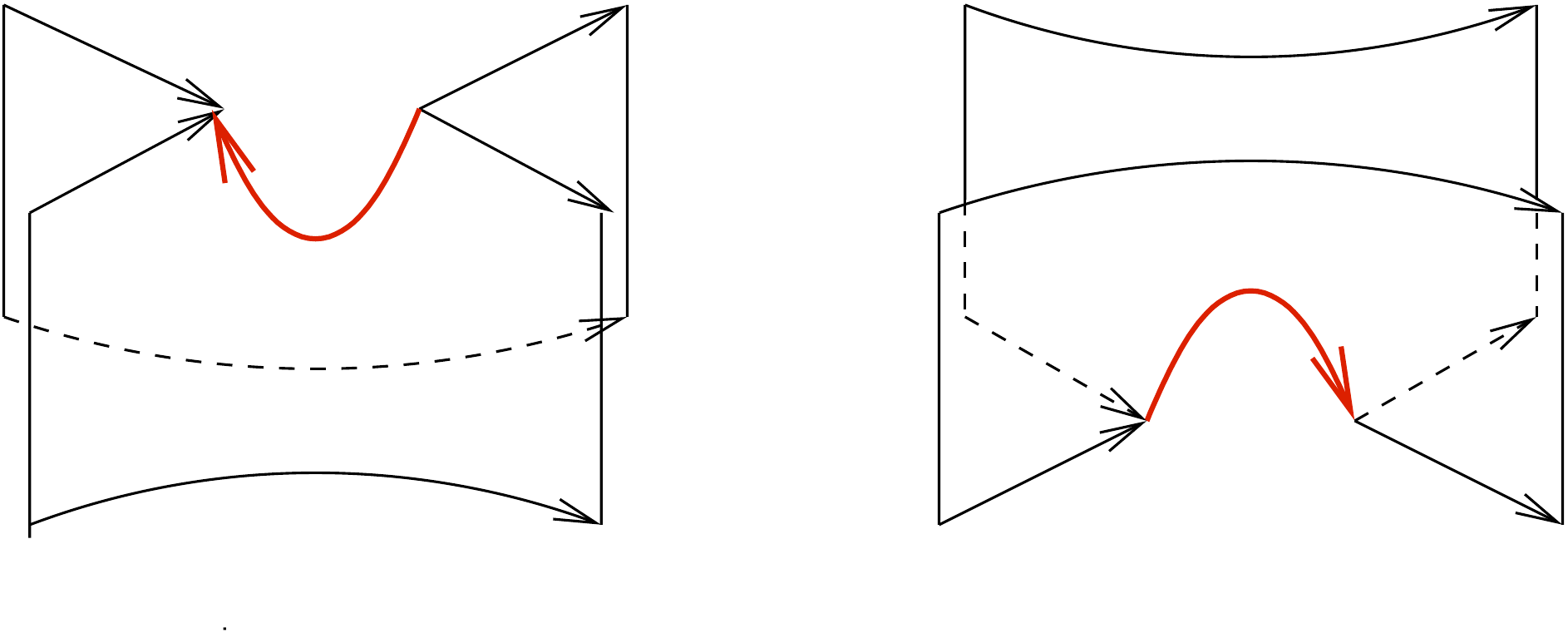}\end{center}
\caption{Examples of piecewise oriented saddles}
\label{fig:saddle}
\end{figure}

For each singular arc, there is an ordering of the facets that are incident with it, in the sense that one of the facets is the \textit{preferred facet} for the corresponding singular arc. This ordering is induced by the ordering of edges at bivalent vertices, in the following sense: the preferred facet of a singular arc contains in its boundary the preferred edges of the two bivalent vertices that it connects. In particular, a pair of bivalent vertices can be connected by a singular arc only if the above rule is satisfied. 
If the preferred facet of a singular arc is at its left (where the concept of ``left'' and ``right'' is given by the orientation of the singular arc), then we will usually represent the singular arc by a continuous red curve. Otherwise, it will be represented by a dashed red curve. In Figure~\ref {fig:ordering of facets} we have two examples of foams with boundary and singular arcs (we labeled with 1 the preferred facets for the given singular arcs). 
\begin{figure}[ht!]\begin{center}
\raisebox{-8pt}{\includegraphics[height=.5in]{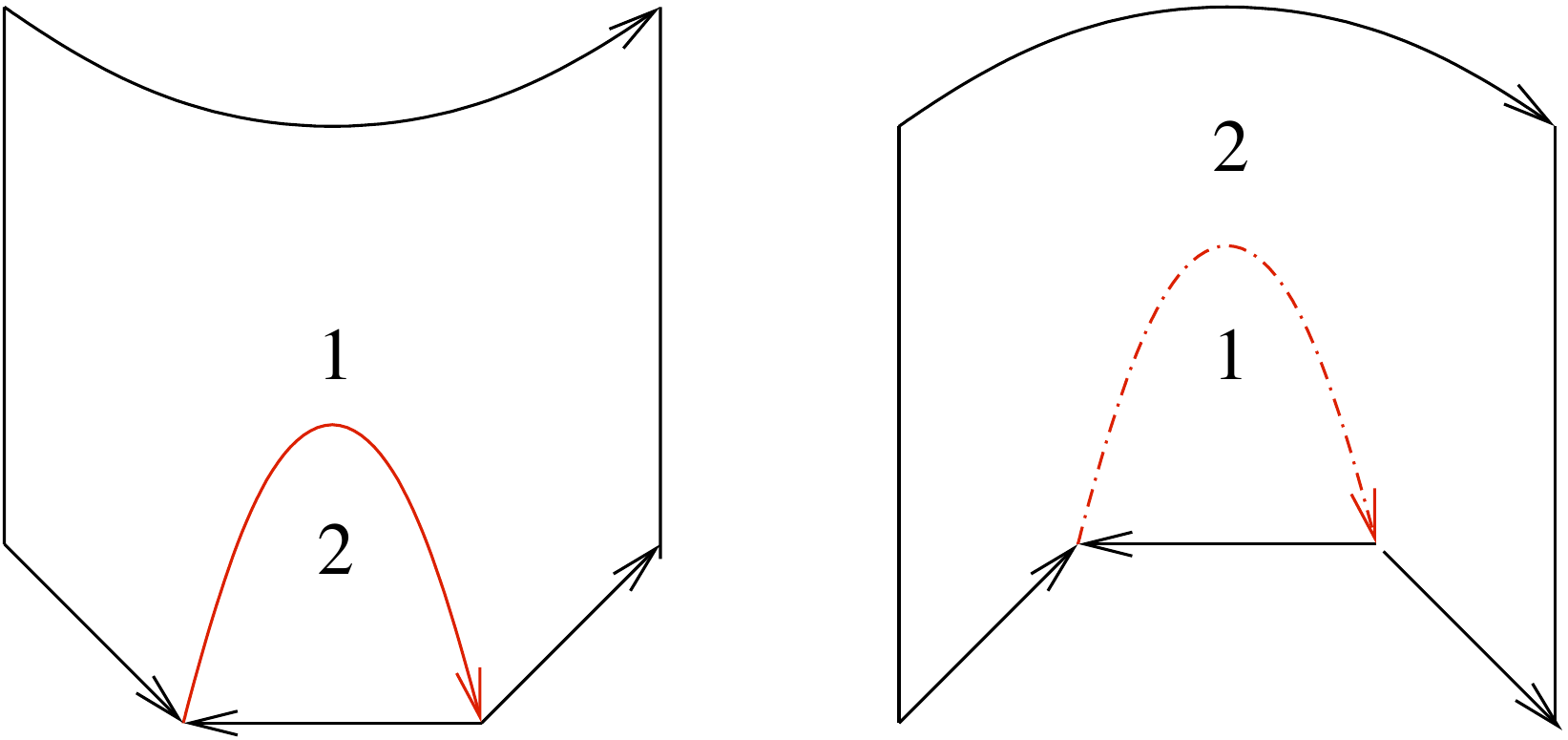}}\end{center}
\caption{Ordering of facets near a singular arc}\label{fig:ordering of facets}
\end{figure}

We say that two foams are \textit{isomorphic} if they differ by an isotopy during which the boundary is fixed. A cobordism from the empty web to itself gives rise to a foam with empty boundary, therefore a \textit{closed foam}. Figure~\ref{fig:ordering of facets of ufo} shows an example of a closed foam, called the \textit{ufo}-foam, with the two choices of ordering of its facets. In what follows, we fix the ordering of the \textit{ufo}'s facets as shown in the picture on the left.

\begin{figure}[ht!]\begin{center}
\raisebox{-8pt}{\includegraphics[height=.4in]{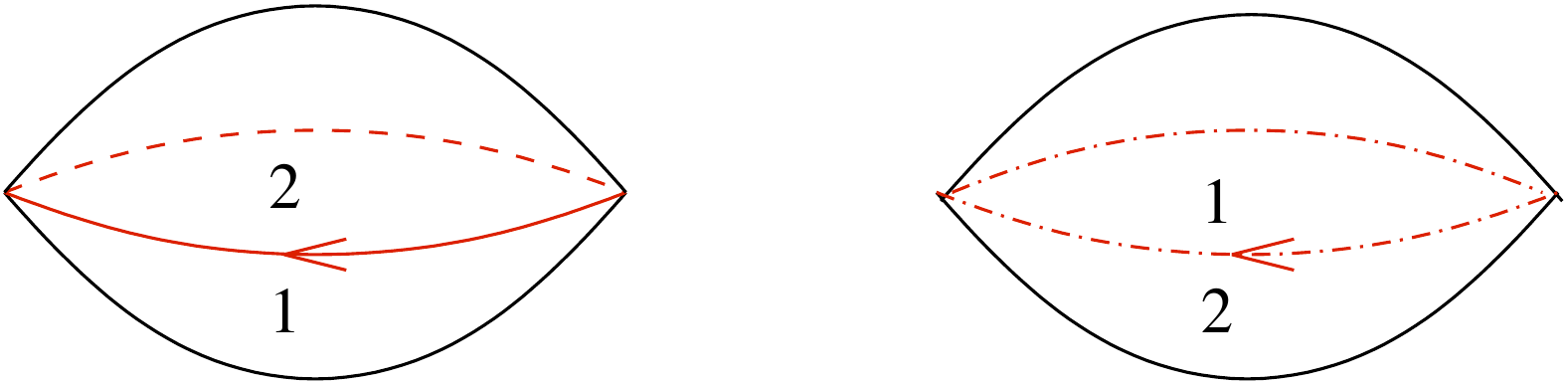}}\end{center}
\caption{ufo-foams and the ordering of their facets}\label{fig:ordering of facets of ufo}
\end{figure}
Finally, foams can have dots that are allowed to move freely along the facet they belong to, but can't cross singular arcs.

We denote by $\textit{Foams}(B)$ (or $\textit{Foams}(\emptyset)$) the category whose objects are web diagrams with boundary $B$ (or empty boundary) and whose morphisms are foams between such webs. We use notation $\textit{Foams}$ as a generic reference either to $\textit{Foams}(\emptyset)$ or to $\textit{Foams}(B),$ for some finite set $B.$


\subsection{A (1 + 1)-dimensional TQFT with dots} \label{sec:TQFT} 

Consider the polynomial ring $\mathbb{Z}[i][a, h]$ with Gaussian integer coefficients, and define a grading on it by letting $\deg (1) = 0 = \deg(i), \deg(a) = 4$, and $\deg(h) = 2.$

Let $\mathcal{A} = \mathbb{Z}[i][a, h, X]/(X^2 -hX- a)$ be the $\mathbb{Z}[i][a, h]$-module with generators $1$ and $X,$ and with inclusion map $\iota \co \mathbb{Z}[i][a, h] \rightarrow \mathcal{A}, \iota(1) = 1.$ The ring $\mathcal{A}$ is commutative Frobenius with the trace map $ \epsilon \co \mathcal{A} \rightarrow \mathbb{Z}[i][a, h], \text{where}\,\, \epsilon(1) = 0, \, \epsilon(X) = 1.$

Multiplication $m \co \mathcal{A} \otimes \mathcal{A} \rightarrow \mathcal{A}$ and comultiplication $\Delta \co \mathcal{A} \rightarrow \mathcal{A} \otimes \mathcal{A}$ are defined by
$$ \begin{cases}
 m(1 \otimes X) =X, & m(X \otimes 1) = X\\ 
m(1 \otimes 1) =1, & m(X \otimes X) =hX+ a
 \end{cases},\quad
 \begin{cases}
 \Delta(1) = 1 \otimes X + X \otimes 1-h1\otimes 1\\ 
 \Delta(X) = X \otimes X + a 1 \otimes 1.\end{cases} $$

We make $\mathcal{A}$ graded by setting $\deg(1) = -1$ and $\deg(X) = 1.$ The multiplication and comultiplication are maps of degree $1,$ while $\epsilon$ and $\iota$ are maps of degree $-1.$

The commutative Frobenius algebra $\mathcal{A}$ gives rise to a functor---denoted here by $\mathsf{F}$---from the category of oriented $(1+1)$--dimensional cobordisms to the category of graded $\mathbb{Z}[i][a, h]$-modules. The functor assigns the ground ring $\mathbb{Z}[i][a, h]$ to the empty 1-manifold, and $\mathcal{A}^{\otimes k}$ to the disjoint union of oriented $k$ circles (the tensor product is taken over $\mathbb{Z}[i][a,h]$). 

On the generating morphisms of the category of oriented $(1+1)$-cobordisms, the functor $\mathsf{F}$ is defined as follows: 
$\mathsf{F}(\raisebox{-2pt}{\includegraphics[height=0.13in]{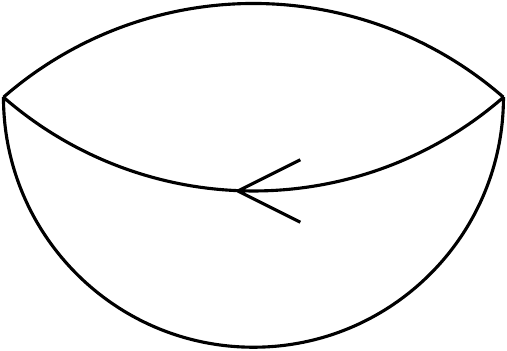}}) = \iota,\,  \mathsf{F}(\raisebox{-2pt}{\includegraphics[height=0.14in]{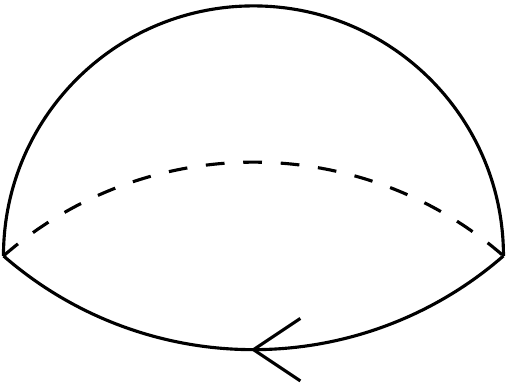}}) = \epsilon,\,\mathsf{F}(\raisebox{-3pt}{\includegraphics[height=0.18in]{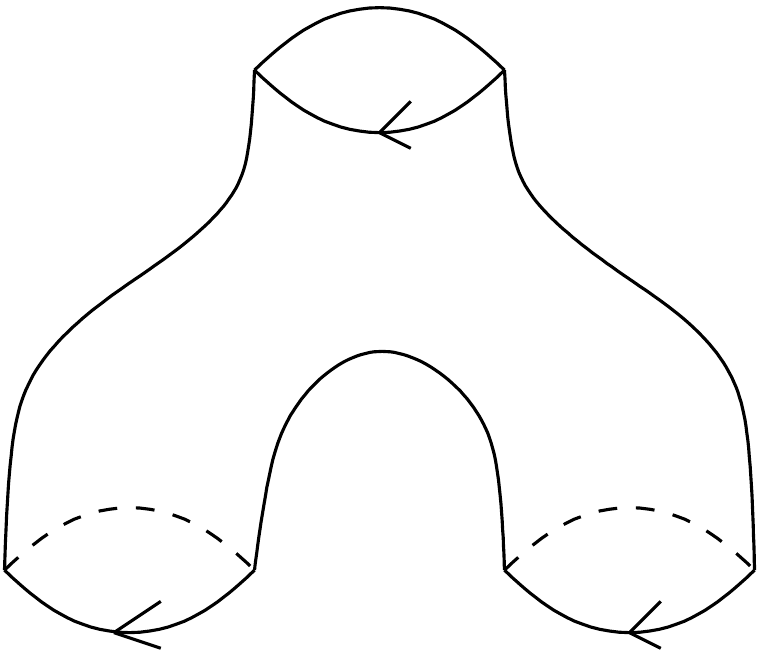}}) = m, \, \mathsf{F}(\raisebox{-3pt}{\includegraphics[height=0.18in]{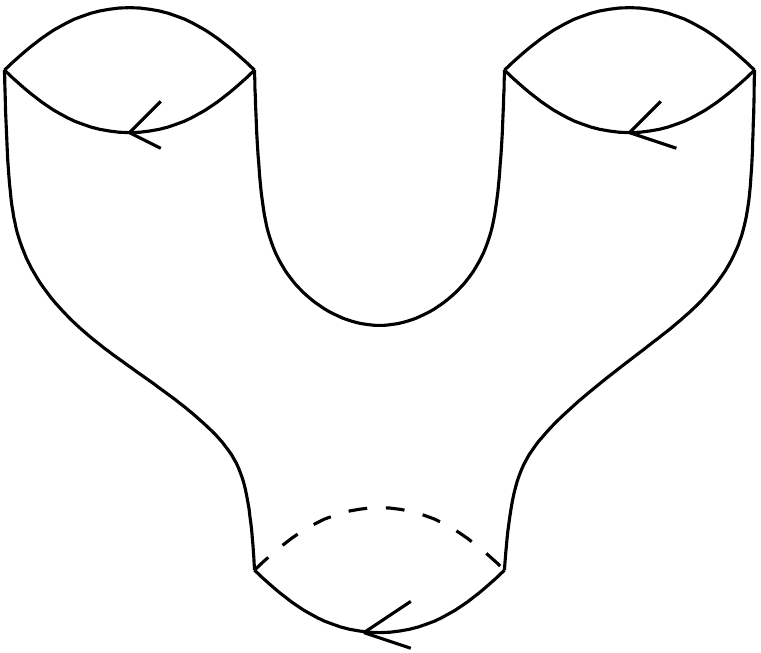}}) = \Delta.$

The annulus $S^1 \times [0,1]$ is the identity cobordism from a circle to itself, and $\mathsf{F}$ associates to it the identity map $\id: \mathcal{A} \rightarrow \mathcal{A}.$

A dot on a surface will denote multiplication by $X.$ The functor $\mathsf{F}$ asociates to the annulus $S^1 \times [0,1]$ with a dot the multiplication by $X$ endomorphism of $\mathcal{A},$ thus it associates the map $\mathcal{A} \to \mathcal{A}$ which takes $1$ to $X$ and $X$ to $X^2 = hX + a.$ Here we used that $X^2 - hX - a = 0$ in $\mathcal{A},$ which also gives the algebraic interpretation of a twice dotted surface, as explained below. The cobordism \raisebox{-2pt}{\includegraphics[height=0.13in]{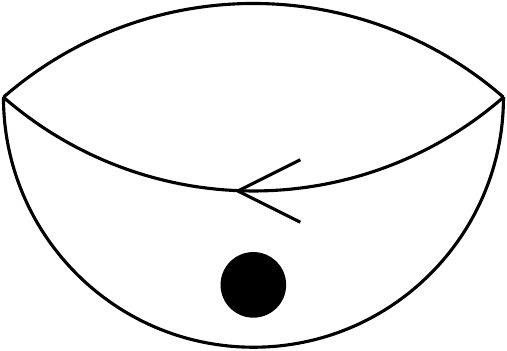}} can be regarded as the composition of the ``cup" cobordism \raisebox{-2pt}{\includegraphics[height=0.13in]{cuplo.pdf}} with the singly dotted annulus, and $\mathsf{F}(\raisebox{-2pt}{\includegraphics[height=0.13in]{cuplod.pdf}})$ produces the map $\mathbb{Z}[i][a,h] \to \mathcal{A}$ which takes $1$ to $X,$ obtained by composing $\iota$ with the multiplication by $X$ endomorphism of $\mathcal{A}.$ Moreover, since the cobordism $\raisebox{-2pt}{\includegraphics[height=0.13in]{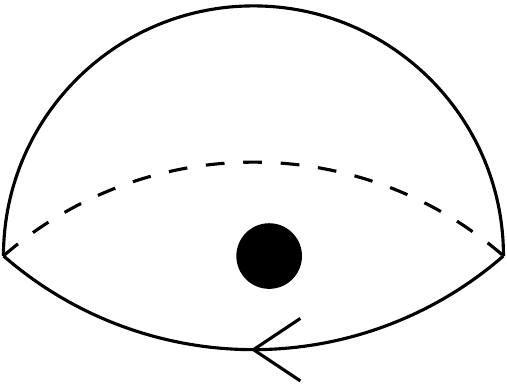}}$ can be regarded as the composition of the singly dotted annulus with the ``cap" cobordism $\raisebox{-2pt}{\includegraphics[height=0.13in]{caplo.pdf}},$ to define $\mathsf{F}(\raisebox{-2pt}{\includegraphics[height=0.13in]{caplod.pdf}})$ we compose the multiplication by $X$ endomorphism of $\mathcal{A}$ with $\epsilon.$ Therefore, $\mathsf{F}(\raisebox{-2pt}{\includegraphics[height=0.13in]{caplod.pdf}})$ stands for the map $\mathcal{A} \to \mathbb{Z}[i][a,h]$ which sends $1$ to $1$ and $X$ to $h.$ Two dots on a surface stand for  multiplication by $X^2 = hX + a.$ 

We want to extend the functor $\mathsf{F}$ to the subcategory of $\textit{Foams}(\emptyset)$ whose objects are disjoint unions of clockwise and counterclockwise oriented circles, and whose morphisms are foams between such 1-manifolds. For this purpose, we define the following maps associated to annuli with a singular circle:
\[  \raisebox{-22pt}{\includegraphics[height =0.7in]{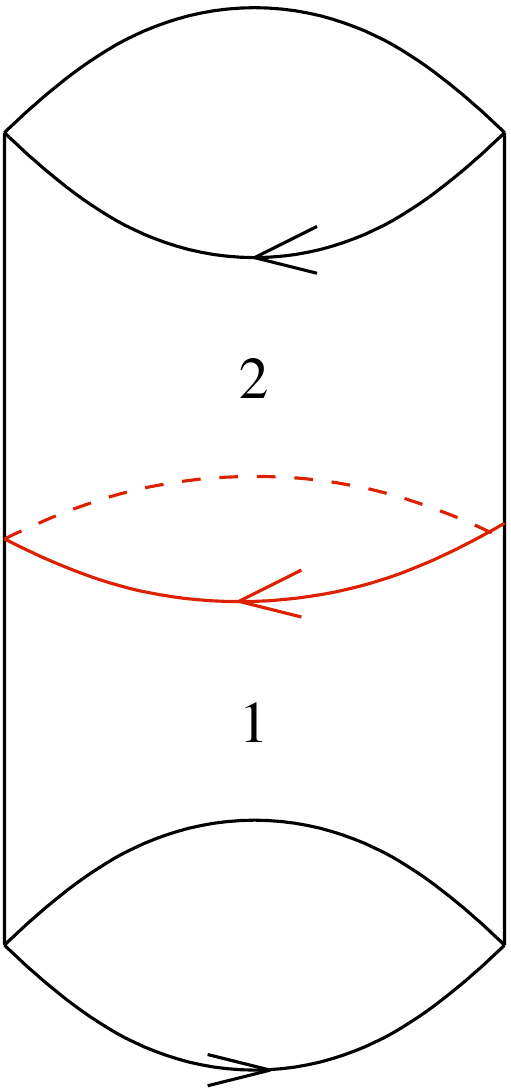}} \co \mathcal{A} \to \mathcal{A}, \begin{cases}1 \mapsto -i\\X \mapsto -i(h-X) \end{cases} \hspace{1cm}  \raisebox{-22pt}{\includegraphics[height =0.7in]{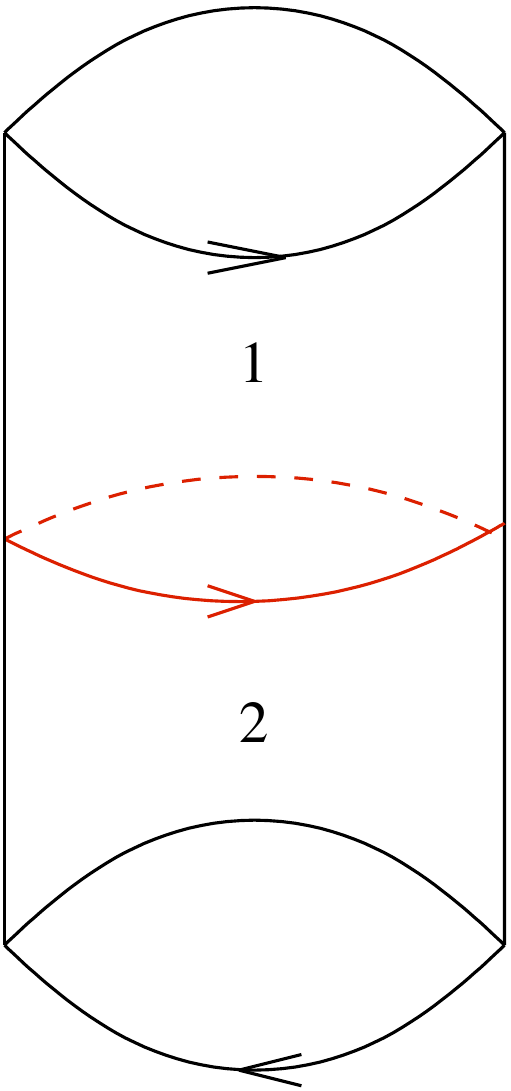}} \co \mathcal{A} \to \mathcal{A}, \begin{cases}1 \mapsto i\\X \mapsto i(h-X). \end{cases}\]
Consequently, we have the following maps for the particular foams depicted below:
  $$ \raisebox{-8pt}{\includegraphics[width=0.35in,height =0.35in]{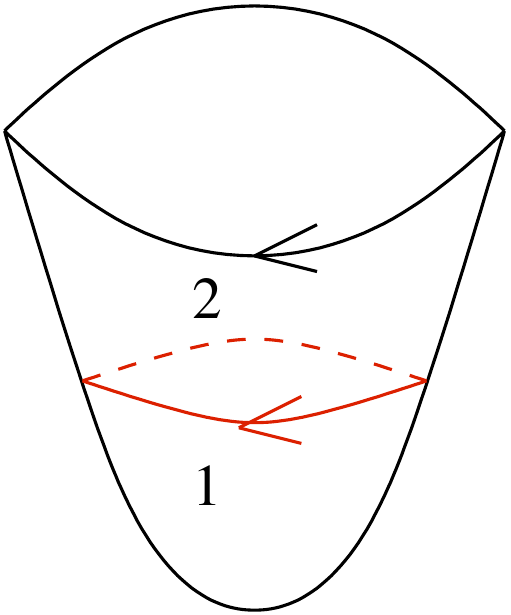}} \co \mathbb{Z}[i][a] \longrightarrow \mathcal{A},1\rightarrow -i \hspace{2cm} \raisebox{-8pt}{\includegraphics[width=0.35in,height =0.35in]{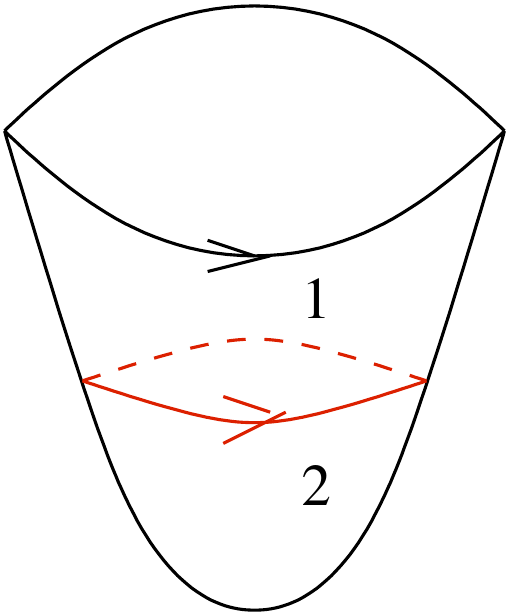}} \co \mathbb{Z}[i][a] \longrightarrow \mathcal{A}, 1\rightarrow i$$
  $$\raisebox{-8pt}{\includegraphics[width=0.35in, height =0.35in]{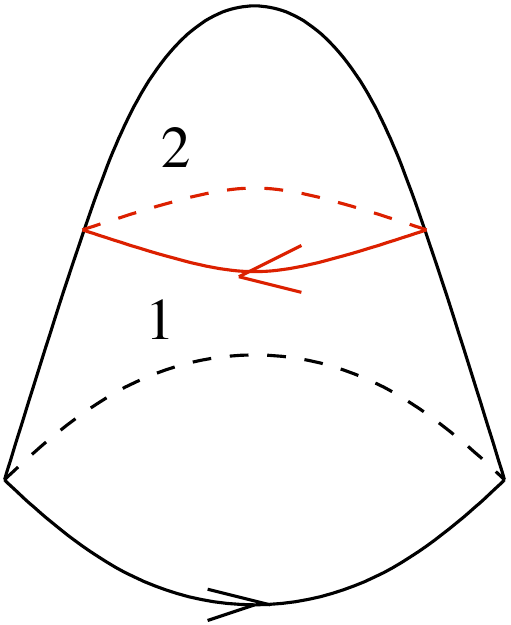}}:\mathcal{A}\longrightarrow \mathbb{Z}[i][a], \begin{cases} 1\rightarrow 0 \\X\rightarrow i \end{cases} \hspace{1cm}
\raisebox{-8pt}{\includegraphics[width=0.35in, height =0.35in]{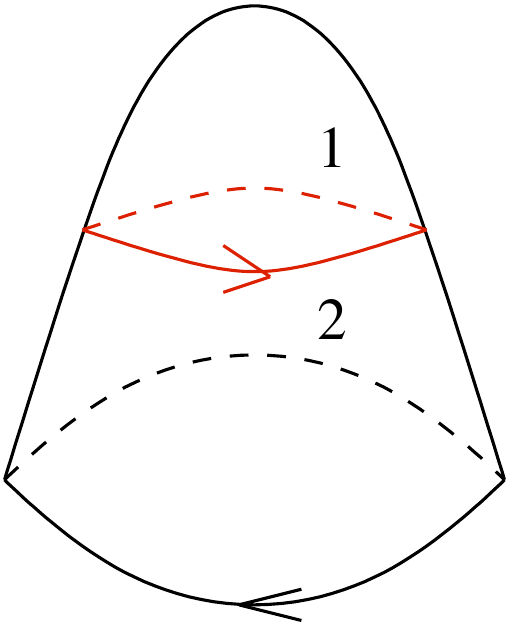}}:\mathcal{A}\longrightarrow \mathbb{Z}[i][a], \begin{cases} 1\rightarrow 0\\ X\rightarrow -i.\end{cases}$$
Therefore, $\mathsf{F}$ extends to a functor from the category of  dotted foams between oriented circles to the category of graded $\mathbb{Z}[i][a,h]$-modules, since any connected dotted foam whose boundary components are (clockwise and/or counterclockwise) oriented circles can be decomposed into annuli with exactly one singular circle, dotted annuli, and the generating morphisms of the category of oriented $(1+1)$-cobordisms.

Given a cobordism $S$ with $d$ dots, the homomorphism $\mathsf{F}(S)$ has degree given by the formula $\deg(S) =  -\chi(S) + 2d,$ where $\chi$ is the Euler characteristic of $S.$ Note that the functor $\mathsf{F}$ is degree-preserving.

\subsection{Local relations}\label{sec:relations l}

We mod out the morphisms of the category \textit{Foams} by the local relations $\ell$ = (2D, SF, S, UFO) below, and denote by $\textit{Foams}_{/\ell}$ the quotient of the category $\textit{Foams}$ by these relations.
$$\xymatrix@R=2mm
{
\raisebox{-5pt}{
\includegraphics[height=0.2in]{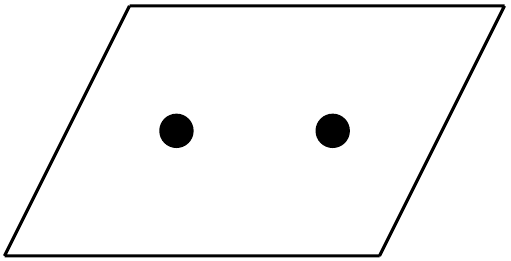}}= h\raisebox{-5pt}{\includegraphics[height=0.2in]{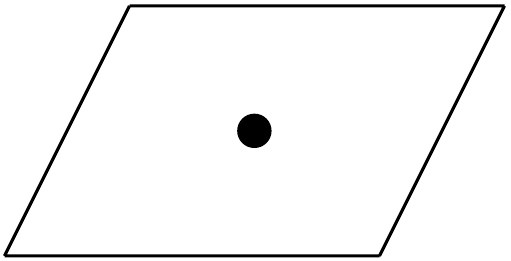}} +
a\raisebox{-5pt}{
\includegraphics[height=0.2in]{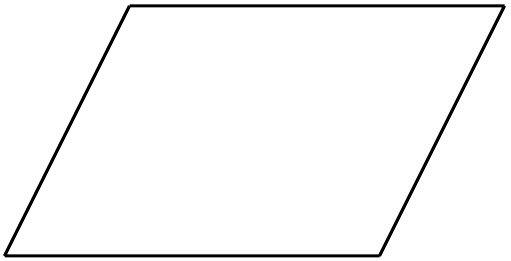}}
&\text{(2D)}\\
\raisebox{-18pt}{
\includegraphics[height=0.6in]{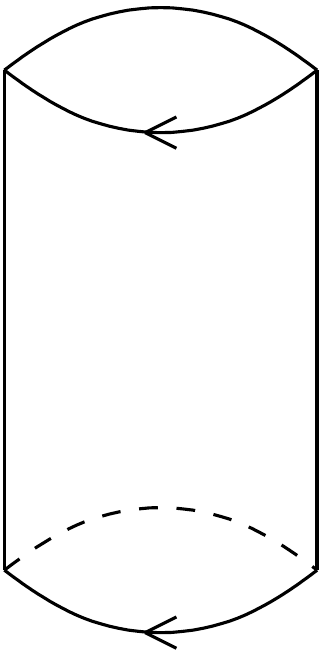}}=
\raisebox{-18pt}{
\includegraphics[height=0.6in]{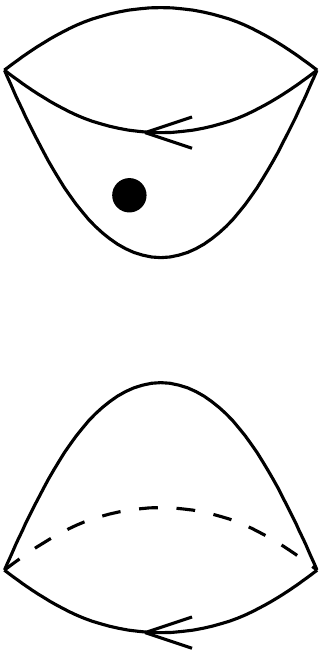}}+
\raisebox{-18pt}{
\includegraphics[height=0.6in]{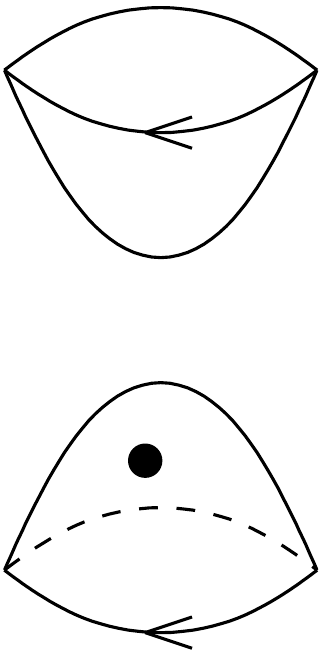}}-h
\raisebox{-18pt}{
\includegraphics[height=0.6in]{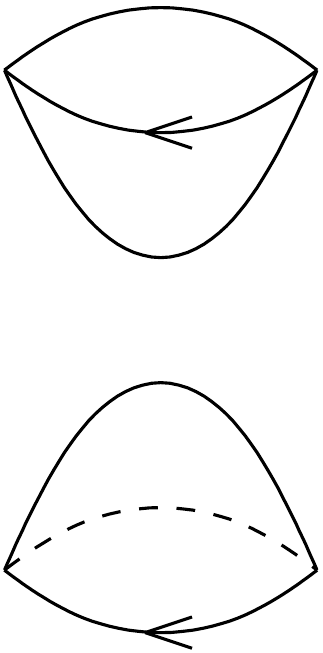}}
&\text{(SF)} \\
\raisebox{-10pt}{\includegraphics[width=0.4in]{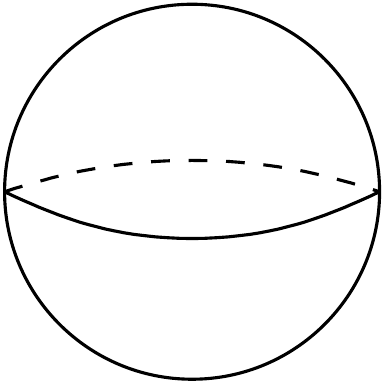}}=0,\quad
\raisebox{-10pt}{\includegraphics[width=0.4in]{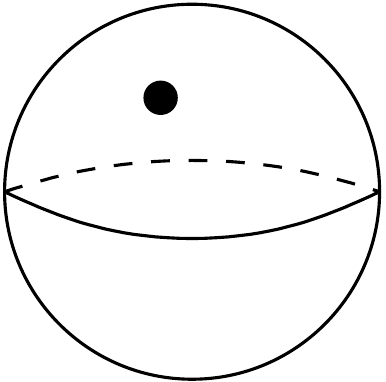}}=1
&\text{(S)} \\
\raisebox{-10pt}{\includegraphics[width=0.5in]{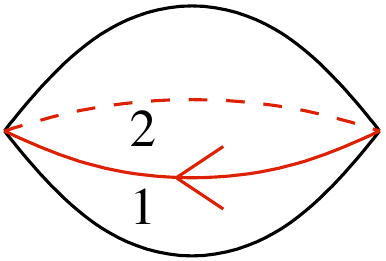}}=0=
\raisebox{-10pt}{\includegraphics[width=0.5in]{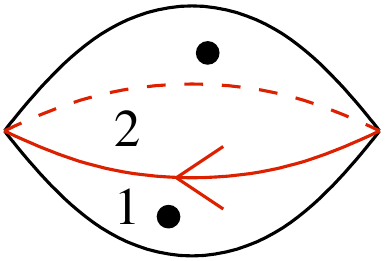}},\quad
\raisebox{-10pt}{\includegraphics[width=0.5in]{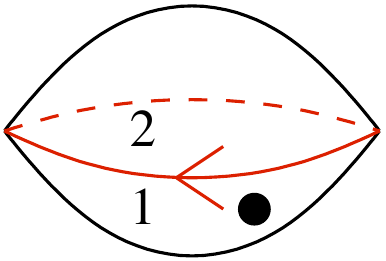}}=i= 
-  \raisebox{-10pt}{\includegraphics[width=0.5in]{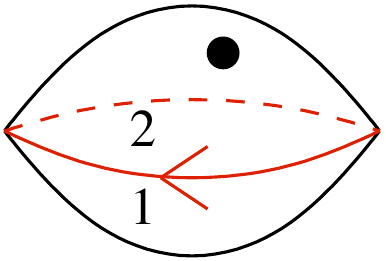}}
&\text{(UFO)} 
}$$
When there are two or more dots on a facet of a foam we can use the (2D) relation to reduce it to the case when there is at most one dot. 
From the surgery formula (SF) we obtain the following \textit{genus reduction} formula 
$$\quad \raisebox{-3pt}{\includegraphics[height=0.3in, width =0.3in]{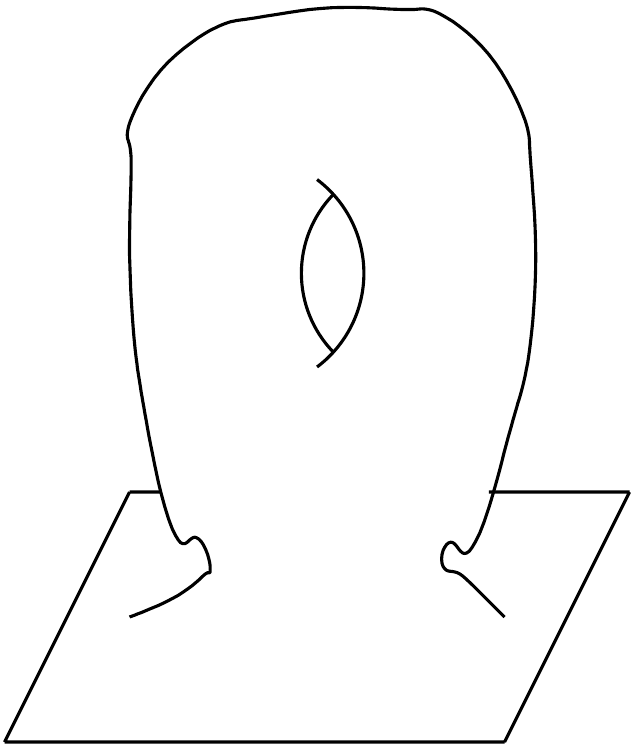}} = 2
\raisebox{-5pt}{\includegraphics[width=0.35in]{plane1d.pdf}} - h\raisebox{-5pt}{
\includegraphics[width=0.35in]{plane.pdf}},$$
 which in particular yields
$$\raisebox{-5pt}{\includegraphics[width=0.4in]{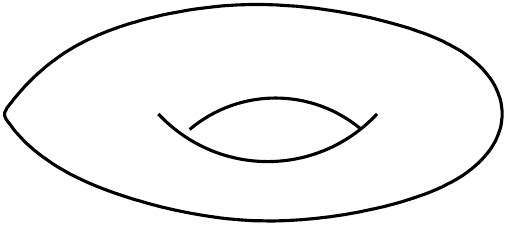}}=2,\,\, \,
\raisebox{-5pt}{\includegraphics[width=0.4in]{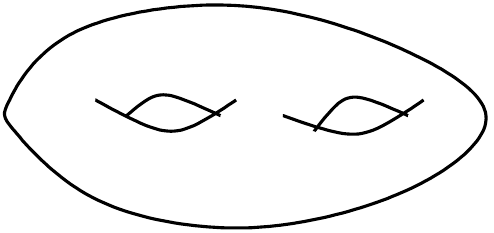}}=0,\,\,\, \raisebox{-5pt}{\includegraphics[width=0.4in]{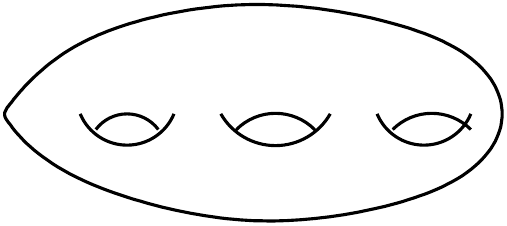}}=2h^2 + 8a.$$

A closed foam $S$ can be viewed as a morphism from the empty web to itself. By the relations $\ell$, we assign to $S$ an element of $\mathbb{Z}[i][a,h],$ called the \textit{evaluation} of $S$ and denoted by 
$\mathcal{F}(S).$ We view $\mathcal{F}$ as a functor from the category $\textit{Foams}_{/\ell}(\emptyset)$ to the category of $\mathbb{Z}[i]h[a]$-modules. 

\begin{remark} The following statements hold, and can be verified similarly as their analogous in~\cite[Section 4]{CC}. 
\begin{enumerate}
\item The functor $\mathsf{F}$ introduced in Section~\ref{sec:TQFT} satisfies the local relations $\ell$. In particular, $\mathsf{F}$ descends to a functor from $\textit{Foams}_{/\ell}(\emptyset)$ to $\mathbb{Z}[i][a, h]$-Mod.
\item The set of local relations $\ell$ are consistent and determine uniquely the \linebreak evaluation of every closed foam. \end{enumerate}
\end{remark}

The evaluation of closed foams is multiplicative with respect to disjoint unions of closed foams:
$\mathcal{F}(S_1 \cup S_2) = \mathcal{F}(S_1) \mathcal{F}(S_2)$. If $S'$ and $S$ are closed foams such that $S'$ is obtained from $S$ by reversing the ordering of the facets at a singular circle, then
$\mathcal{F}(S')  = - \mathcal{F}(S)$. Moreover, relations $\ell$ imply the identities (ED) depicted in Figure~\ref{fig:exchanging dots}, which establish the way we can exchange dots between two neighboring facets. \begin{figure}[ht!]
$$\xymatrix@R=2mm{
\raisebox{-22pt}{\includegraphics[height=.6in]{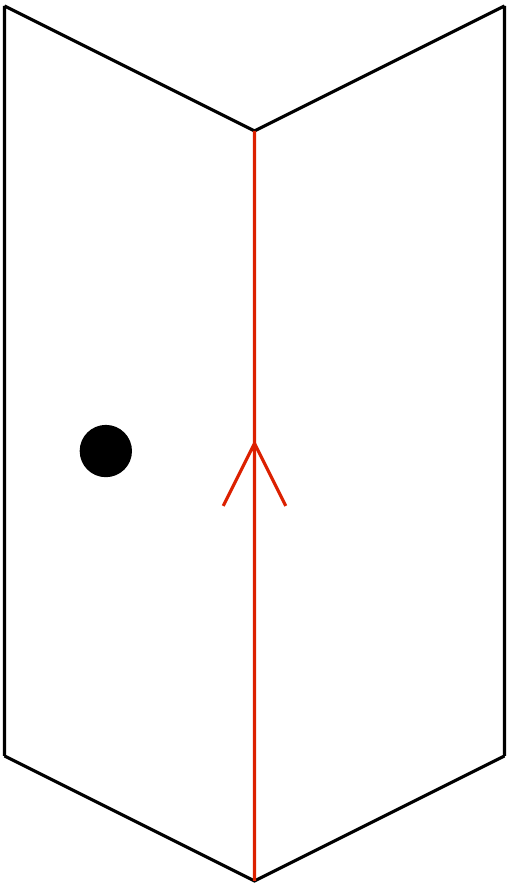}} +
\raisebox{-22pt}{\includegraphics[height=.6in]{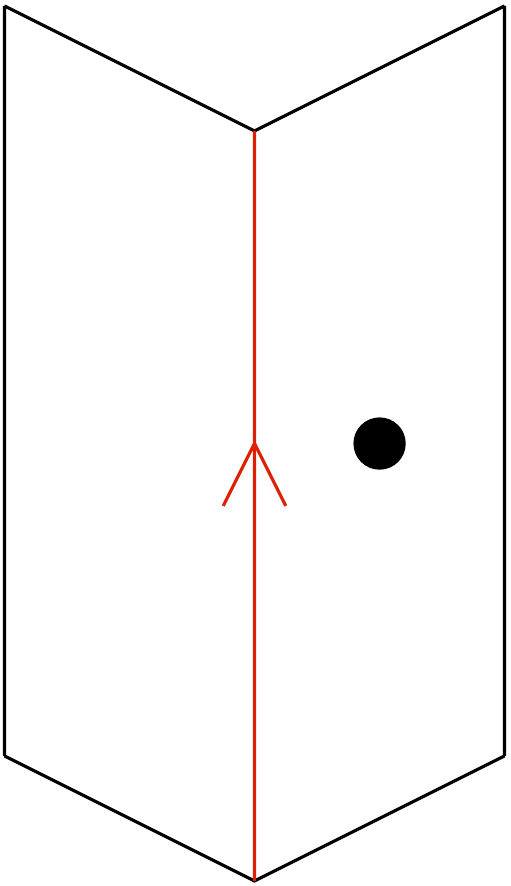}} = h\,\,\raisebox{-22pt}{\includegraphics[height=.6in]{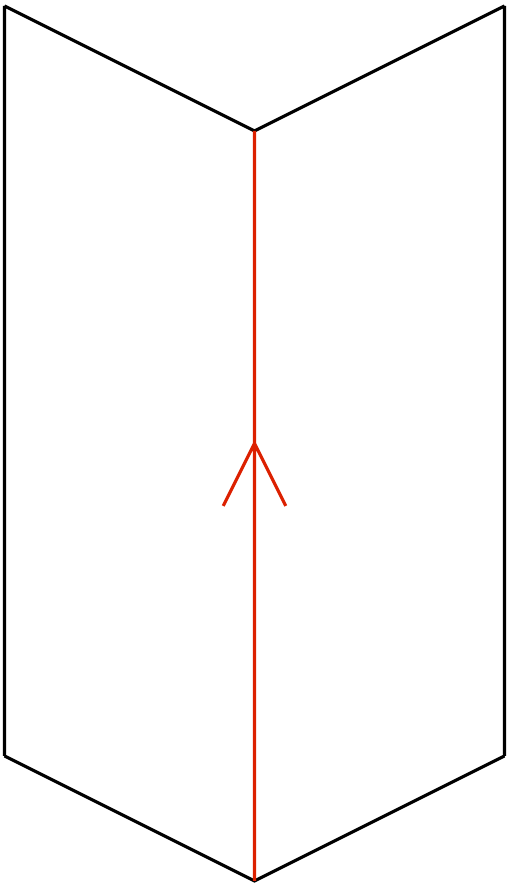}} \quad , \quad
\raisebox{-22pt}{\includegraphics[height=.6in]{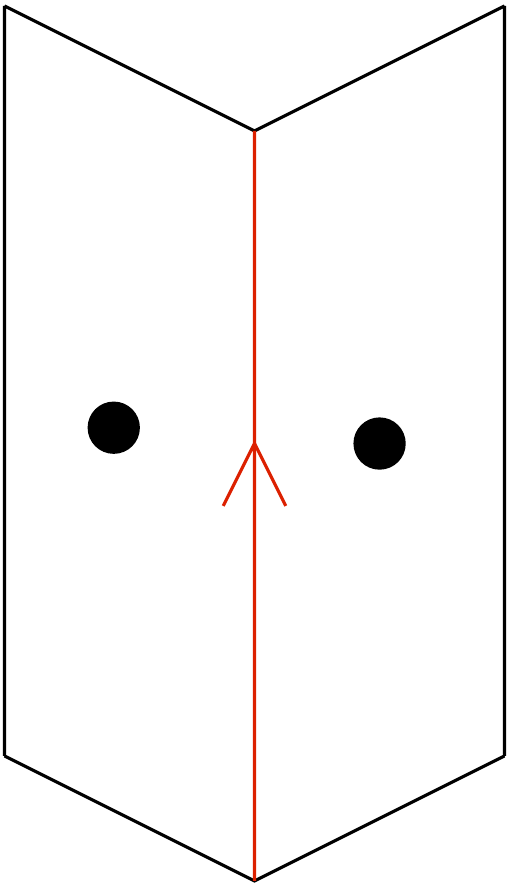}} = -a\, \,
 \raisebox{-22pt}{\includegraphics[height=.6in]{exch.pdf}} & \text{(ED)}
}$$
\caption{Exchanging dots between facets}
\label{fig:exchanging dots}
\end{figure}

\begin{definition}\label{def:quotient category}
For webs $\Gamma, \Gamma',$ foams $S_i \in \Hom_{Foams_{/\ell}}(\Gamma,\Gamma')$ and $c_i \in \mathbb{Z}[i][a,h]$ we say that $\sum_i c_iS_i = 0$ if and only if $\sum_i c_i \mathcal{F}( V'S_i V) = 0$ holds, for any foam $V \in \Hom_{Foams_{/\ell}}(\emptyset, \Gamma)$ and $V' \in \Hom_{Foams_{/\ell}}(\Gamma', \emptyset).$ 
\end{definition}

\begin{definition}
Let $S$ be a foam with $d$ dots in $\textit{Foams}(B).$ The degree formula defined in Section~\ref{sec:TQFT} extends to foams in the natural way, and we define the \textit{grading} of $S$ by $\deg(S) = -\chi(S) + \frac{1}{2}\vert B\vert +2d,$ where $\chi$ is the Euler characteristic and $\vert B\vert$ is the cardinality of $B.$\end{definition}

We give below the degrees of some of the basic foams we work with.
\begin{align*}
&\text{deg} \left(\raisebox{-5pt}{\includegraphics[height=0.28in]{cuplo.pdf}}\right) = 
\text{deg} \left(\raisebox{-5pt}{\includegraphics[height=0.3in]{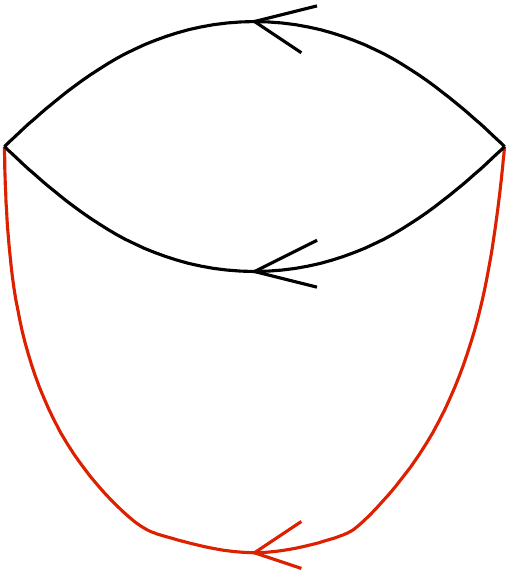}} \right) = 
\text{deg} \left(\raisebox{-5pt}{\includegraphics[height=0.28in]{caplo.pdf}} \right) = 
\text{deg} \left(\raisebox{-5pt}{\includegraphics[height=0.3in]{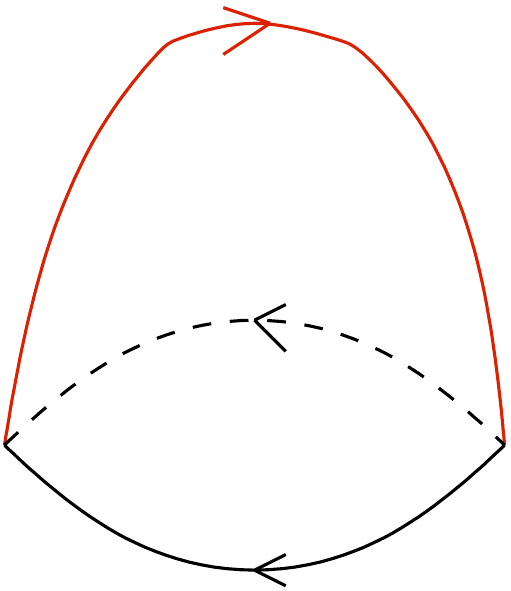}} \right) = -1,\\
&\text{deg} \left(\raisebox{-5pt}{\includegraphics[height=0.28in]{cuplod.pdf}} \right) = 
\text{deg} \left(\raisebox{-5pt}{\includegraphics[height=0.3in]{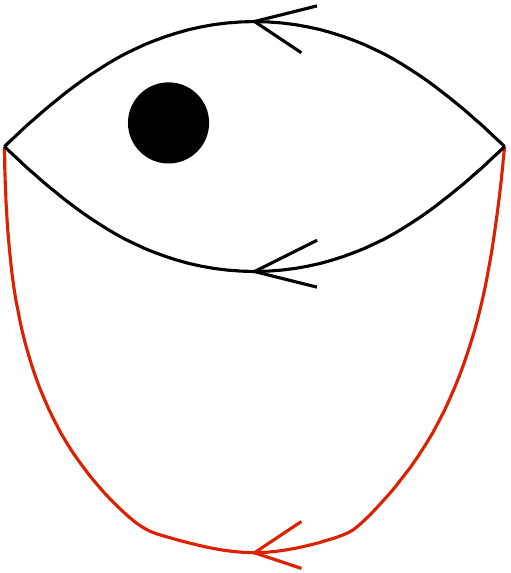}} \right) = 
\text{deg} \left(\raisebox{-5pt}{\includegraphics[height=0.28in]{caplod.pdf}} \right) = 
\text{deg} \left(\raisebox{-5pt}{\includegraphics[height=0.3in]{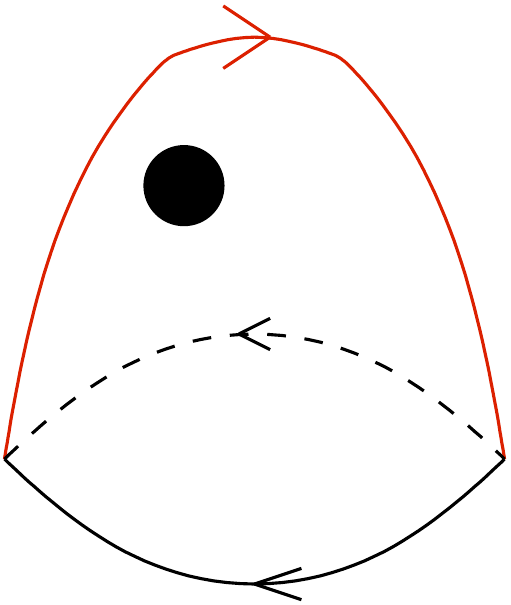}} \right) = 1,\\
&\text{deg} \left(\raisebox{-8pt}{\includegraphics[height=0.35in]{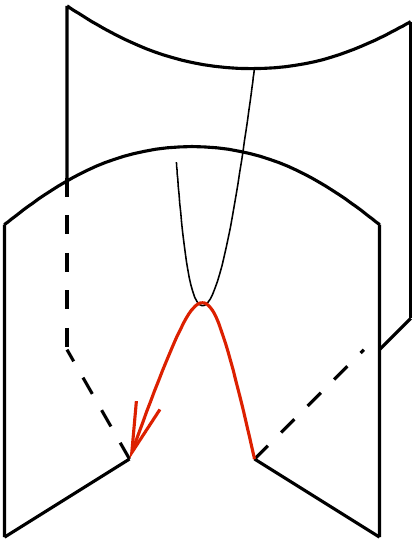}}\right) = \text{deg} \left(\raisebox{-8pt}{\includegraphics[height=0.35in]{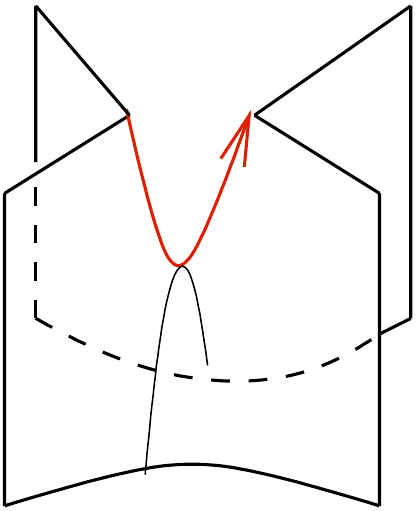}}\right) = 1.
\end{align*}
Note that for any composable foams $S_1, S_2$ we have: $\deg(S_1S_2) = \deg(S_1) + \deg(S_2),$ and that the local relations $\ell$ are degree preserving. Therefore, both categories $\textit{Foams}$ and $\textit{Foam}_{/\ell}$ are graded. 

The next two lemmas and corollaries are proved exactly the same as in~\cite{CC}.
\begin{lemma}\label{handy relations}
The following ``sheet relations'' (SR) hold in $\textit{Foams}_{/\ell}$:
$$\xymatrix@R=2mm{
\raisebox{-13pt}{\includegraphics[height =0.5in]{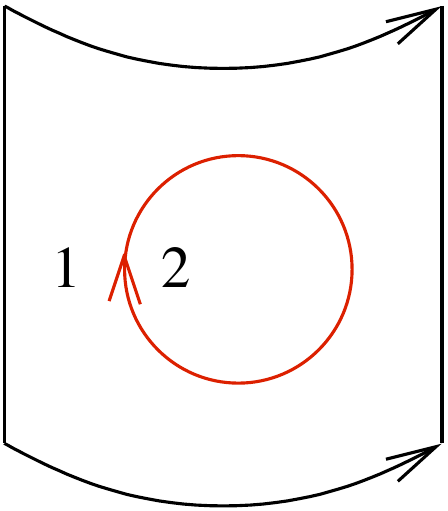}}= i
\,\raisebox{-13pt}{\includegraphics[height=0.5in]{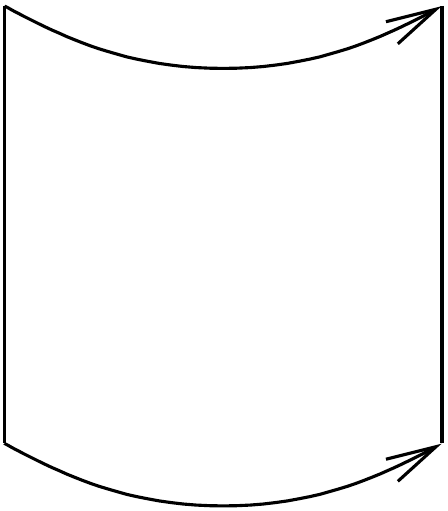}} &
\raisebox{-13pt}{\includegraphics[height =0.5in]{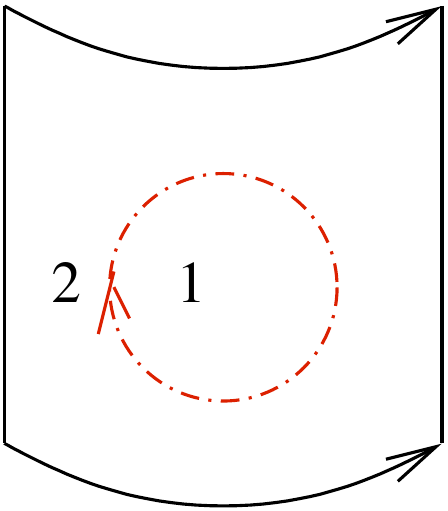}}= -i
\,\raisebox{-13pt}{\includegraphics[height=0.5in]{handy-rel3.pdf}} \\
\raisebox{-13pt}{\includegraphics[height =0.5in]{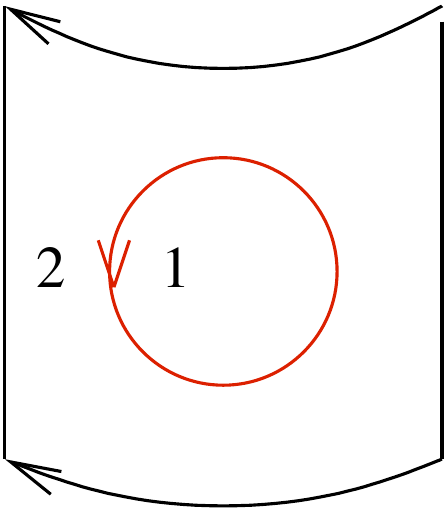}}=-i
\,\raisebox{-13pt}{\includegraphics[ height=0.5in]{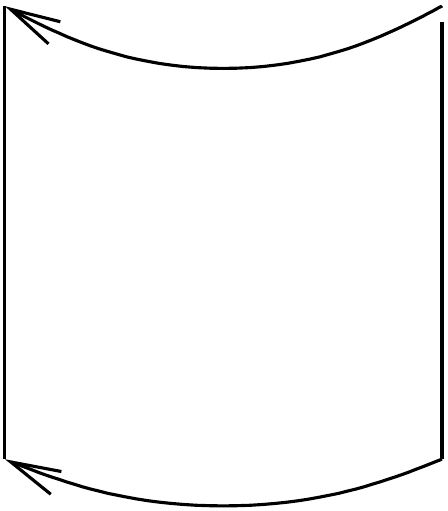}} &
\raisebox{-13pt}{\includegraphics[height =0.5in]{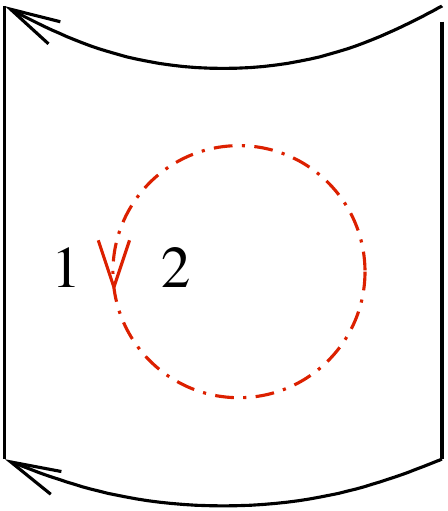}}=i
\,\raisebox{-13pt}{\includegraphics[ height=0.5in]{handy-rel6.pdf}} & \text{(SR)} }$$
\end{lemma}

\begin{lemma}\label{nice relations}
The following relations hold in $\textit{Foams}_{/\ell}$:
$$\xymatrix@R=2mm{
 \raisebox{-22pt}{\includegraphics[height=0.7in]{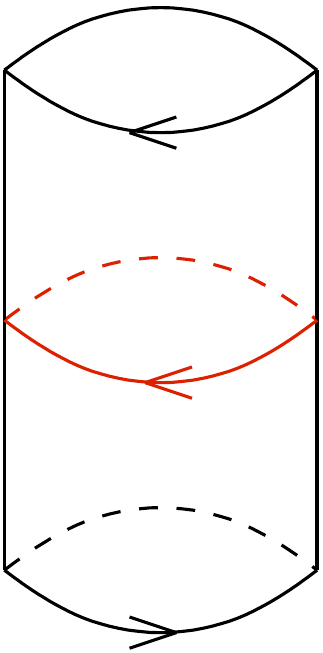}}= i
\raisebox{-22pt}{\includegraphics[height=0.7in]{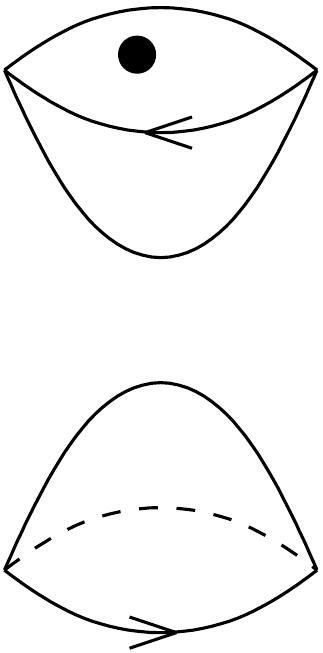}}-i
\raisebox{-22pt}{\includegraphics[height=0.7in]{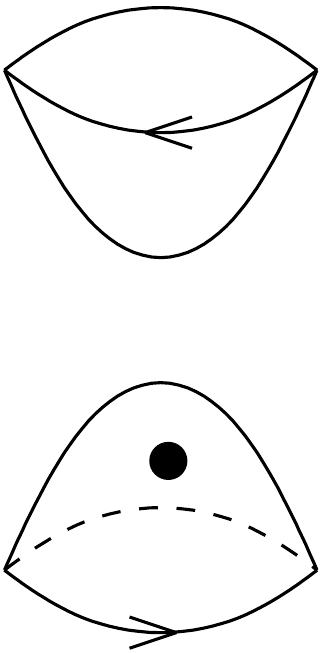}} 
& \text{(RSC)}\\
\raisebox{-22pt}{\includegraphics[height=0.7in]{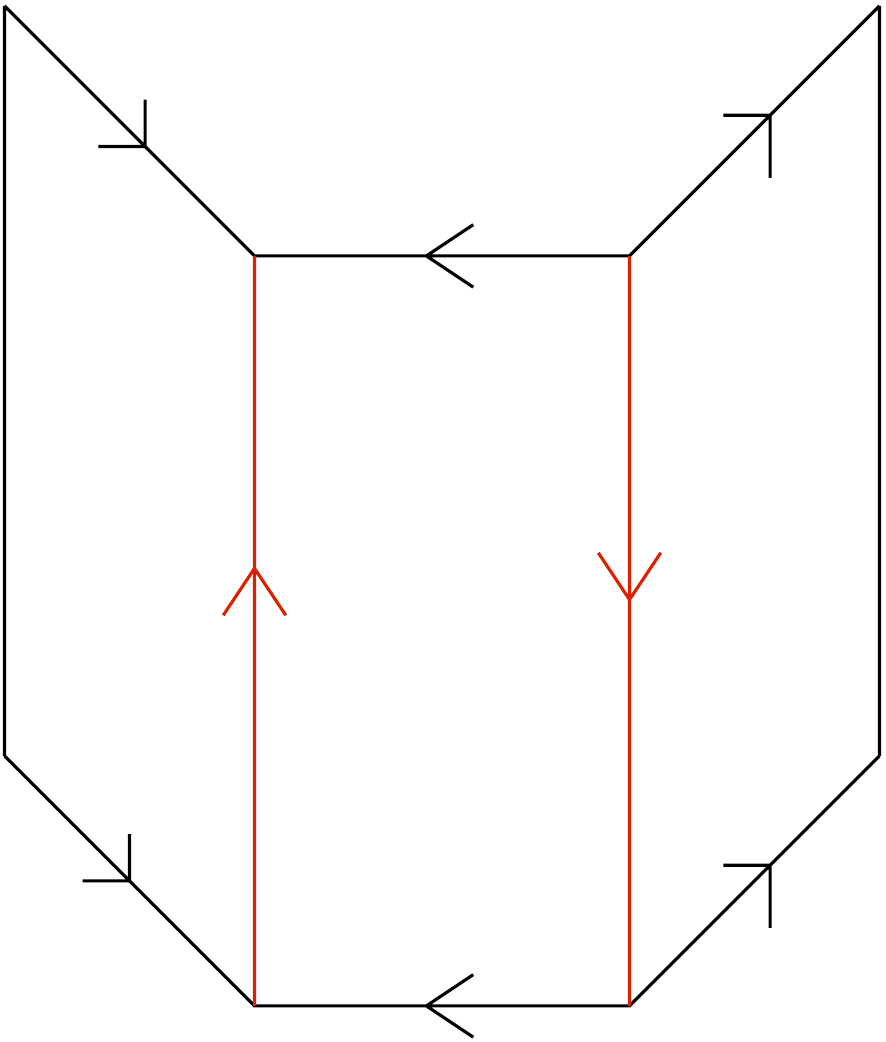}}= -i \,
\raisebox{-22pt}{\includegraphics[height=0.7in]{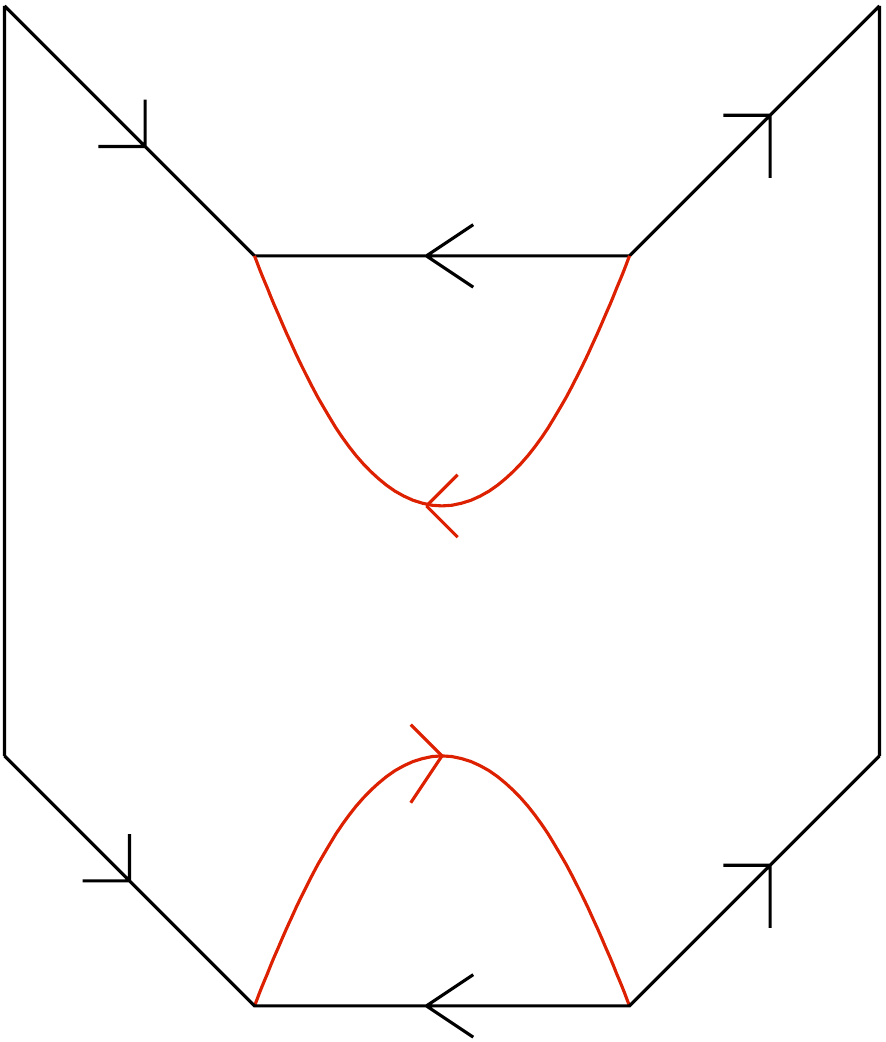}} \quad \text{and}\quad
\raisebox{-22pt}{\includegraphics[height=0.7in]{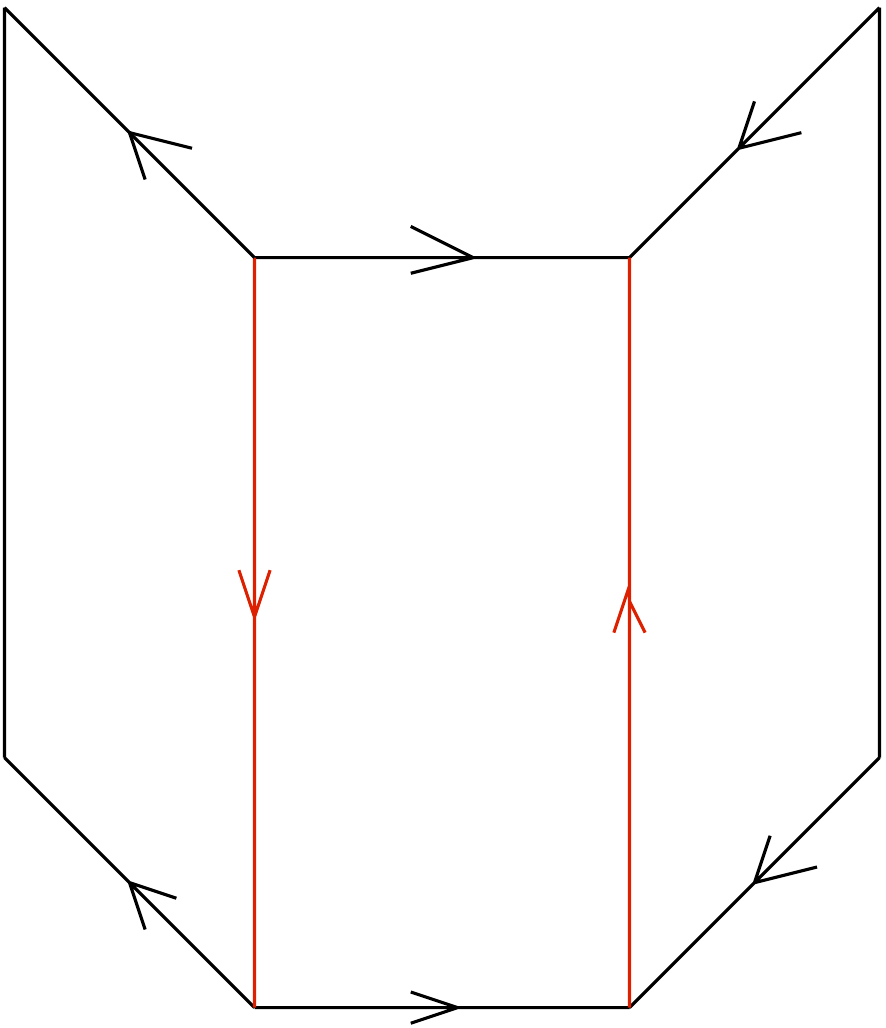}}= i\,
\raisebox{-22pt}{\includegraphics[height=0.7in]{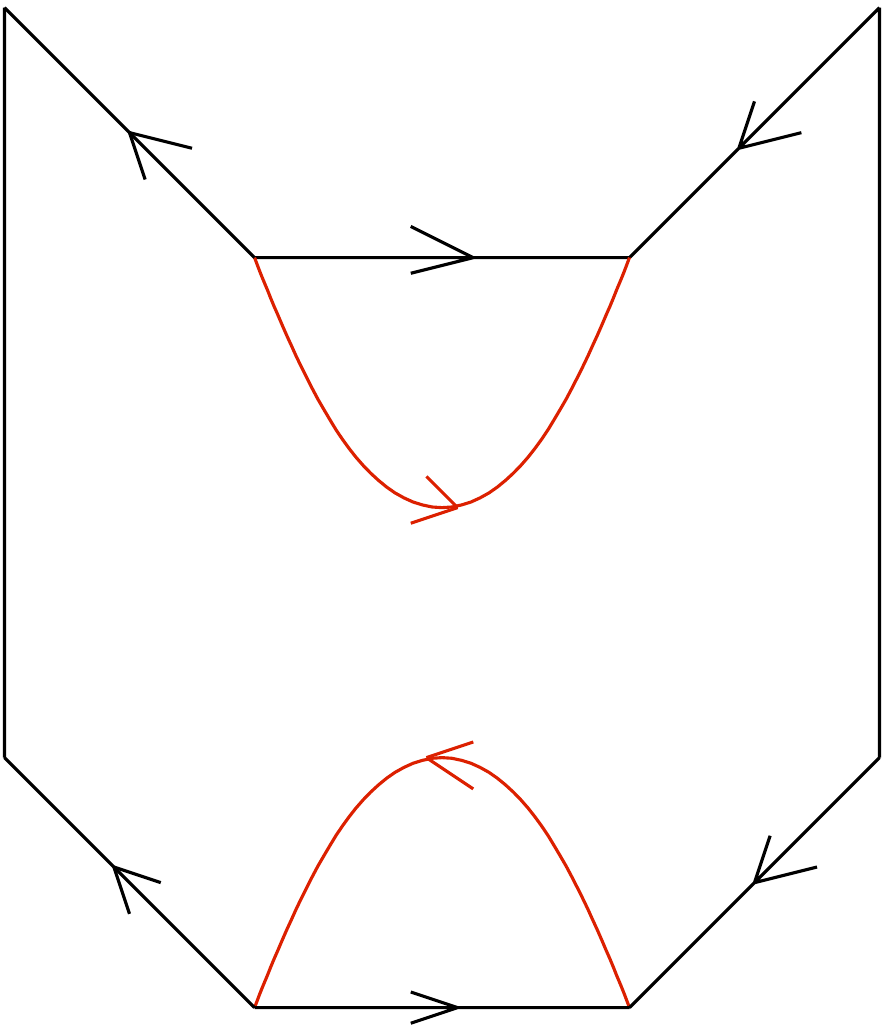}}
&\text{(CI)} \\
\raisebox{-22pt}{\includegraphics[height=0.7in]{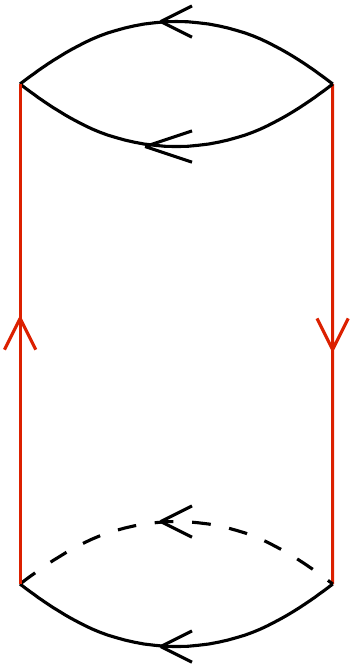}}= -i
\raisebox{-22pt}{\includegraphics[height=0.7in]{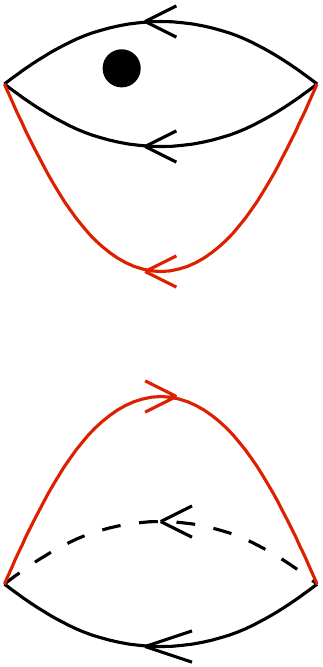}} -i
\raisebox{-22pt}{\includegraphics[height=0.7in]{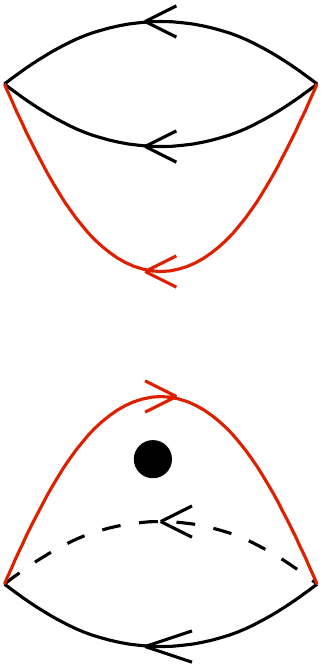}} + hi
\raisebox{-22pt}{\includegraphics[height=0.7in]{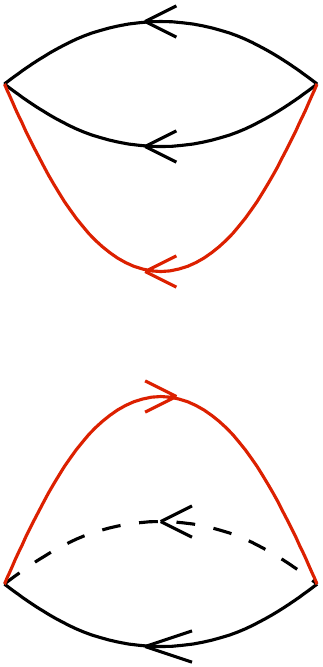}}
& \text{(CN)}
}$$
where the dots in (CN) are on the preferred facets (those in the back).
\end{lemma}

 \begin{corollary}\label{removing singular points in pairs}
 The isomorphisms given in Figure~\ref{fig:removing vertices in pairs1} hold in the category $\textit{Foam}_{/\ell}$.
 \begin{figure}[ht]
$$\xymatrix@R=2mm{
\raisebox{-25pt}{\includegraphics[height=.8in]{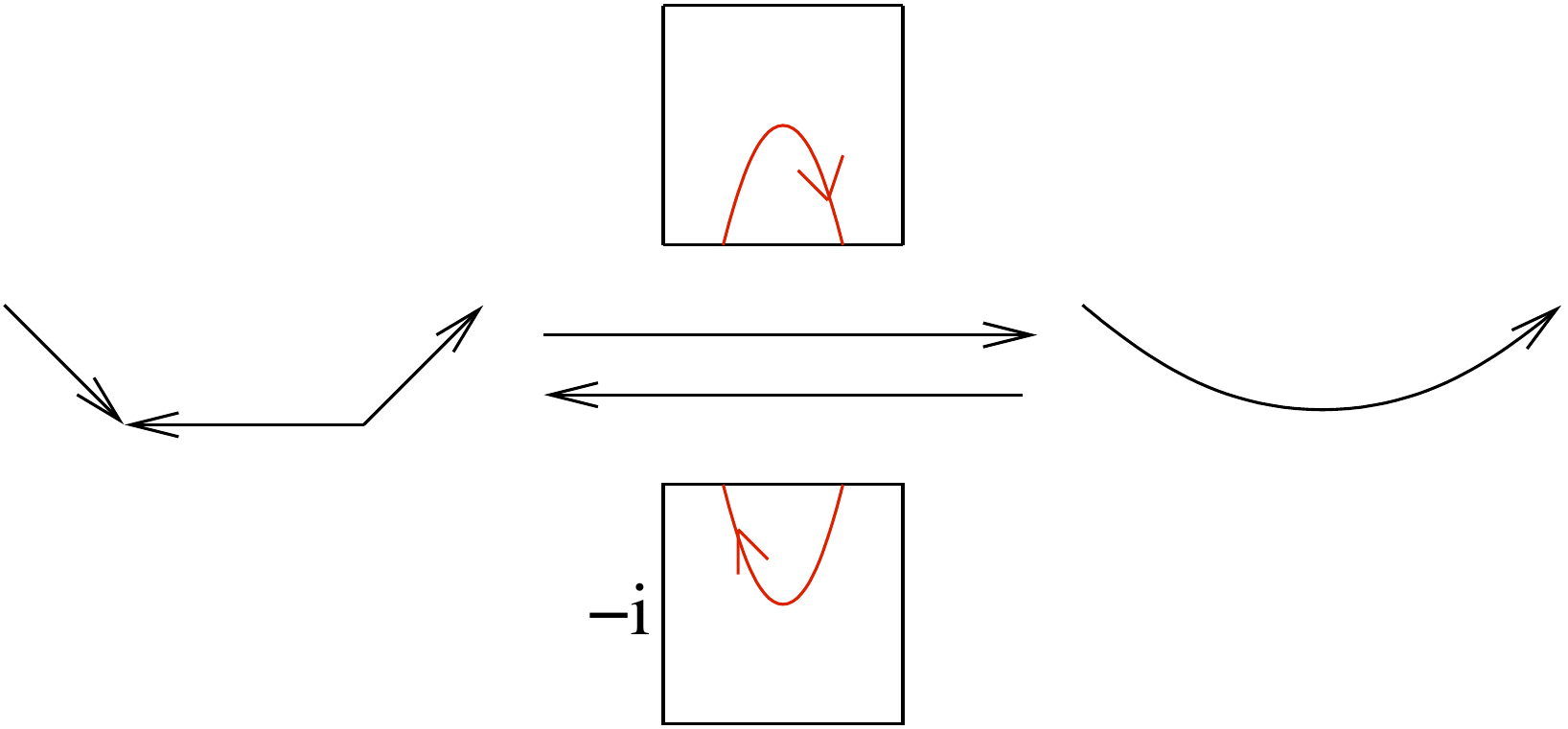}}
 \quad \text{and} \quad
\raisebox{-25pt}{\includegraphics[height=.8in]{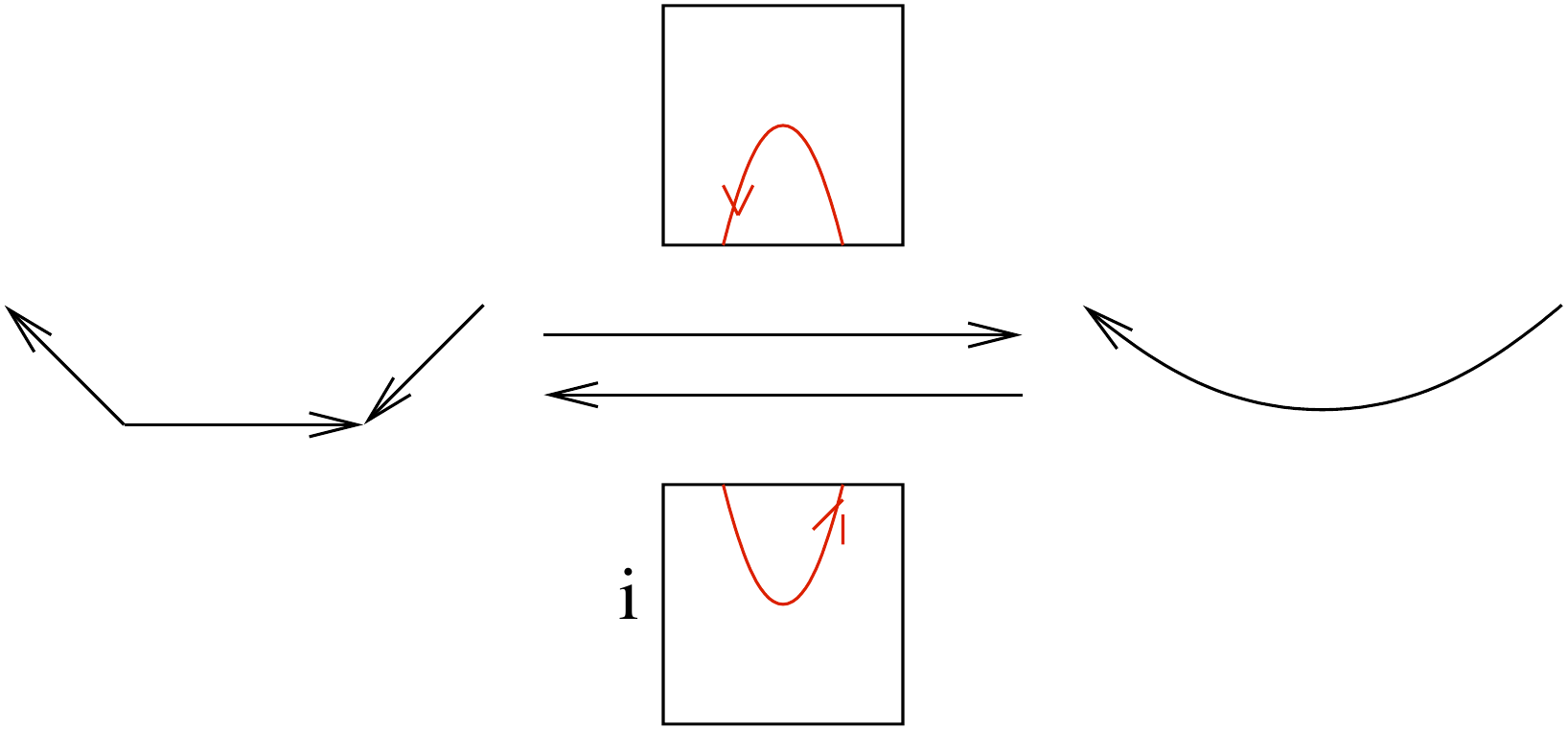}}
}$$
\caption{Removing/creating pairs of vertices of the same type}
\label{fig:removing vertices in pairs1}
\end{figure}  
\end{corollary}
\begin{corollary}\label{cor:Isomorphisms 1 and 2}
The following isomorphisms hold in the category $\textit{Foam}_{/\ell}.$
$$\raisebox{-25pt}{\includegraphics[height=.7in]{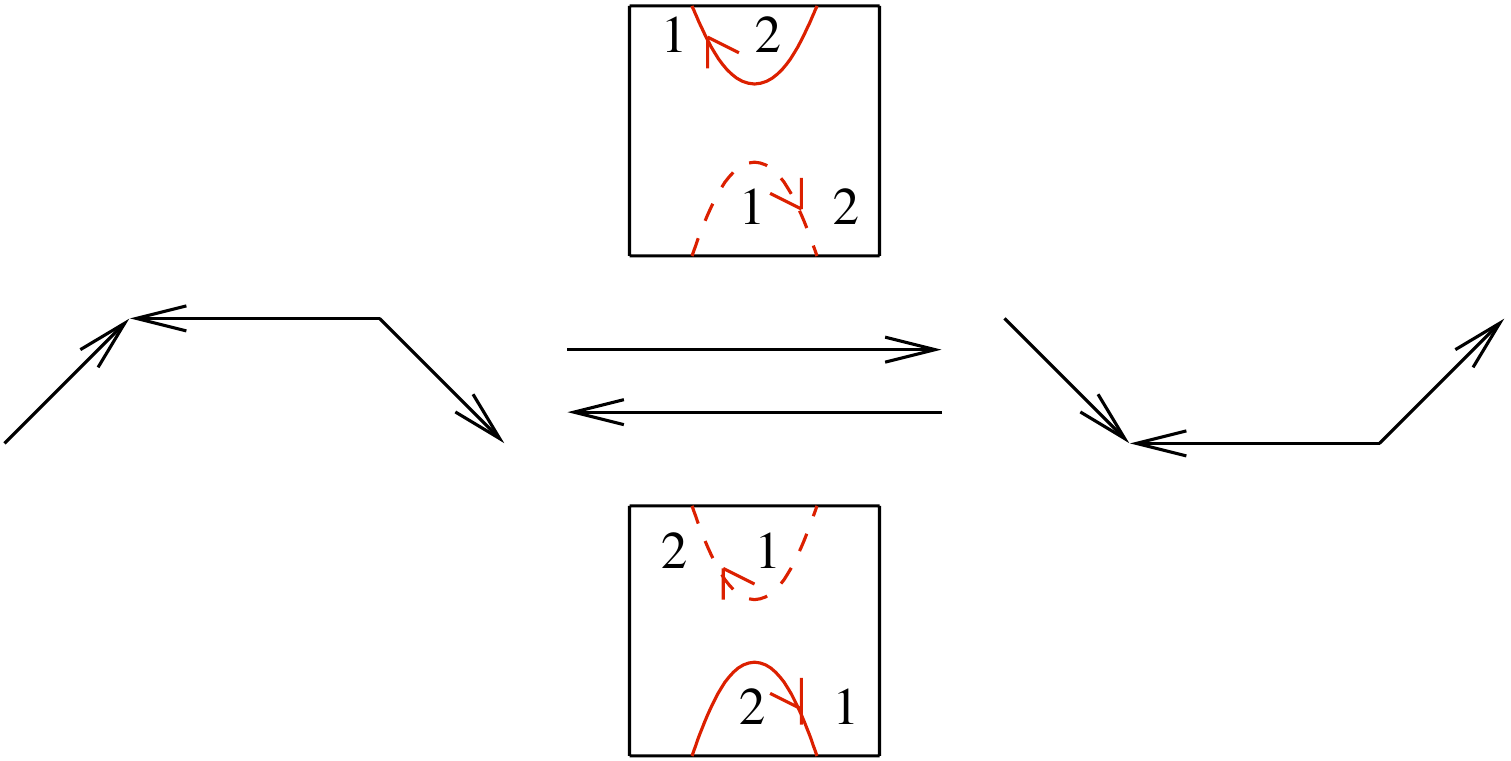}} \quad \mbox{and} \quad
\raisebox{-25pt}{\includegraphics[height=.7in]{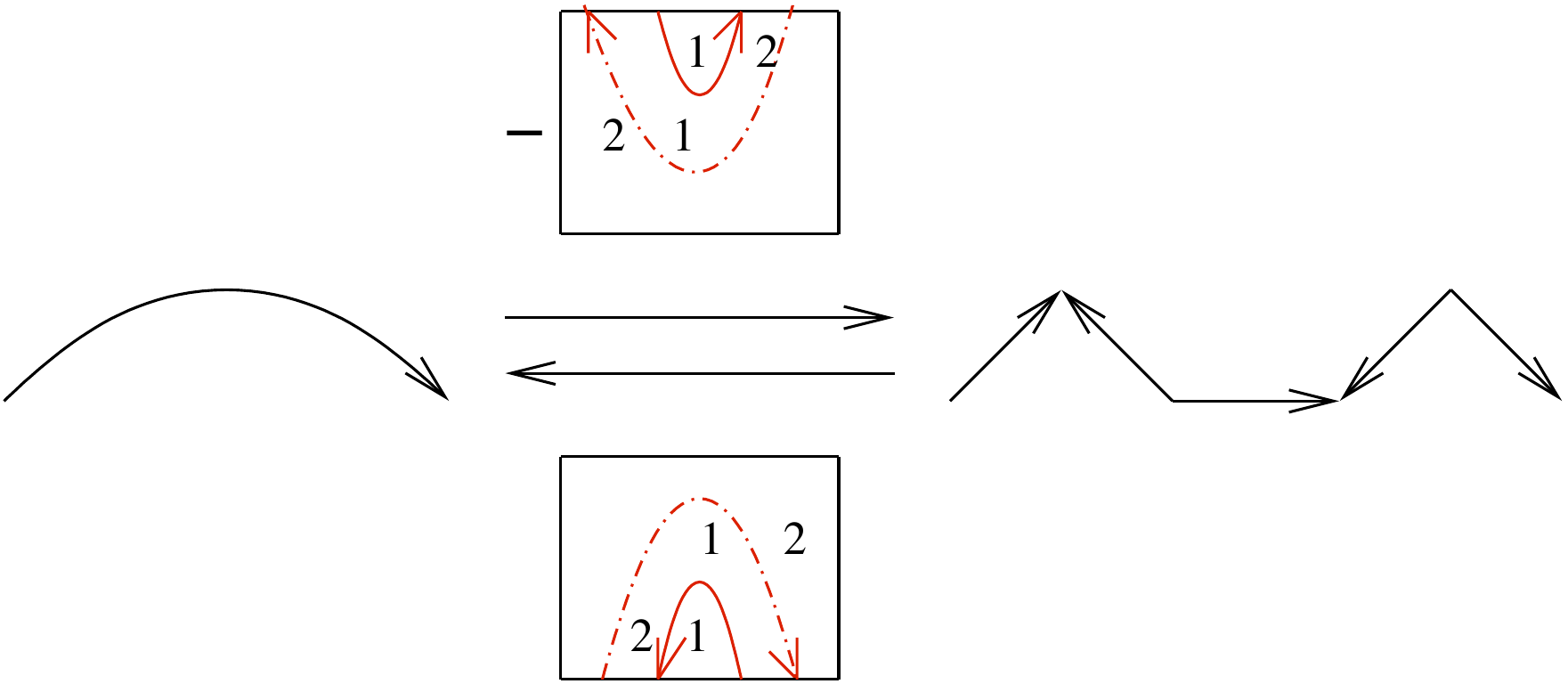}} $$
\end{corollary}

\section{\textbf{The geometric invariant of a tangle}}\label{sec:complex}

To categorify the quantum $\mf{sl}(2)$-link invariant, we replace the relations in Figure~\ref{fig:decomposition of crossings} by formal chains as in Figure~\ref{fig:mapcones} (where $\{r\}$ stands for the grading shift operator).

\begin{figure}[ht!]\begin{center}
\includegraphics[height=2in]{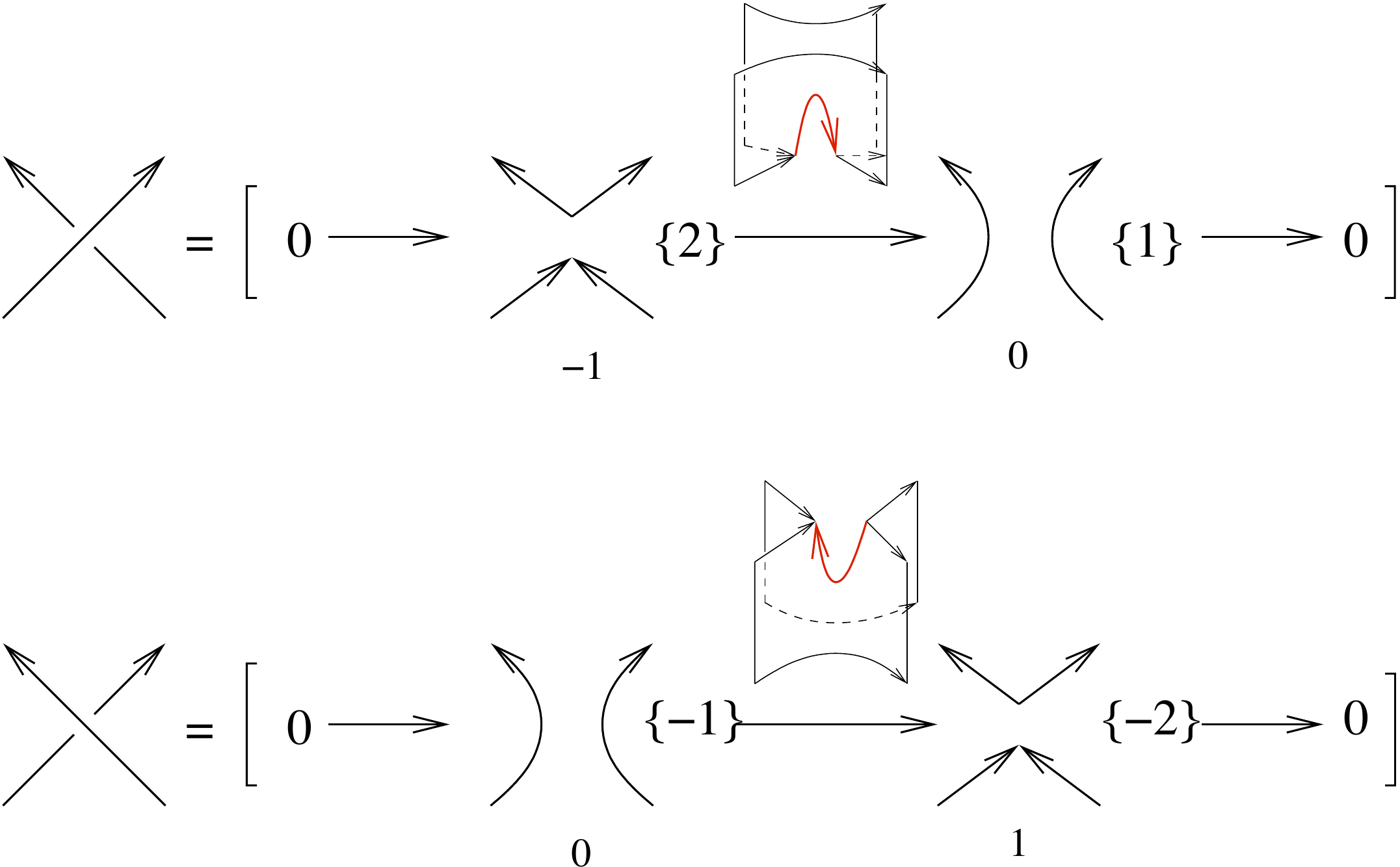} 
\caption{The oriented resolutions are at cohomological degrees 0} \label{fig:mapcones}\end{center}
\end{figure}

Given a tangle diagram $T,$ we associate to it a formal chain complex $[T]$ which is obtained by taking the (formal) tensor product of the chains associated to all crossings in $T.$ The chain objects are formal direct sums of webs---resolutions of $T$---and differentials are matrices of foams. We assume that the reader is familiar with such construction (see~\cite{BN1} and~\cite{CC} for more details).

\subsection{Invariance under the Reidemeister moves}\label{sec:invariance}

Let $\textit{Kof} = $Kom(Mat($\textit{Foams}_{/\ell}))$ be the category of complexes over $\textit{Foams}_{/\ell}$ and $\textit{Kof}_{/h} = $Kom$_{/h}$(Mat($\textit{Foams}_{/\ell}))$ the homotopy subcategory of the earlier. We remark that both categories are graded by degree.

 \begin{theorem}(Invariance Theorem)\label{invariance} The complex $[T]$ is invariant under the Reidemeister moves up to homotopy. In other words, it is an invariant in $\textit{Kof}_{/h}.$
 \end{theorem}
 \begin{proof}
\textit{Reidemeister\, 1a.}\quad
Consider diagrams $D_1$ and $D'$ that differ only in a circular region as in the figure below.
$$D_1=\raisebox{-13pt}{\includegraphics[height=0.4in]{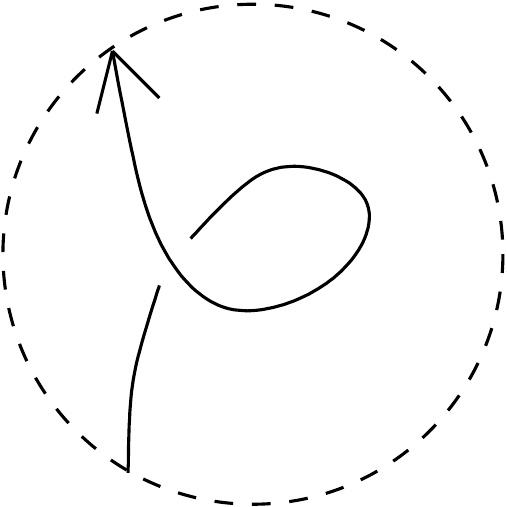}}\qquad
D'=\raisebox{-13pt}{\includegraphics[height=0.4in]{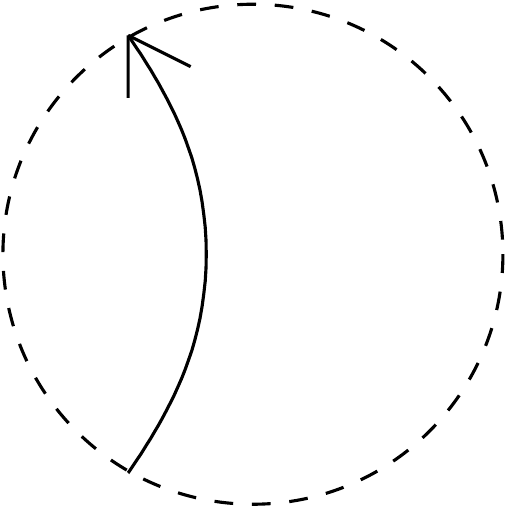}}$$

We give the homotopy equivalence between the formal complexes $[D_1]= (0 \longrightarrow \underline{\raisebox{-4pt}{\includegraphics[height=0.2in]{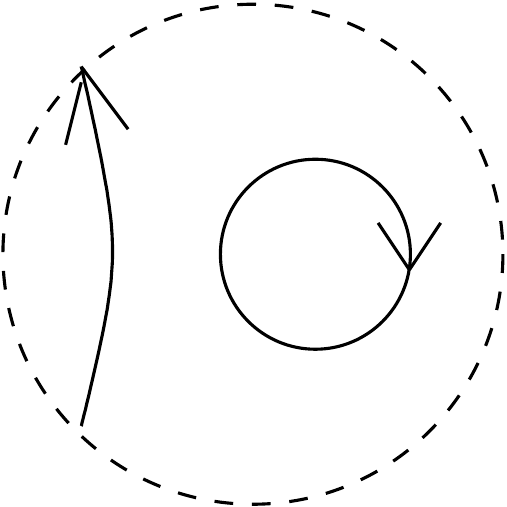}}\{-1\}}\stackrel{d}{\longrightarrow}\raisebox{-4pt} {\includegraphics[height=0.2in]{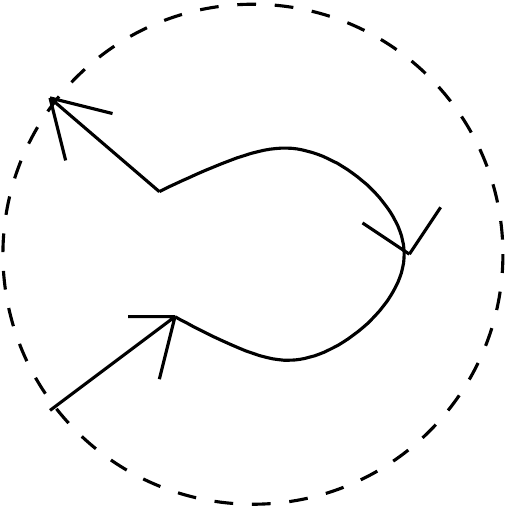}}\{-2\} \longrightarrow 0)$ and $[D']  = (0 \longrightarrow \underline{\raisebox{-4pt} {\includegraphics[height=0.2in]{reid1-1.pdf}}} \longrightarrow 0)$ in Figure~\ref{fig:R1a_Invariance}. We underlined the objects at the cohomological degree 0.
\begin{figure}[ht!]\begin{center}
$\xymatrix@R=27mm{
  [D'] \ar [d]^f:
\\
  [D_1] \ar@<4pt> [u]^g:
}
\xymatrix@C=45mm@R=22mm{
 \includegraphics[height=0.3in]{reid1-1.pdf}
  \ar@<3pt>[d]^{
        f^0 \;=\; \raisebox{-15pt}{\includegraphics[height=0.4in]{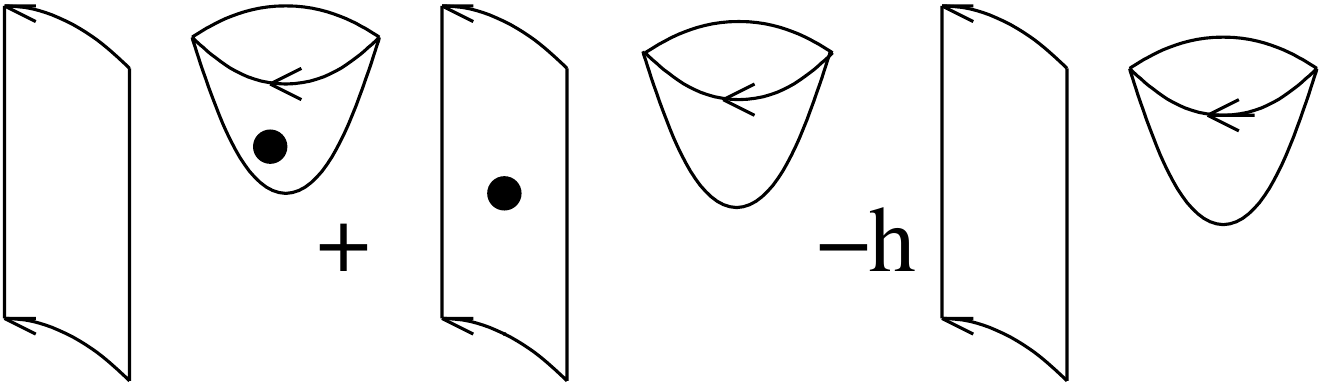}}} 
  \ar@<4pt>[r]^0
         &0 \\
 \includegraphics[height=0.3in]{reid1-2.pdf}\ar[r]^{
        d \;=\; \raisebox{-13pt}{\includegraphics[height=0.4in]{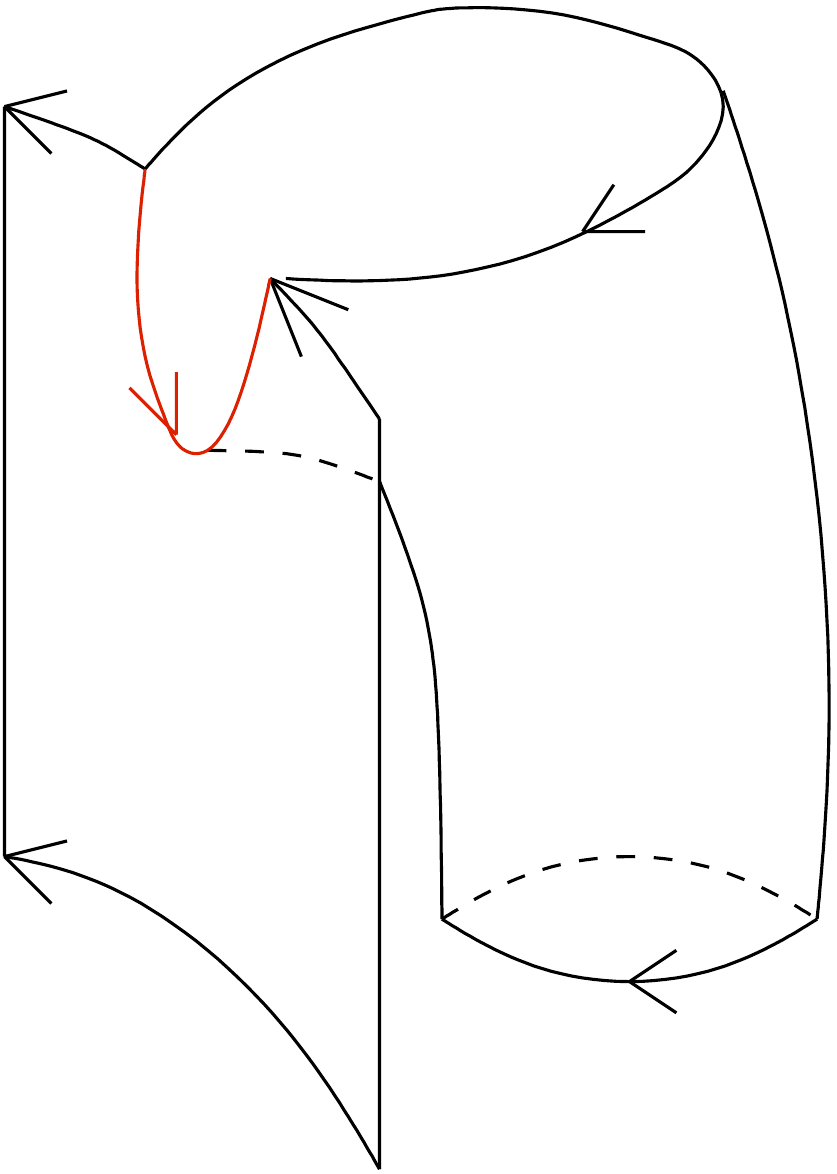}} }
  \ar@<3pt>[u]^{
        g^0 \;=\; \raisebox{-15pt}{\includegraphics[height=0.4in]{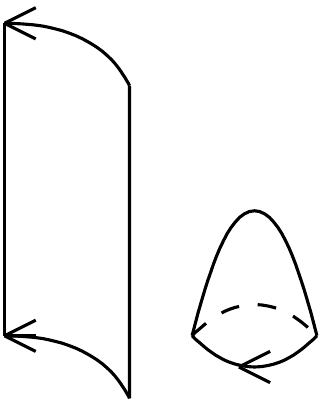}} } &
   \includegraphics[height=0.3in]{reid1-3.pdf} \ar@{<->}[u]_0 \ar@<2pt>[l]^{
     \tilde{h} \;=\; i \; \raisebox{-13pt}{\includegraphics[height=0.45in]{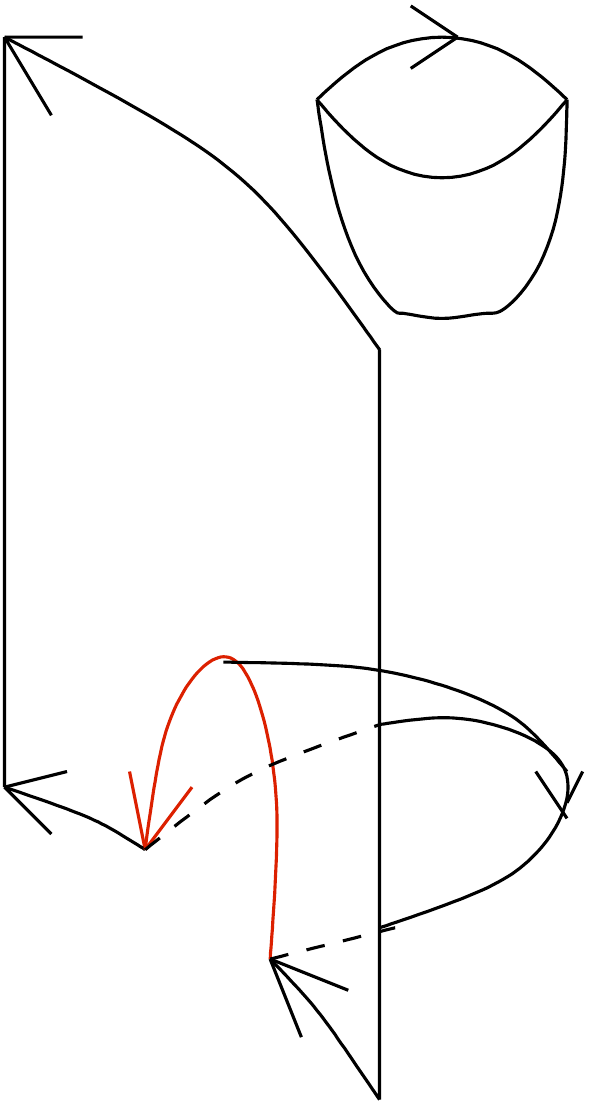}}  }
}$\end{center}
\caption{Invariance under $Reidemeister\,1a$} \label{fig:R1a_Invariance}
\end{figure}
The first (ED) identity implies that $df^0 = 0$ and (S) yields $g^0f^0 = \id(\raisebox{-4pt} {\includegraphics[height=0.2in]{reid1-1.pdf}}).$ The equality $d\tilde{h} = \id(\raisebox{-4pt} {\includegraphics[height=0.2in]{reid1-3.pdf}})$ follows from (CI). Finally, identity $f^0g^0 + \tilde{h}d  = \id(\raisebox{-4pt}{\includegraphics[height=0.2in]{reid1-2.pdf}})$ is obtained from (SF) and (SR). Therefore $[D_1]$ and $[D']$ are homotopy equivalent complexes.

\textit{Reidemeister\, 1b.}\quad
Consider diagrams $D_2$ and $D'$ that differ only in a circular region as in the following figure.
$$D_2=\raisebox{-13pt}{\includegraphics[height=0.4in]{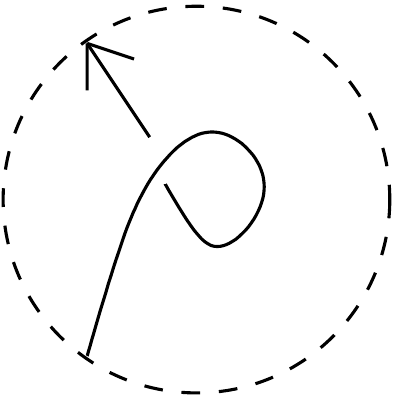}}\qquad
D'=\raisebox{-13pt}{\includegraphics[height=0.4in]{reid1-1.pdf}}$$

The diagram in Figure~\ref{fig:R1b_Invariance} gives the homotopy equivalence between formal complexes $[D_2] = (0 \longrightarrow \raisebox{-4pt}{\includegraphics[height=0.2in]{reid1-3.pdf}}\{2\}\stackrel{d}{\longrightarrow} \underline{\raisebox{-4pt} {\includegraphics[height=0.2in]{reid1-2.pdf}}\{1\}} \longrightarrow 0) \quad \text{and}\quad [D'] = (0 \longrightarrow \underline{\raisebox{-4pt} {\includegraphics[height=0.2in]{reid1-1.pdf}}} \longrightarrow 0).$
\begin{figure}[ht!]\begin{center}
\centerline{$\xymatrix@R=27mm{
  [D'] \ar [d]^f:
\\
  [D_2] \ar@<4pt>[u]^g:
}
\xymatrix@C=45mm@R=22mm{
0  \ar@<4pt>[r]^0
& \includegraphics[height=0.3in]{reid1-1}
  \ar@<3pt>[d]^{
       f^0 \;=\; \raisebox{-13pt}{\includegraphics[height=0.4in]{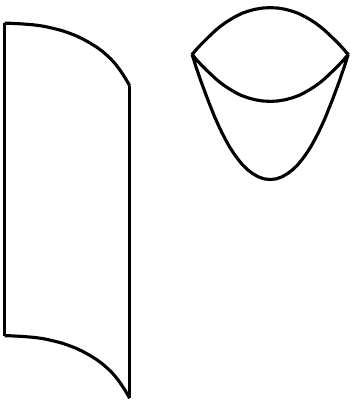}}} 
        \\
 \includegraphics[height=0.3in]{reid1-3} \ar@{<->}[u]_0
 \ar[r]^{
        d \;=\; \raisebox{-8pt}{\includegraphics[height=0.38in]{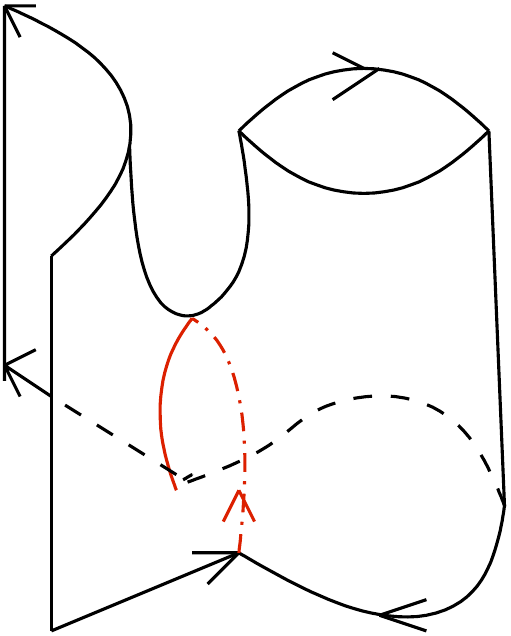}} }
 & \includegraphics[height=0.3in]{reid1-2} 
  \ar@<3pt>[u]^{
  g^0 \;=\; \raisebox{-13pt}{\includegraphics[height=0.4in]{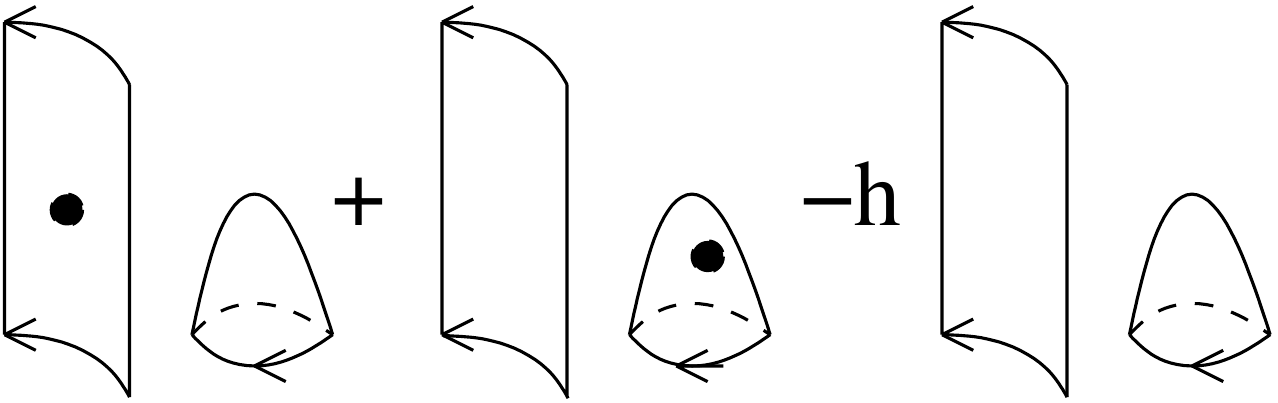}} } 
  \ar@<2pt>[l]^{
     \tilde{h} \;=\; i \;\raisebox{-10pt}{\includegraphics[height=0.38in]{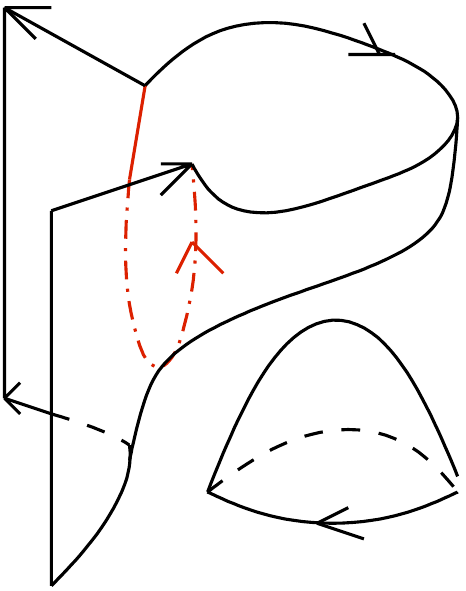}}  }
}$}\end{center}
\caption{Invariance under $Reidemeister\,1b$} \label{fig:R1b_Invariance}
\end{figure}

We have that $g^0f^0 = \id(\raisebox{-4pt} {\includegraphics[height=0.2in]{reid1-1.pdf}}),$ which follows from (S). The first (ED) identity implies $g^0d = 0,$ and (CI) gives $\tilde{h}d = \id(\raisebox{-4pt} {\includegraphics[height=0.2in]{reid1-3.pdf}}).$ Finally, $f^0g^0 + d\tilde{h} = \id(\raisebox{-4pt} {\includegraphics[height=0.2in]{reid1-2.pdf}})$ is obtained from (CI), (SF), (SR) and (ED). Thus $[D_2]$ is homotopy equivalent to $[D'].$

\textit{Reidemeister\, 2a.}\quad
Consider diagrams $D$ and $D'$ that differ in a circular region, as in figure below.
$$D=\raisebox{-13pt}{\includegraphics[height=0.4in]{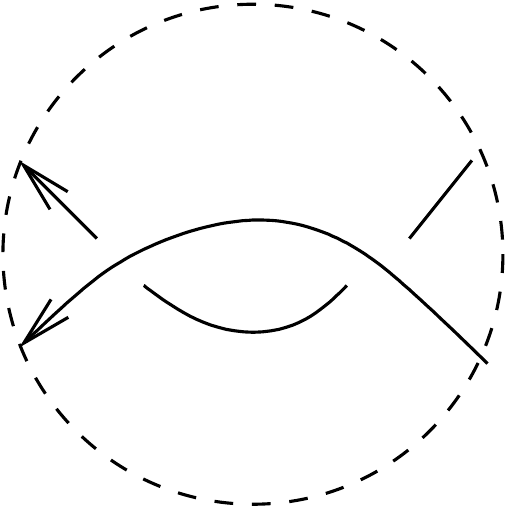}}\qquad
D'=\raisebox{-13pt}{\includegraphics[height=0.4in]{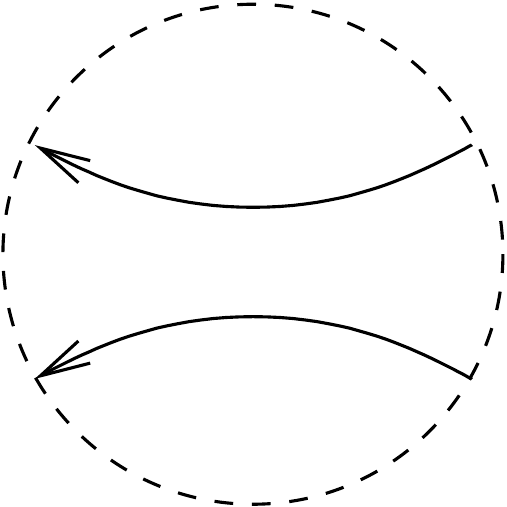}}\ $$
The homotopy equivalence between complexes $[D]$ and $[D']$ is given in Figure~\ref{fig:R2a_Invariance}, and we let the reader to check that the following equalities hold.
\begin{figure}[ht!]\begin{center}
\includegraphics[ height=3.5in]{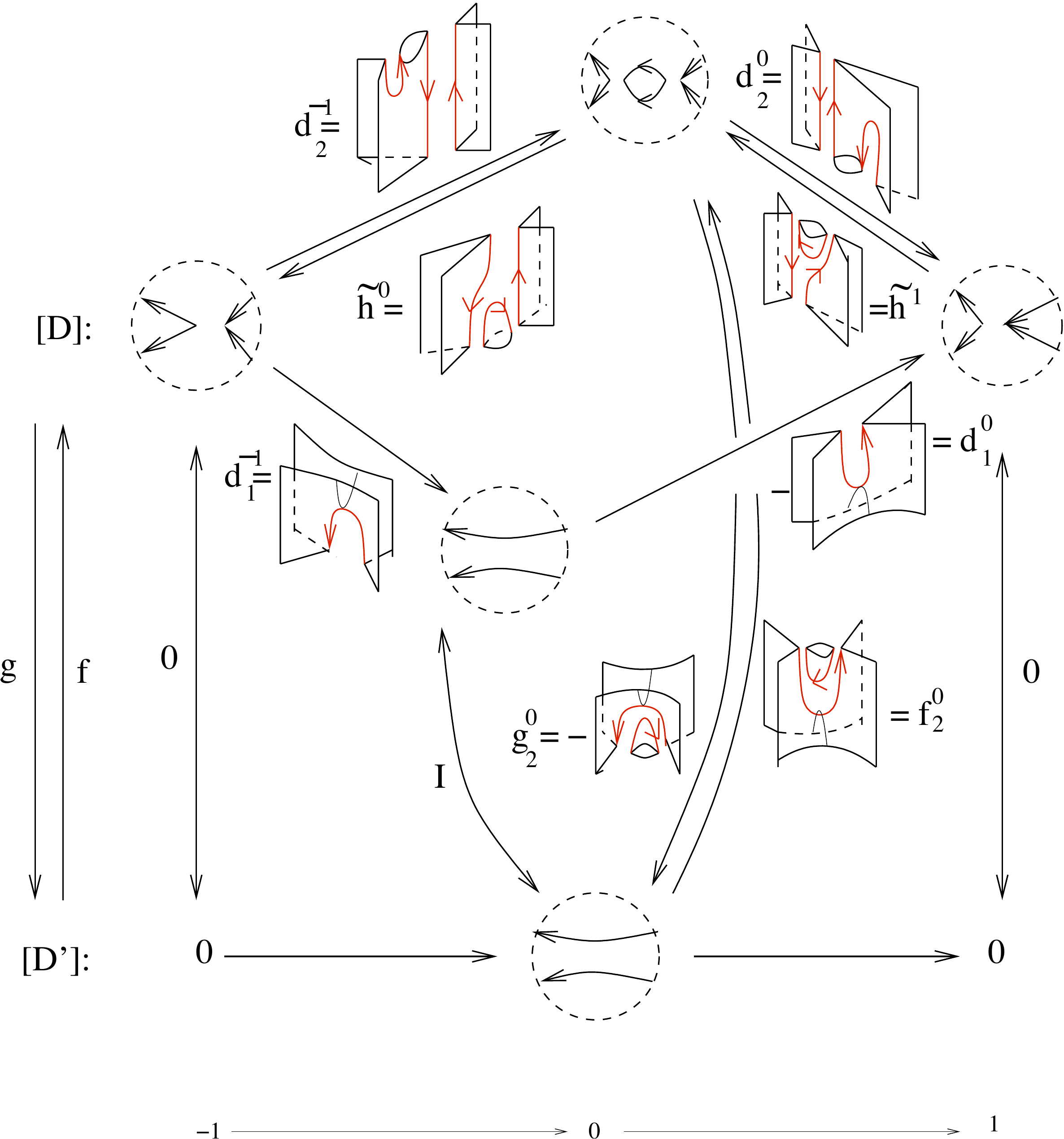}\end{center}
\caption{Invariance under $Reidemeister\,2a$} \label{fig:R2a_Invariance}
\end{figure}
\begin{itemize} 
\item $d_1^{-1} +  g_2^0d_2^{-1} = 0, \quad  d_1^0 + d_2^0f_2^0 = 0$ (it uses isotopies)
\item  $\tilde{h}^0d_2^{-1} = \id(\raisebox{-4pt} {\includegraphics[height=0.2in]{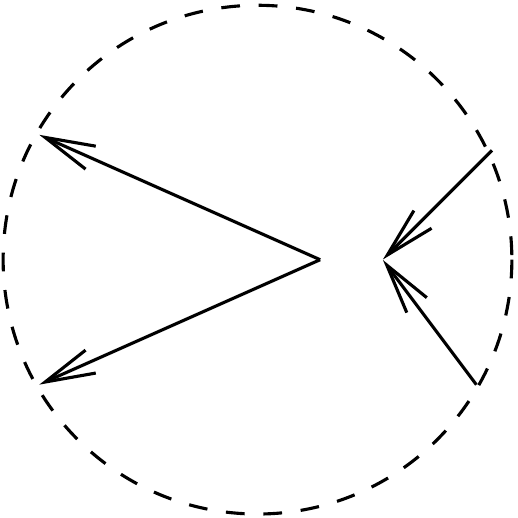}}), \quad d_2^0\tilde{h}^{1} = \id(\raisebox{-4pt} {\includegraphics[height=0.2in]{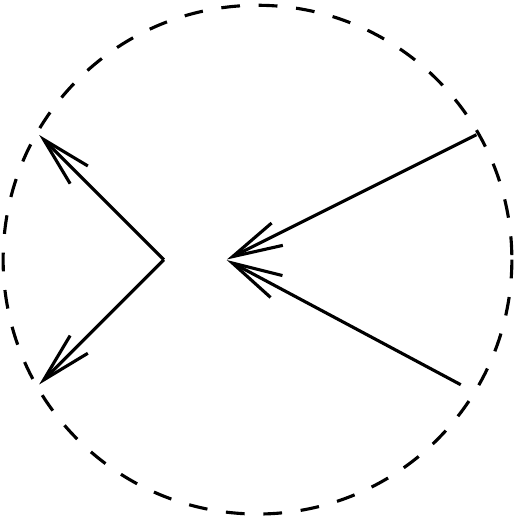}})$ (it uses isotopies)
\item $ f_2^0g_2^0 + d_2^{-1}\tilde{h}^0 + \tilde{h}^{1}d_2^0= \id(\raisebox{-4pt} {\includegraphics[height=0.2in]{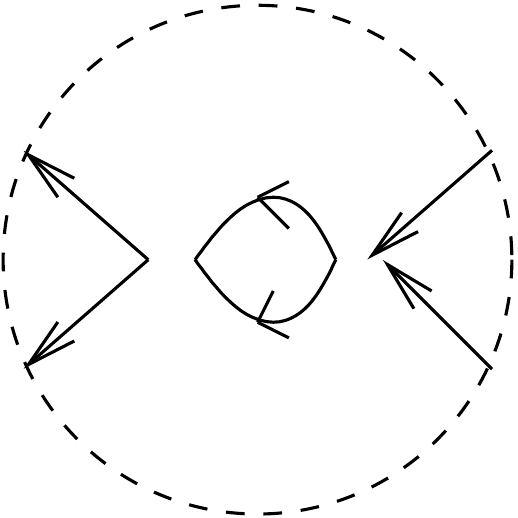}})$ (it follows from (CN))
\item $g_2^0f_2^0 = 0$ (it follows from (UFO)), thus $gf = \id(\raisebox{-4pt} {\includegraphics[height=0.2in]{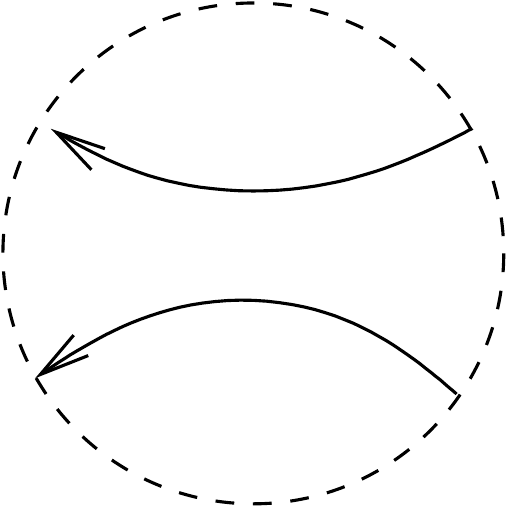}}).$
\end{itemize}

\textit{Reidemeister\, 2b.}\quad
Consider diagrams $D$ and $D'$ depicted below.
$$D=\raisebox{-13pt}{\includegraphics[height=0.4in]{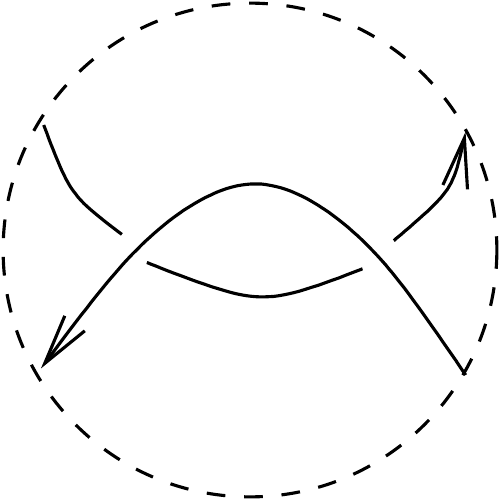}}\qquad
D'=\raisebox{-13pt}{\includegraphics[height=0.4in]{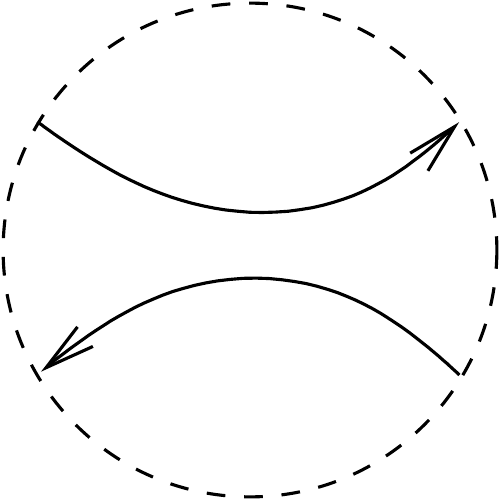}}\ $$
Checking that the diagram in Figure~\ref{fig:R2b_Invariance} defines a homotopy between $[D]$ and $[D']$ is left to the reader:

\begin{itemize}
\item $g_1^0 d_1^{-1} +  g_2^0d_2^{-1} = 0, \quad  d_1^0f_1^0 + d_2^0f_2^0 = 0$ (it uses isotopies)
\item $\tilde{h}_2^0d_2^{-1}  = \id( \raisebox{-4pt} {\includegraphics[height=0.2in]{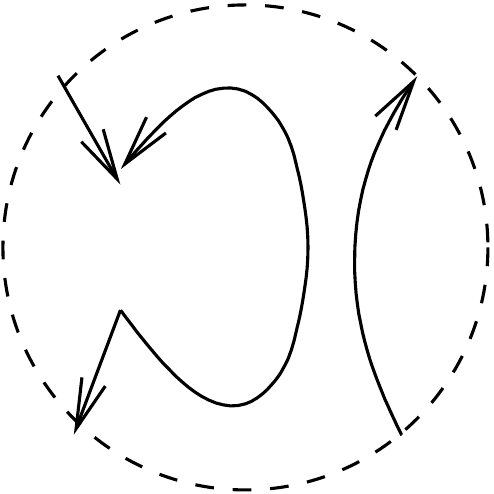}}),\quad  d_2^0\tilde{h}_2^1  = \id(\raisebox{-4pt} {\includegraphics[height=0.2in]{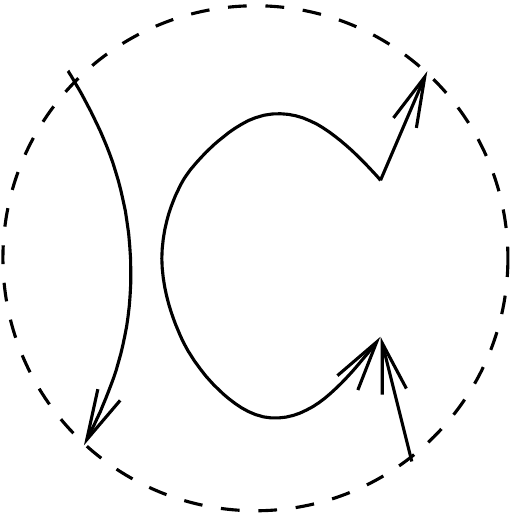}}), \quad f_1^0g_1^0 = \id(\raisebox{-4pt} {\includegraphics[height=0.2in]{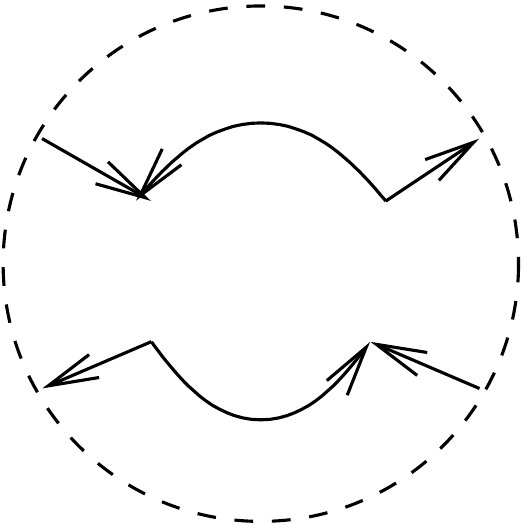}})$ (it follows from (CI))
 \item $f_2^0g_2^0 + d_2^{-1}\tilde{h}_2^0 +\tilde{h}_2^1d_2^0= \id(\raisebox{-4pt} {\includegraphics[height=0.2in]{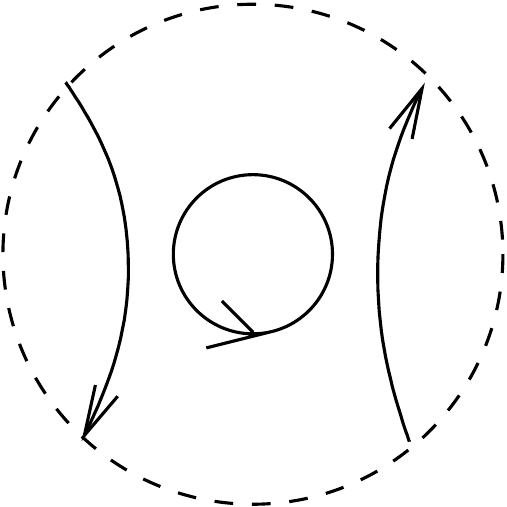}})$ (it follows from (SF) and (SR))
 \item  $g_1^0f_1^0 + g_2^0f_2^0 = \id(\raisebox{-4pt} {\includegraphics[height=0.2in]{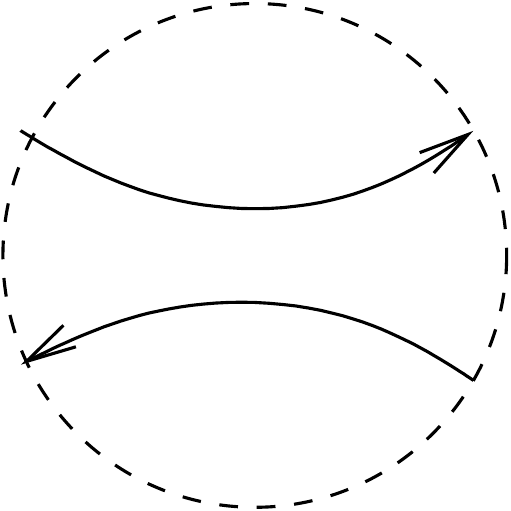}})$ (it follows from (S) and (SR)).
 \end{itemize}
 
\begin{figure}[ht!]\begin{center}
\includegraphics[width= 3.5in, height = 3.8in]{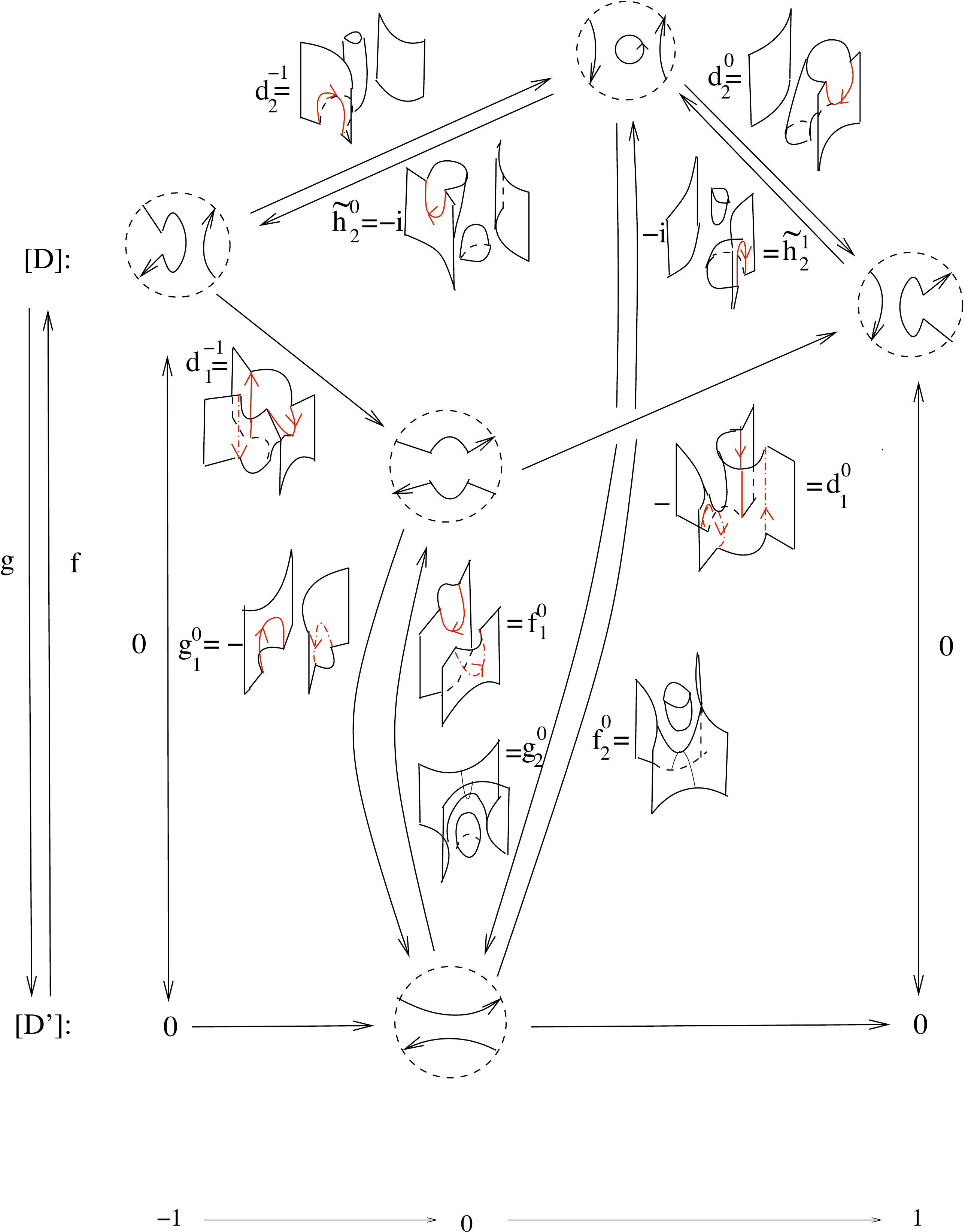}\end{center}
\caption{Invariance under $Reidemeister\,2b$} \label{fig:R2b_Invariance}
\end{figure}

\textit{Reidemeister\, 3.}\quad Any two Reidemeister moves of type 3 are equivalent modulo type 1 and 2 moves, thus it suffices to approach one case of the type 3 moves. We choose the one in figure below.
$$D = \raisebox{-10pt}{\includegraphics[height=.4in]{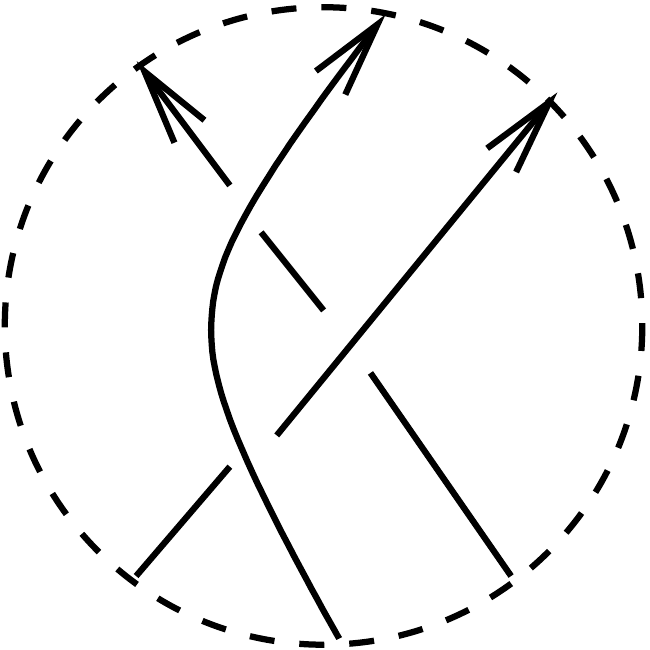}} \hspace{2cm} D' = \raisebox{-10pt}{\includegraphics[height=.4in]{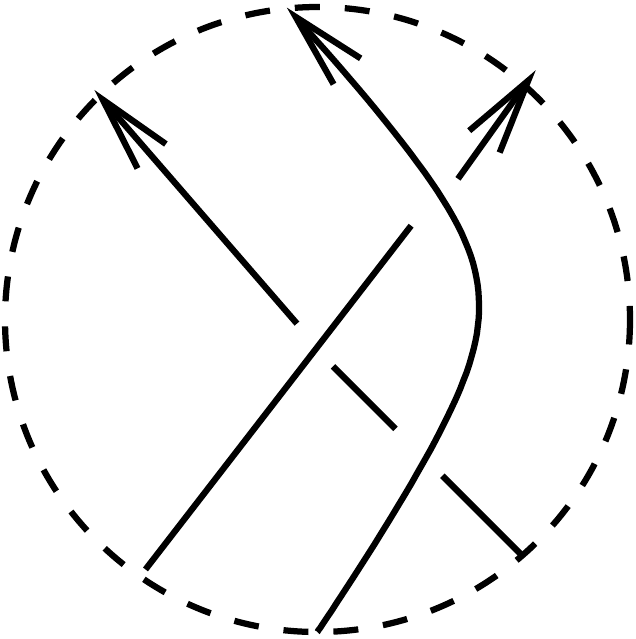}}$$

Given a morphism of complexes $\Phi,$ we denote its mapping cone by $\mathbf{M}(\Phi).$ Notice that  the mapping cone is invariant under composition with isomorphisms (see~\cite[Lemma 6.4]{CC}). Moreover, it was proved in~\cite{BN1} that the mapping cone construction is invariant up to homotopy under composition with strong deformation retracts, and with inclusions in strong deformation retracts. From the proof of invariance under type 2 Reidemeister moves, we know that morphisms $\raisebox{-5pt}{\includegraphics[height=0.25in]{twoarcs.pdf}}\stackrel{f} {\rightarrow} \raisebox{-5pt}{\includegraphics[height=0.25in]{Dreid2a.pdf}}$ and  $\raisebox{-5pt}{\includegraphics[height=0.25in]{twoarcsop.pdf}}\stackrel{f} {\rightarrow} \raisebox{-5pt}{\includegraphics[height=0.25in]{Dreid2b.pdf}}$
are inclusions in the strong deformation retracts $\raisebox{-5pt}{\includegraphics[height=0.25in]{Dreid2a.pdf}}\stackrel{g} {\rightarrow} \raisebox{-5pt}{\includegraphics[height=0.25in]{twoarcs.pdf}}$ and $\raisebox{-5pt}{\includegraphics[height=0.25in]{Dreid2b.pdf}}\stackrel{g} {\rightarrow} \raisebox{-5pt}{\includegraphics[height=0.25in]{twoarcsop.pdf}}$ respectively.
 
We have that $[\,\raisebox{-5pt}{\includegraphics[height=0.2in]{poscrossing.pdf}}\,] = \mathbf{M}(\,[\,\raisebox{-5pt}{\includegraphics[height=0.2in]{singres.pdf}}\,]\, \longrightarrow \, [ \,\raisebox{-5pt}{\includegraphics[height=0.2in]{orienres.pdf}}\, ]\,)$ and $[ \,\raisebox{-5pt}{\includegraphics[height=0.2in]{negcrossing.pdf}}\,] = \mathbf{M}(\,[ \,\raisebox{-5pt}{\includegraphics[height=0.2in]{orienres.pdf}}\, ] \longrightarrow \, [ \,\raisebox{-5pt}{\includegraphics[height=0.2in]{singres.pdf}}\, ]\,)[-1],$ where [$s$] is the shift operator that shifts complexes $s$ steps to the left.

Therefore, the complex $[\, \raisebox{-8pt}{\includegraphics[height=.25in]{reid3left.pdf}}\,]$ (or complex $[\, \raisebox{-8pt}{\includegraphics[height=.25in]{reid3right.pdf}}\,]$) is the cone of the morphism $\Phi_1$ (or $\Phi_2$) given in Figure~\ref{fig:reid3-1}, morphism that switches between the two resolutions of the central crossing (note that we could have used any crossing of the diagram). 
\begin{figure}[ht!]\begin{center}
\includegraphics[height=2.5in]{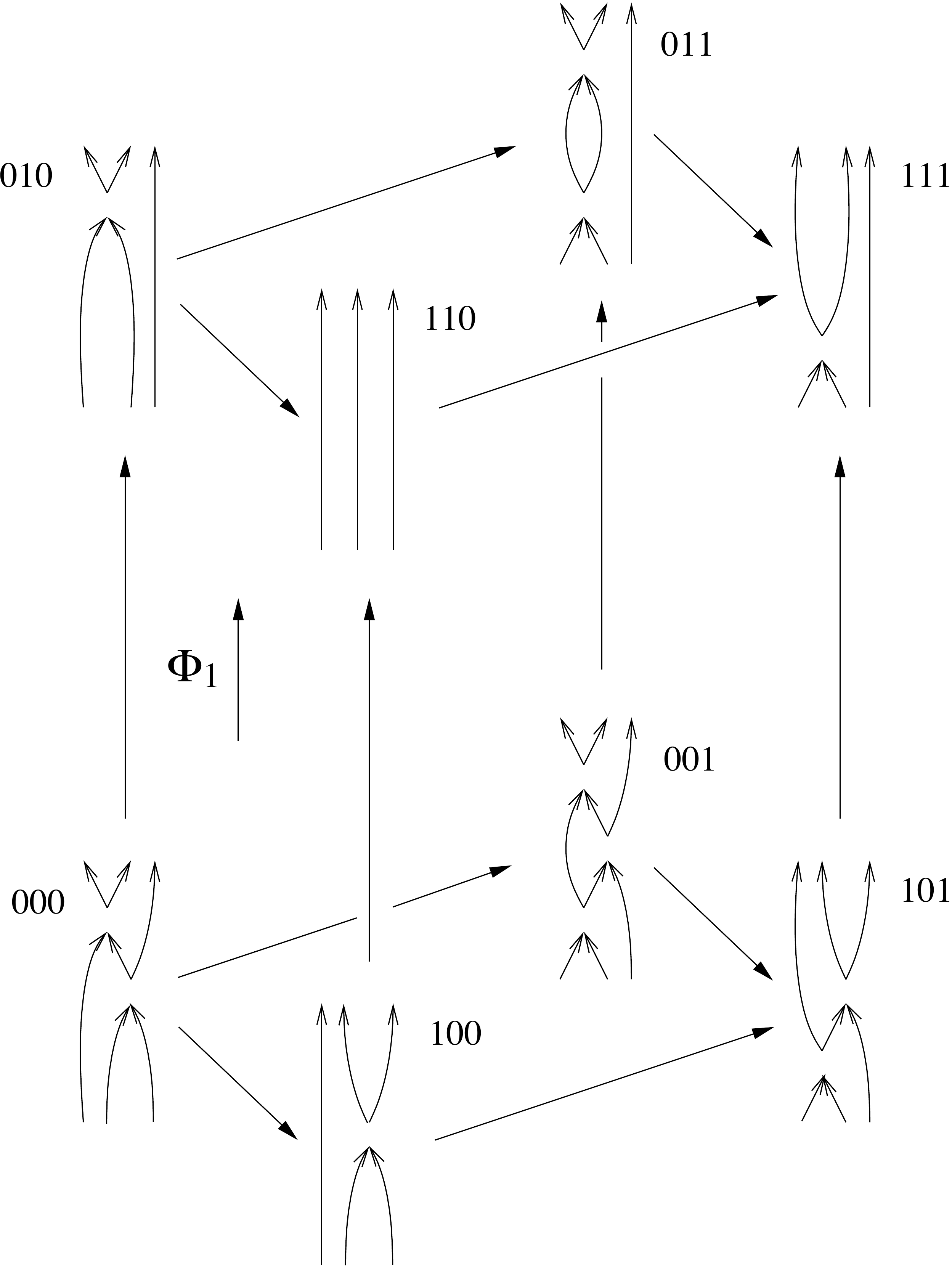}\hspace{2cm}
\includegraphics[height=2.5in]{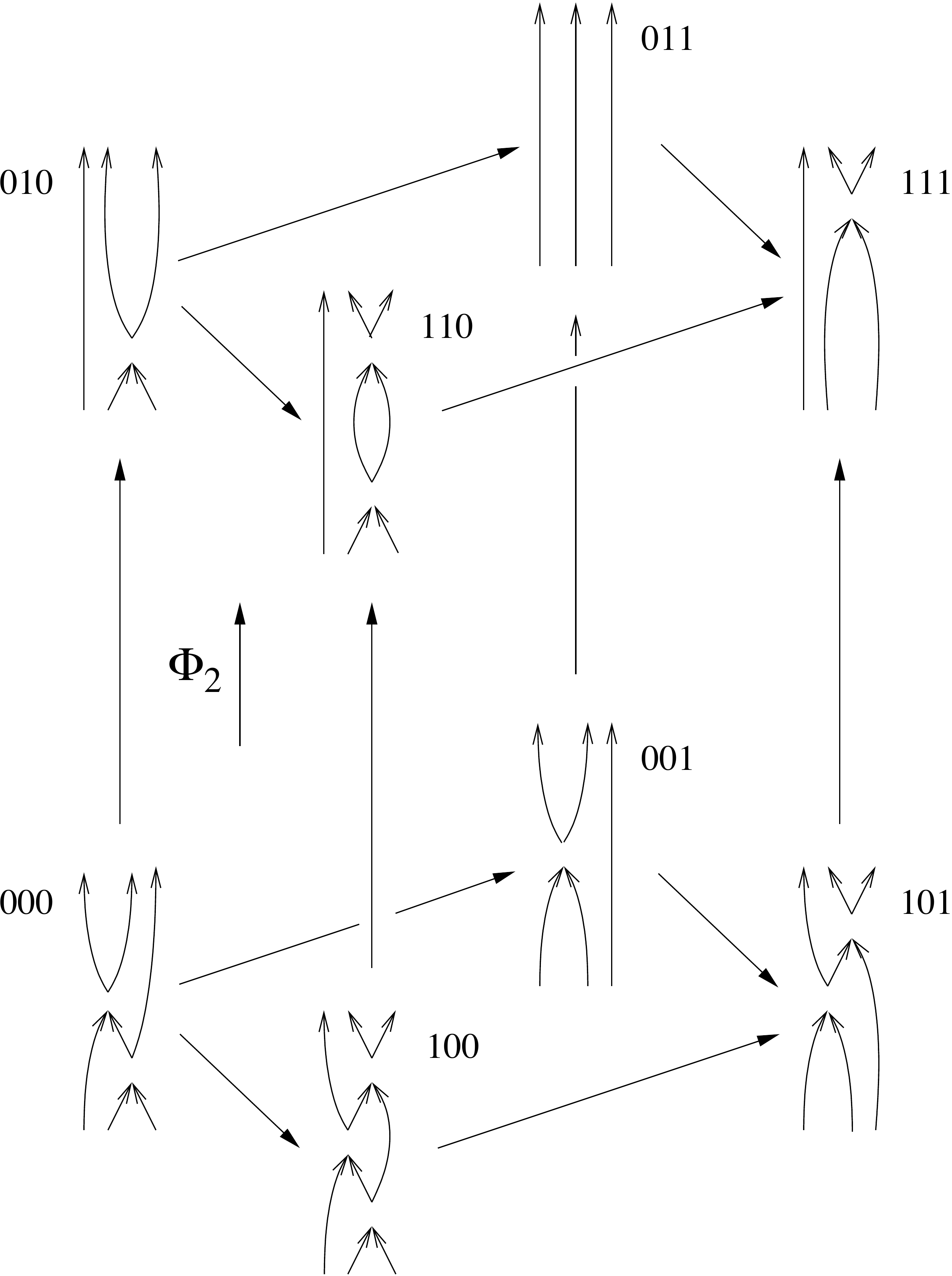}\end{center}
\caption{Cubes of resolutions of $D$ and $D'$}\label{fig:reid3-1}\end{figure}

The top layers of the cubes contain the four resolutions corresponding to a Reidemeister 2a move and an additional vertical string. Composing the morphisms $\Phi_1$ and $\Phi_2$ with strong deformation retracts $g,$ we can replace the top layers with the resolution containing three vertical strings. In particular, the complex $[ \, \raisebox{-8pt}{\includegraphics[height=.25in]{reid3left.pdf}}\,]$ is homotopy equivalent to the cone of the morphism $\Phi_L = g \Phi_1,$ while $[\, \raisebox{-8pt}{\includegraphics[height=.25in]{reid3right.pdf}}\,]$ is homotopy equivalent to the cone of $\Phi_R = g \Phi_2.$
\begin{figure}[ht!]\begin{center}
\raisebox{-13pt}{\includegraphics[height=2.2in]{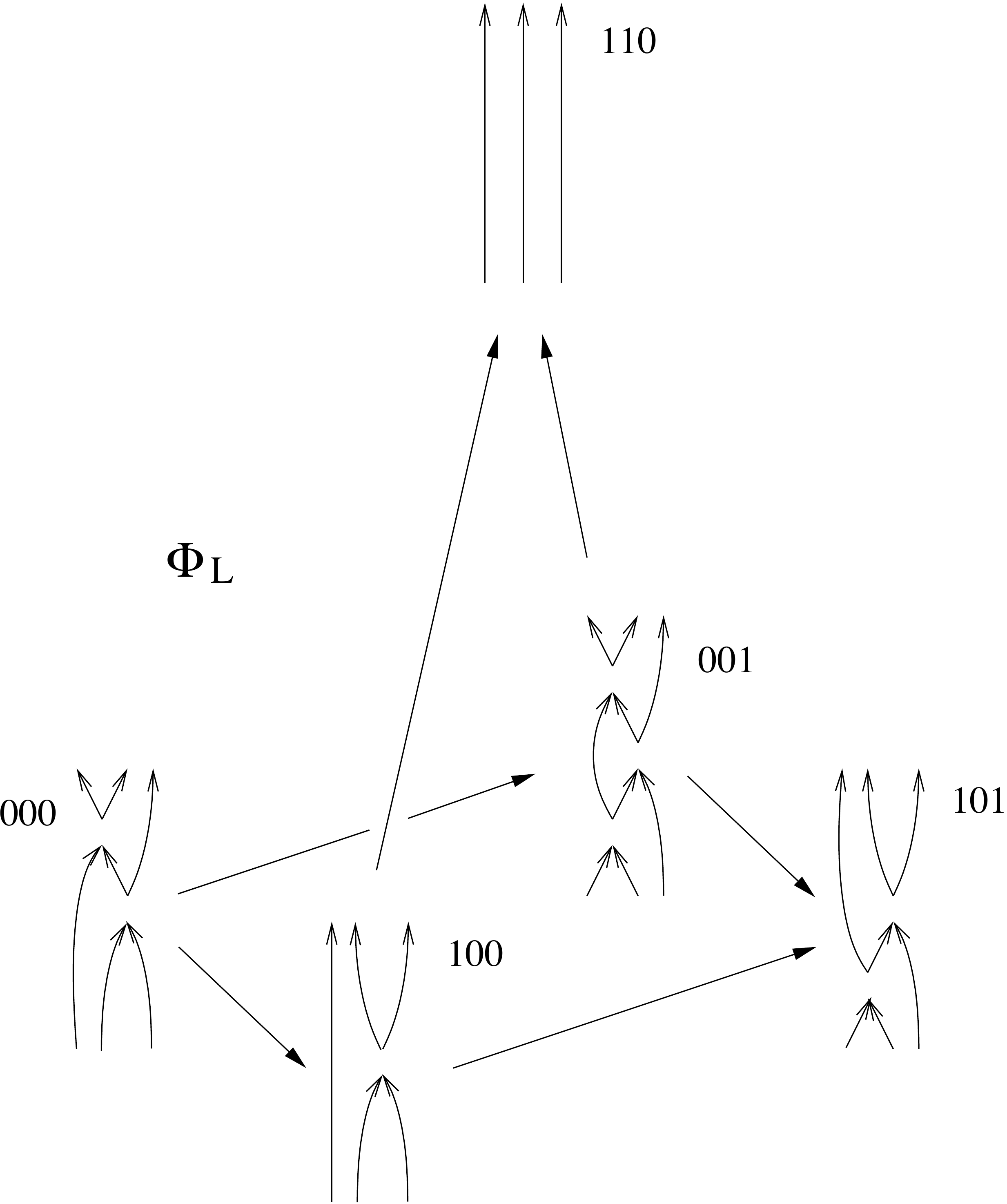}}\hspace{2cm}
\raisebox{-13pt}{\includegraphics[height=2.2in]{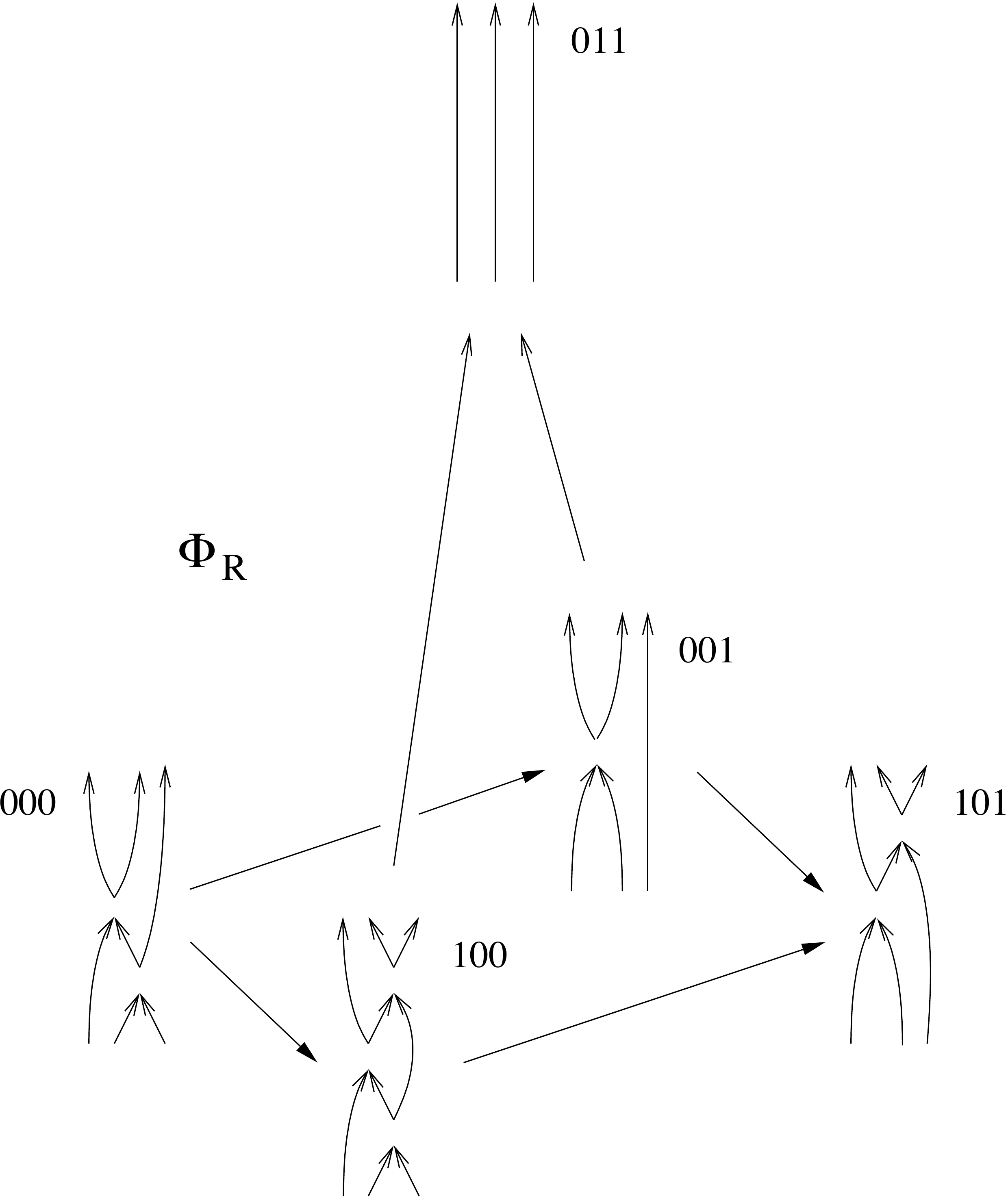}}\end{center}
\caption{Morphisms $\Phi_L$ and $\Phi_R$}\label{fig:reid3-2}
\end{figure}
In other words, we have:
 \begin{align*}
[\,\raisebox{-8pt}{\includegraphics[height=0.3in]{reid3left.pdf}}\,] &= \mathbf{M}\left([\,\raisebox{-8pt}{\includegraphics[height=0.3in]{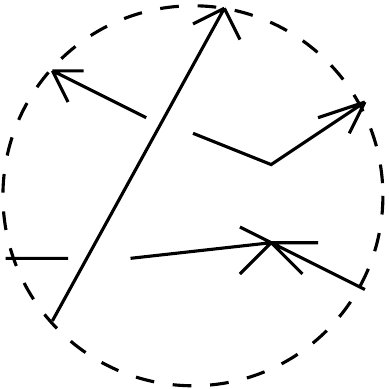}}\,] \xrightarrow{\Phi_1} [\,\raisebox{-8pt}{\includegraphics[height=0.3in]{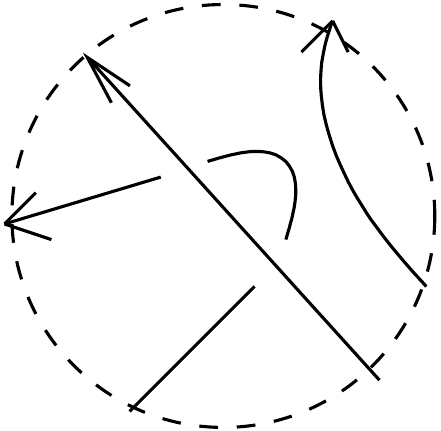}}\,] \right)\xrightarrow {\cong} 
 \mathbf{M}\left([\,\raisebox{-8pt}{\includegraphics[height=0.3in]{reid3-unor4.pdf}}\,] \xrightarrow{\Phi_1} [\,\raisebox{-8pt}{\includegraphics[height=0.3in]{reid3-8.pdf}}\,] \xrightarrow{g} [\,\raisebox{-8pt}{\includegraphics[height=0.3in]{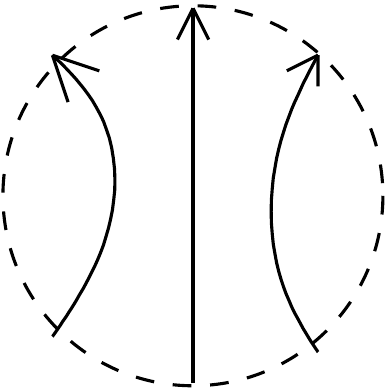}}\,] \right)  = \mathbf{M}(\Phi_L),\\ 
[\,\raisebox{-8pt}{\includegraphics[height=0.3in]{reid3right.pdf}}\,] &= \mathbf{M}\left([\,\raisebox{-8pt}{\includegraphics[height=0.3in]{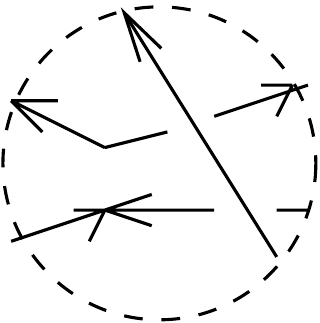}}\,] \xrightarrow{\Phi_2} [\,\raisebox{-8pt}{\includegraphics[height=0.3in]{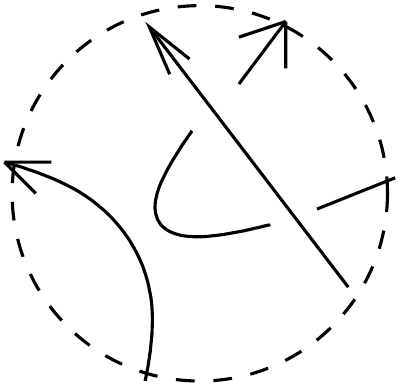}}\,] \right)\xrightarrow {\cong} 
 \mathbf{M}\left([\,\raisebox{-8pt}{\includegraphics[height=0.3in]{reid3-unor3.pdf}}\,] \xrightarrow{\Phi_2} [\,\raisebox{-8pt}{\includegraphics[height=0.3in]{reid3-7.pdf}}\,] \xrightarrow{g} [\,\raisebox{-8pt}{\includegraphics[height=0.3in]{reid3-or.pdf}}\,] \right)  = \mathbf{M}(\Phi_R).\end{align*} 

The resolution $ijk$ in the left drawing of Figure~\ref{fig:reid3-2} is either isotopic or isomorphic in $\textit{Foams}_{/\ell}$ to the resolution $ijk$ in the right drawing of the same figure (see Corollary~\ref{cor:Isomorphisms 1 and 2}). We claim that the cones of the morphisms $\Phi_L$ and  $\Phi_R$ are isomorphic.

The chain complexes associated to the diagrams \raisebox{-8pt}{\includegraphics[height=0.3in]{reid3-unor4.pdf}} and \raisebox{-8pt}{\includegraphics[height=0.3in]{reid3-unor3.pdf}} are isomorphic in the category $\textit{Kof}_{/h},$ and the corresponding isomorphism is given in Figure~\ref{fig:sliding1}.
\begin{figure}[ht]
 \raisebox{-13pt}{\includegraphics[height=2.7in]{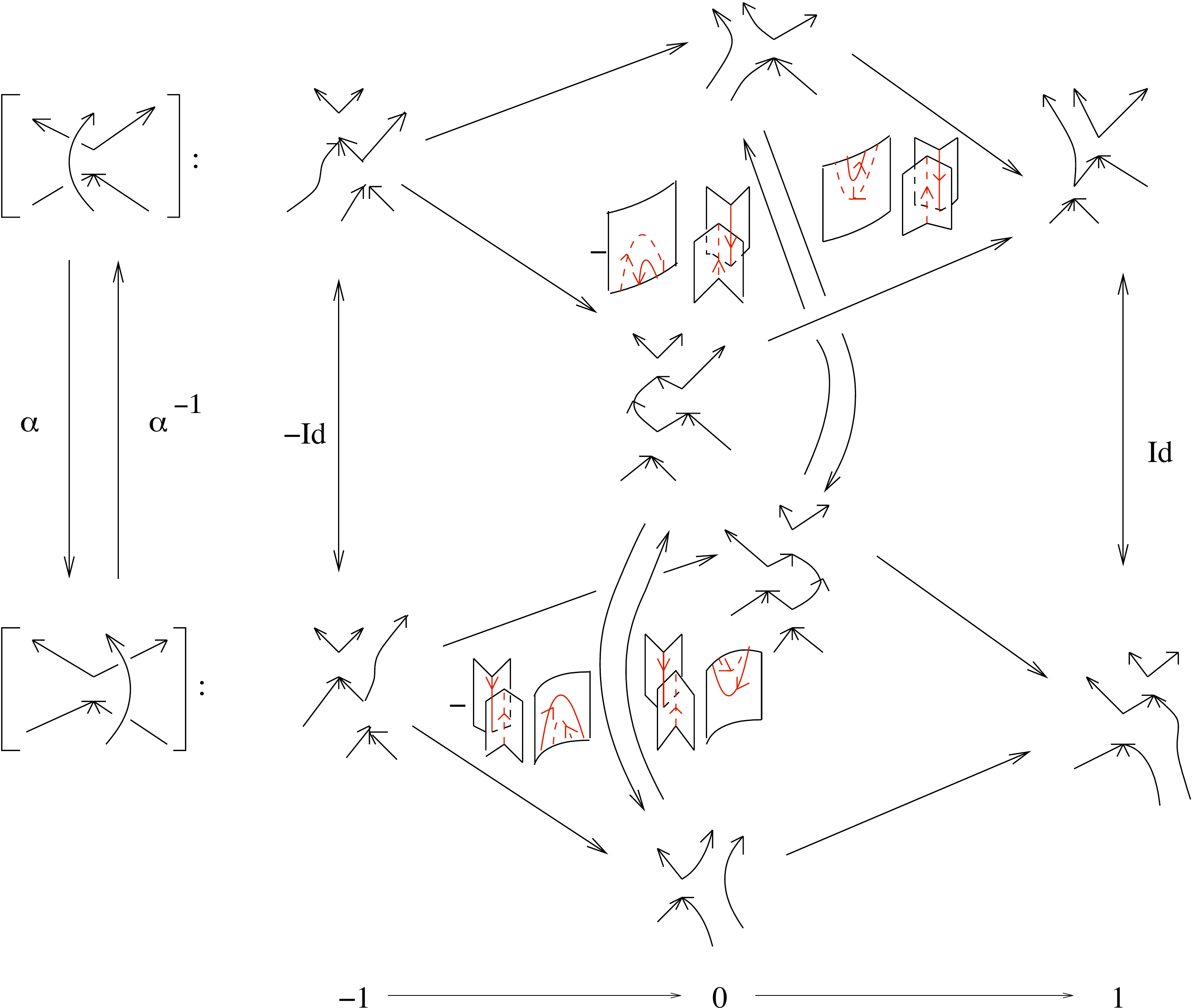}}
 \caption{Isomorphism $\alpha$}
 \label{fig:sliding1}
 \end{figure}

We obtain that
\begin{align*}
\mathbf{M}(\Phi_R)  \xrightarrow {\cong}  \mathbf{M}\left([\,\raisebox{-8pt}{\includegraphics[height=0.3in]{reid3-unor4.pdf}}\,] \xrightarrow{\alpha} [\,\raisebox{-8pt}{\includegraphics[height=0.3in]{reid3-unor3.pdf}}\,] \xrightarrow{\Phi_2} [\,\raisebox{-8pt}{\includegraphics[height=0.3in]{reid3-7.pdf}}\,] \xrightarrow{g} [\,\raisebox{-8pt}{\includegraphics[height=0.3in]{reid3-or.pdf}}\,] \right)
\end{align*}
where $\alpha$ is the isomorphisms depicted in Figure~\ref{fig:sliding1}.
To complete the proof, we show that the following compositions of chain maps give the same answer:
\[ [\,\raisebox{-8pt}{\includegraphics[height=0.3in]{reid3-unor4.pdf}}\,] \xrightarrow{\Phi_1} [\,\raisebox{-8pt}{\includegraphics[height=0.3in]{reid3-8.pdf}}\,] \xrightarrow{g} [\,\raisebox{-8pt}{\includegraphics[height=0.3in]{reid3-or.pdf}}\,] \quad \mbox{and} \quad
 [\,\raisebox{-8pt}{\includegraphics[height=0.3in]{reid3-unor4.pdf}}\,] \xrightarrow{\alpha} [\,\raisebox{-8pt}{\includegraphics[height=0.3in]{reid3-unor3.pdf}}\,] \xrightarrow{\Phi_2} [\,\raisebox{-8pt}{\includegraphics[height=0.3in]{reid3-7.pdf}}\,] \xrightarrow{g} [\,\raisebox{-8pt}{\includegraphics[height=0.3in]{reid3-or.pdf}}\,] .\]
 
 There are four morphisms in each of these compositions, but two of them are zero. The non-trivial maps are depicted below.
 \[ \includegraphics[height=1.8in]{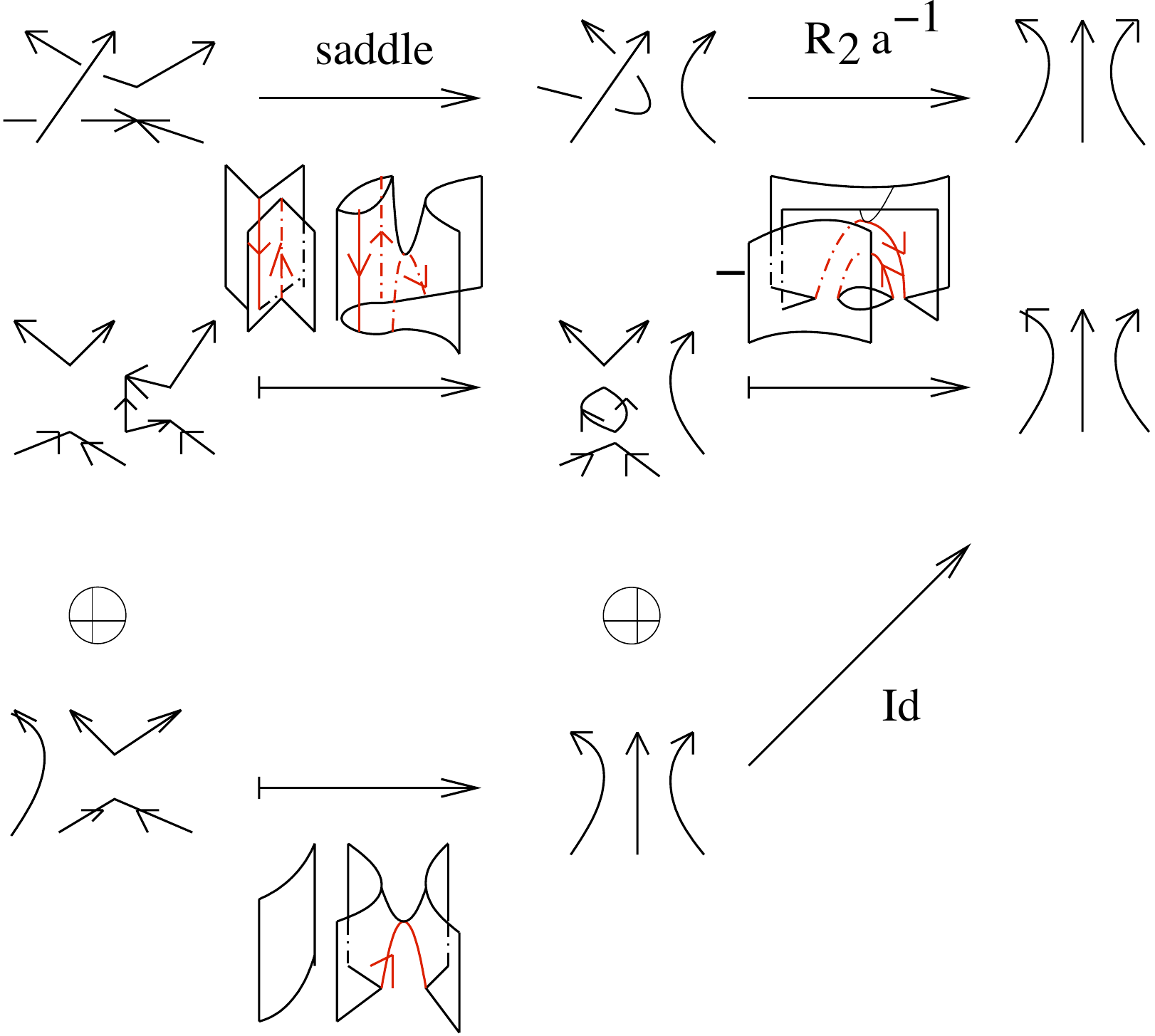}\quad \includegraphics[height=1.8in]{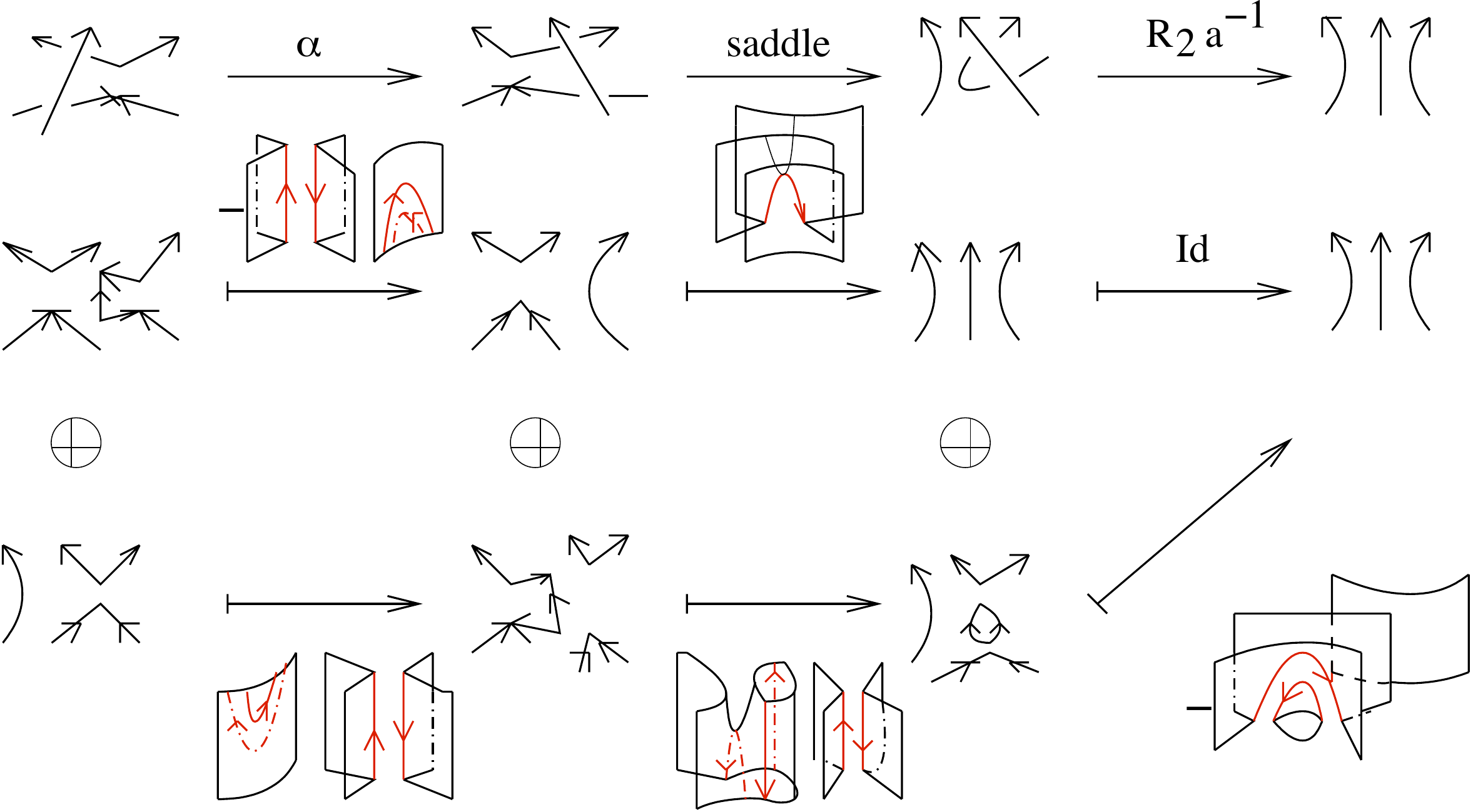}\]
Composing the morphisms above, and applying the `sheet relations' (SR) for the second component of the morphism on the right, we obtain that the non-trivial components of the two chain maps are, up to isotopy, equal to $ \left (\begin{array}{cc} -\raisebox{-8pt}{\includegraphics[height=0.4in]{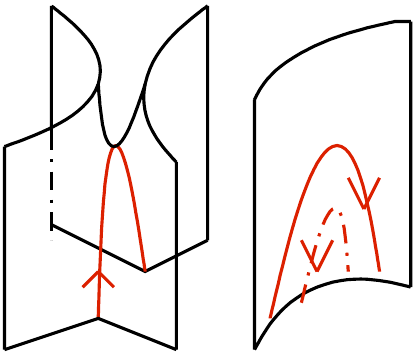}}\,\,\,, & \raisebox{-8pt}{\includegraphics[height=0.4in]{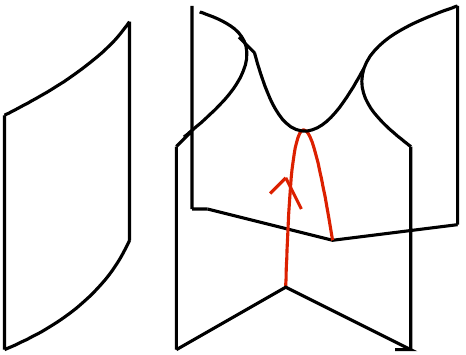}} \end{array} \right ).$

Therefore, the mapping cones $\mathbf{M}(\Phi_L)$ and $\mathbf{M}(\Phi_R)$ are isomorphic, and thus the complex $[\, \raisebox{-8pt}{\includegraphics[height=.25in]{reid3left.pdf}}\,]$ is homotopy equivalent to the complex $[\, \raisebox{-8pt}{\includegraphics[height=.25in]{reid3right.pdf}}\,].$ This completes the proof of the invariance under Reidemeister moves of type 3. \end{proof}

 
\subsection{Functoriality}\label{sec:functoriality}

Let $\textit{Cob}^4_{/i}$ be the category of oriented tangles and ambient isotopy classes of tangle cobordisms properly embedded in the 4-dimensional space.

\begin{theorem}\label{thm:functoriality}
$[ \cdot]$ induces a degree-preserving functor $\textit{Cob}^4_{/i} \to \textit{Kof}_{/h}.$
\end{theorem}

 \begin{proof} The proof is similar to that of~\cite[Theorem 3]{CC}, with the only difference that it uses the homotopy equivalence constructed in the present Invariance Theorem. To not repeat ourselves, we refer the reader to our previous work. The first step there was to show---much as Bar-Natan did in~\cite{BN1}---that the construction satisfies the functoriality property up to multiplication by a unit $\{ \pm1, \pm i \}.$ The second step of the proof consisted in considering each movie move of Carter and Saito~\cite{CS} and checking that the units $\{-1, \pm i \}$ actually don't appear; for some of the movie moves, this was done using the idea of working with ``homotopically isolated objects", which was borrowed from~\cite{CMW}.
 
Since the differences in the homotopy equivalence constructed in the Invariance Theorem of the generalized construction and the one in the earlier work~\cite{CC} appear only in the Reidemeiter 1 moves, we approach here only the movie moves involving R1 moves---these are MM 7, MM 8, MM 12 and MM 13---and refer the reader to~\cite{CC} for the other ones. 
 
On the other hand, we should mention that when checking a movie move involving a Reidemeister 3 move, one needs to know the map between two particular resolutions of the two sides of R3 move. For that, one needs a deeper approach to R3 moves, and we refer the reader again to~\cite{CC}, noting that the results for this type of Reidemeister moves hold in the generalized theory, as well.

\[\includegraphics[height=1.2in]{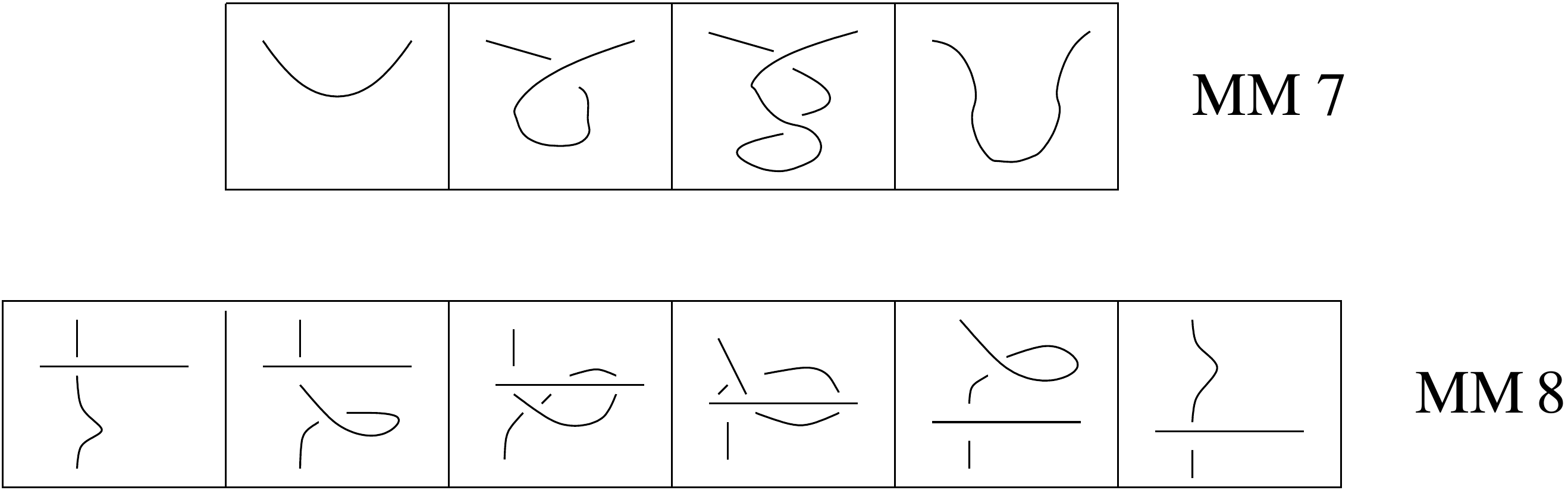}\]

The circular clips MM7 and MM8 have the same initial and final frames and are equivalent to identity, thus we need to show that the associated morphisms at the chain level are homotopy equivalent to the identity morphism.

\textbf{MM7.}\quad
For a negative crossing in the second frame of the clip we have:

\[ \includegraphics[height=1.8in]{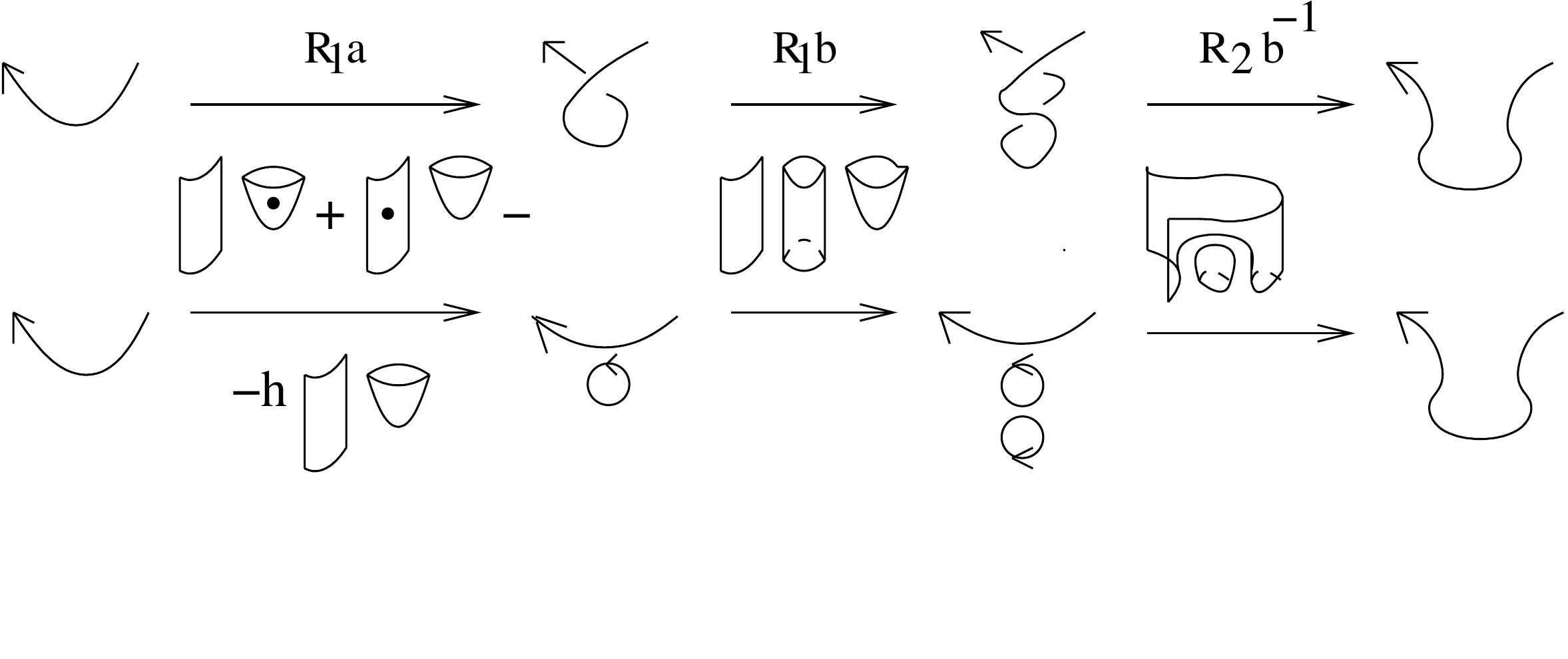}\]
For a positive crossing in the second frame of the same clip we obtain:

\[\includegraphics[height=1.8in]{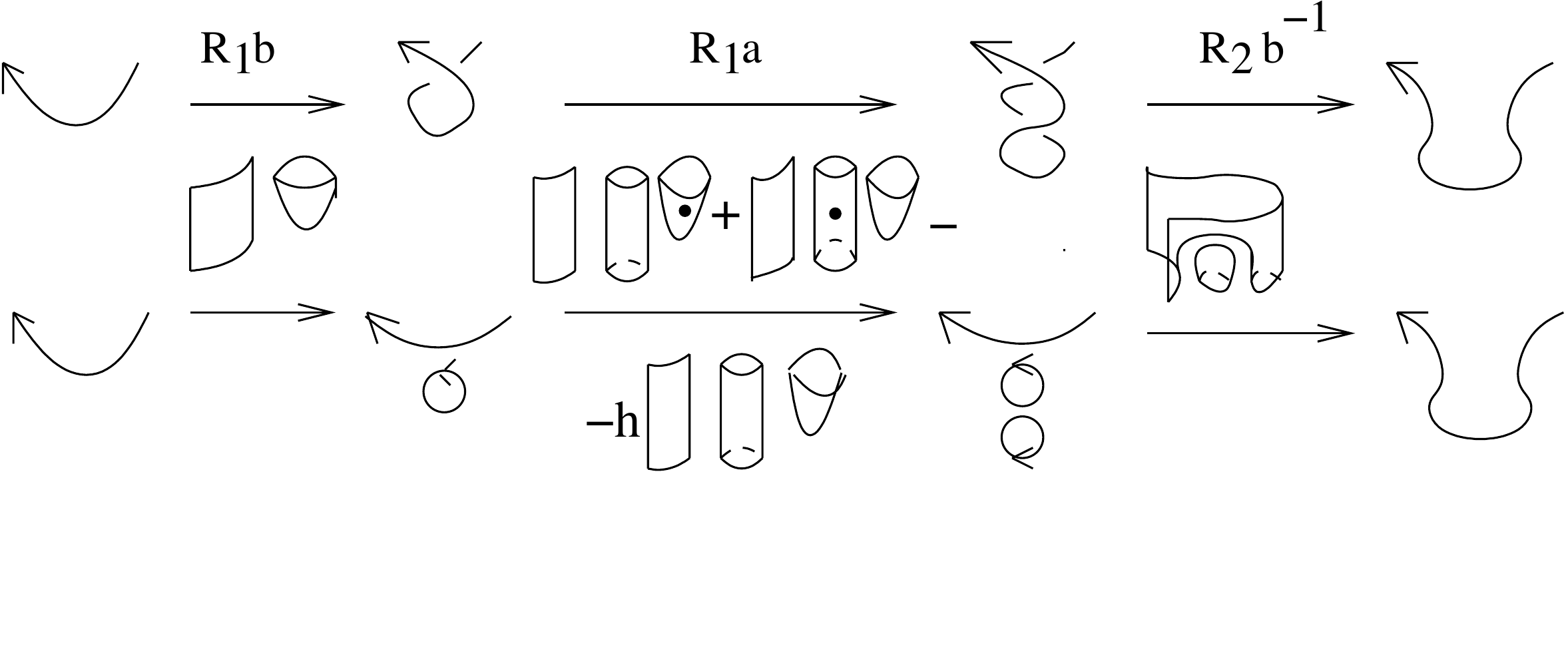}\]
Composing the above cobordisms and applying relation (S) we obtain vertical ``curtains'' in both cases, thus the morphisms are the identity. 

\textbf{MM8.}\quad
Let us look first what happens when the R1 move introduces a negative crossing in the second frame of the clip.	

\[\includegraphics[height=1.5in]{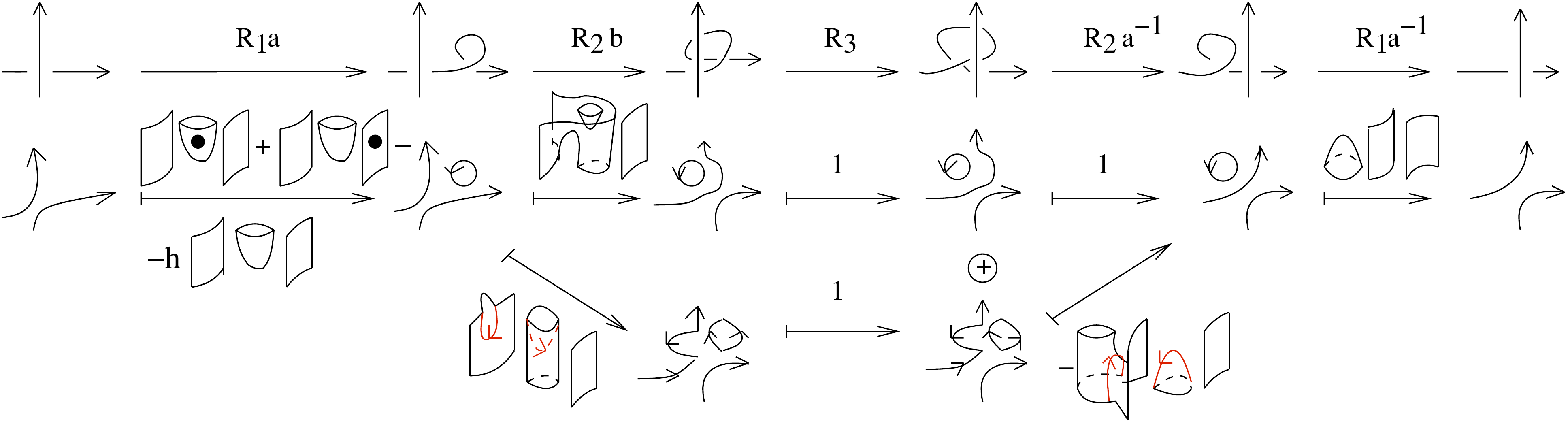}\]	
By composing above and using relation (S) again, we get the zero map in the first row. In the second row we arrive at:

\[\left (\,\raisebox{-10pt}{\includegraphics[height=0.45in]{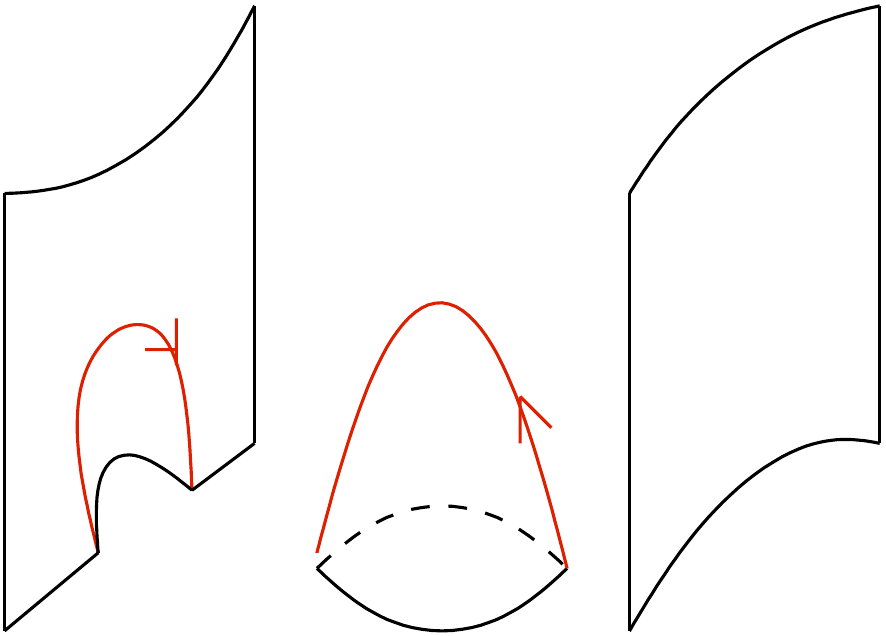}}\,\right) \circ \left(\,-\raisebox{-10pt}{\includegraphics[height=0.45in]{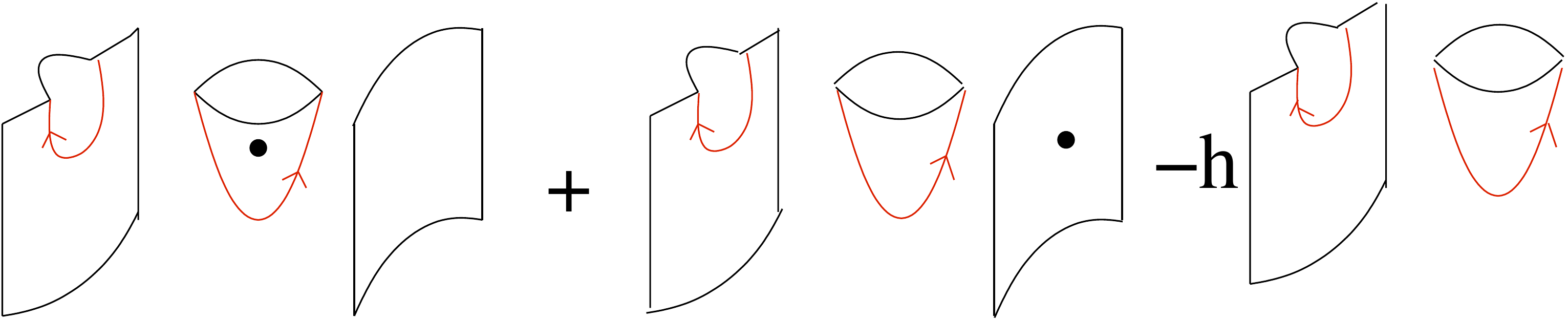}}\,\right) = \raisebox{-10pt}{\includegraphics[height=0.45in]{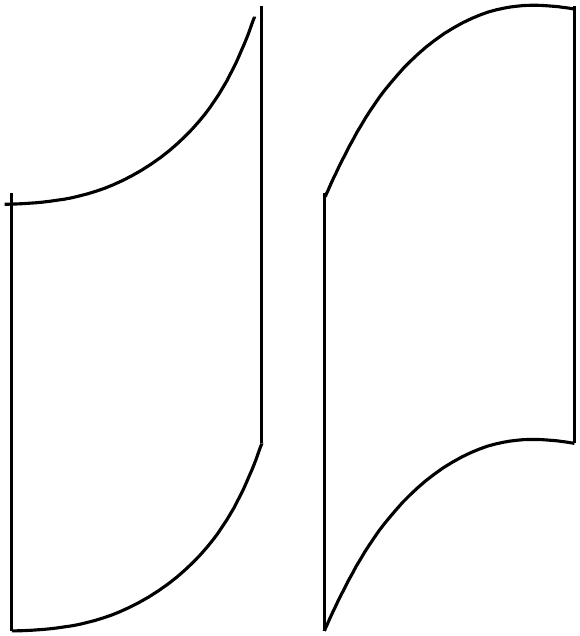}}\,,\]

which is obtained from (S), (UFO) and (SR).
If we consider now the case of a positive crossing introduced by the R1 move, we have:

\[\includegraphics[height=1.5in]{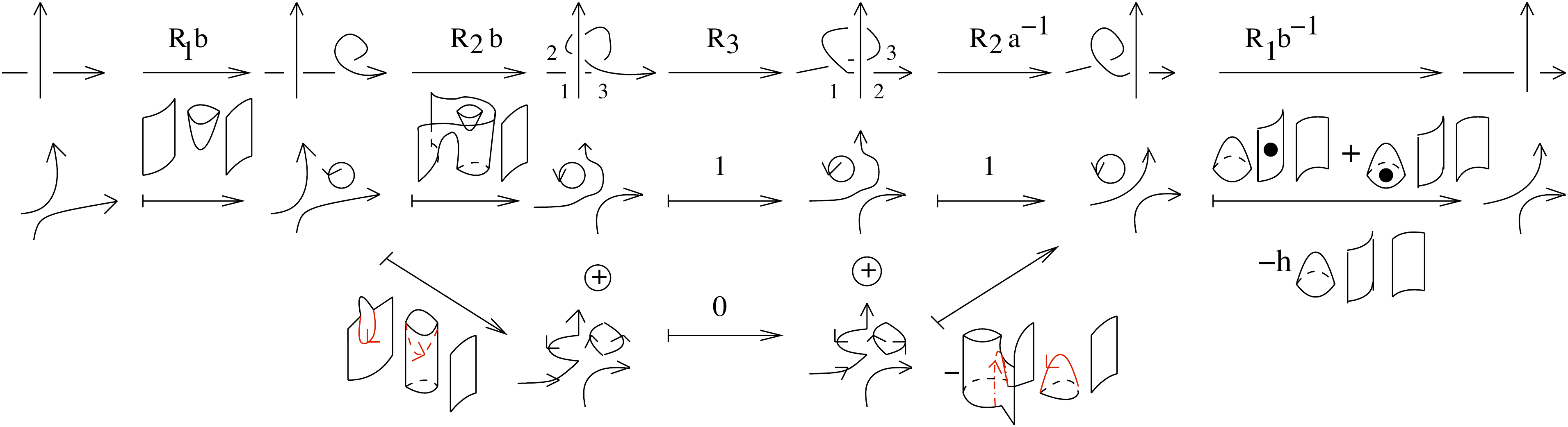}\]	
This time we obtain the zero map in the second row and the identity in the first row:
\[\left (\,\raisebox{-10pt}{\includegraphics[height=0.45in]{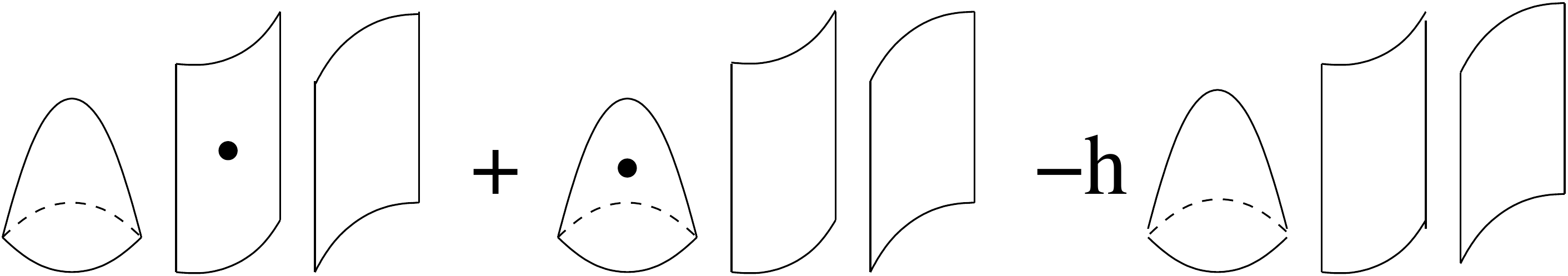}}\,\right) \circ \left(\,\raisebox{-10pt}{\includegraphics[height=0.45in]{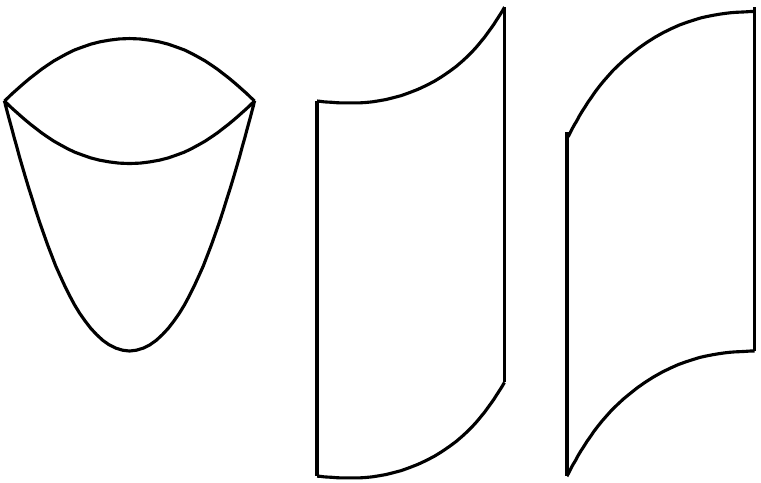}}\,\right) = \raisebox{-10pt}{\includegraphics[height=0.45in]{MM8-maps5}}\]

In both cases, the induced map at the chain level is the identity.

$$\text{MM 12} \quad\raisebox{-28pt}{ \includegraphics[height=1.34in]{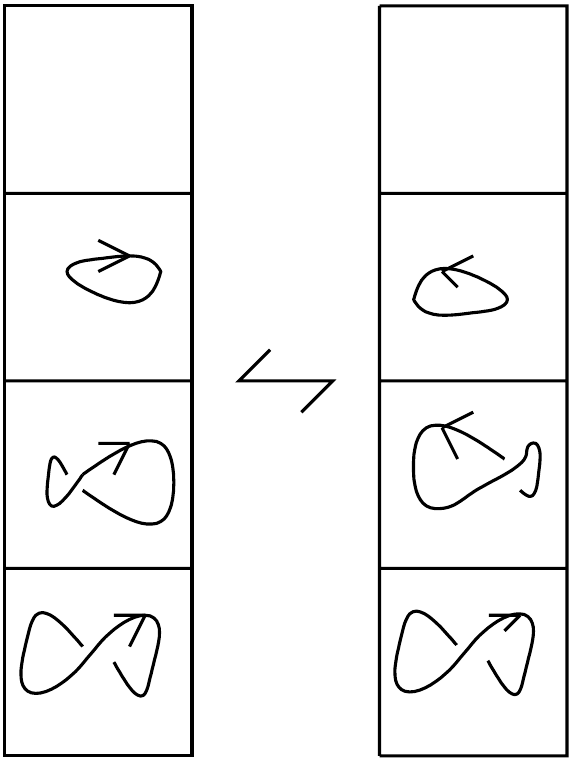}} \hspace{1cm} \raisebox{-28pt}{\includegraphics[height=.95in]{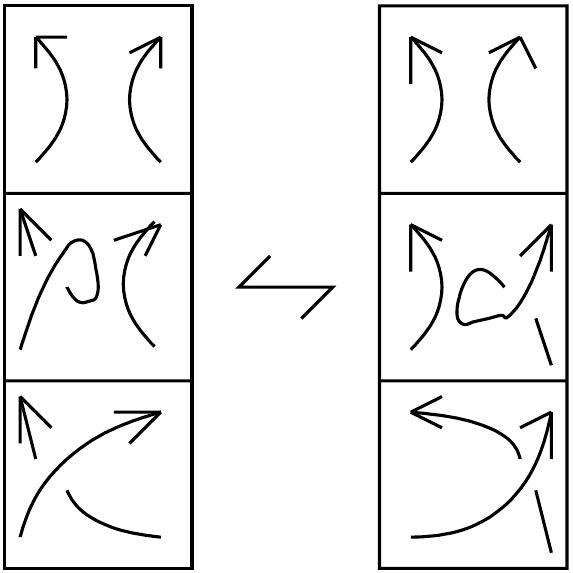}}\quad \text{MM 13}$$

Each pair of clips in MM12 and MM13  should produce the same morphisms when read from top to bottom or from bottom to top. 

\textbf{MM12.}\quad
Going down the left side of MM12 we have
$ \,\, \raisebox{-3pt}{\includegraphics[height=0.3in]{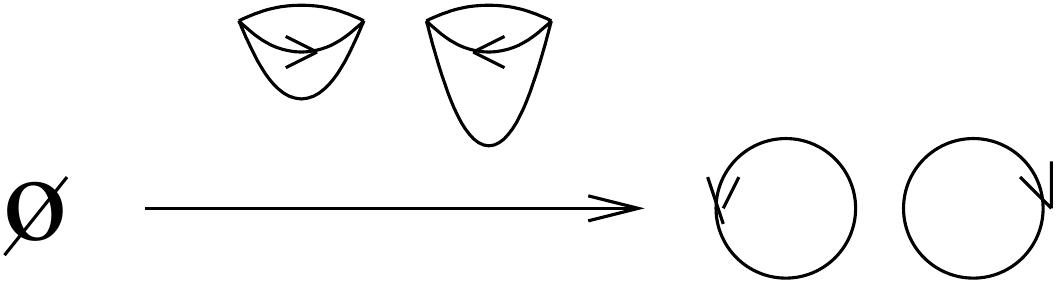}},$
while going down the right side we obtain
$\,\, \raisebox{-3pt}{\includegraphics[height=0.3in]{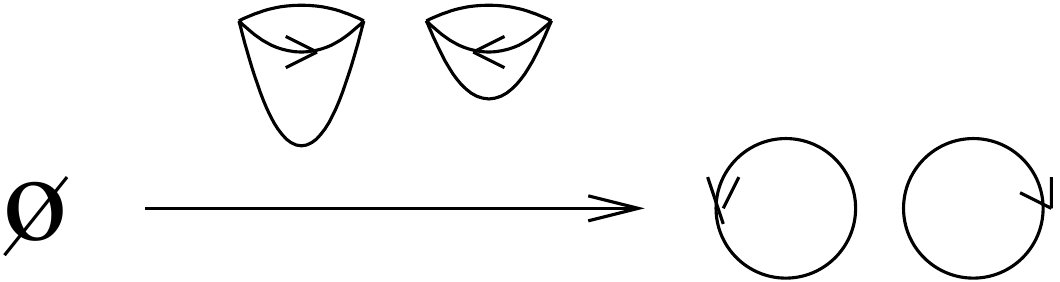}}.$
But these two cobordisms are isotopic. Going up along each side of MM12, the corresponding morphism is $ (\raisebox{-3pt}{\includegraphics[height=0.15in]{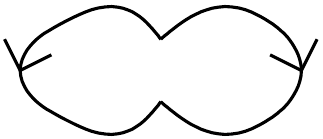}}, 
\raisebox{-3pt}{\includegraphics[height=0.15in]{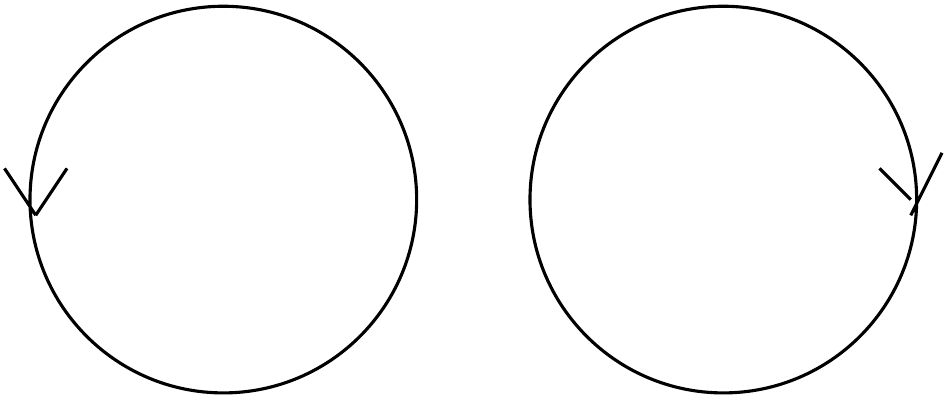}}) \longrightarrow \emptyset$, which is the zero map on the first component, and on the second one is (on the left and right side of the clip, respectively):
$$\raisebox{-5pt}{\includegraphics[height=0.4in]{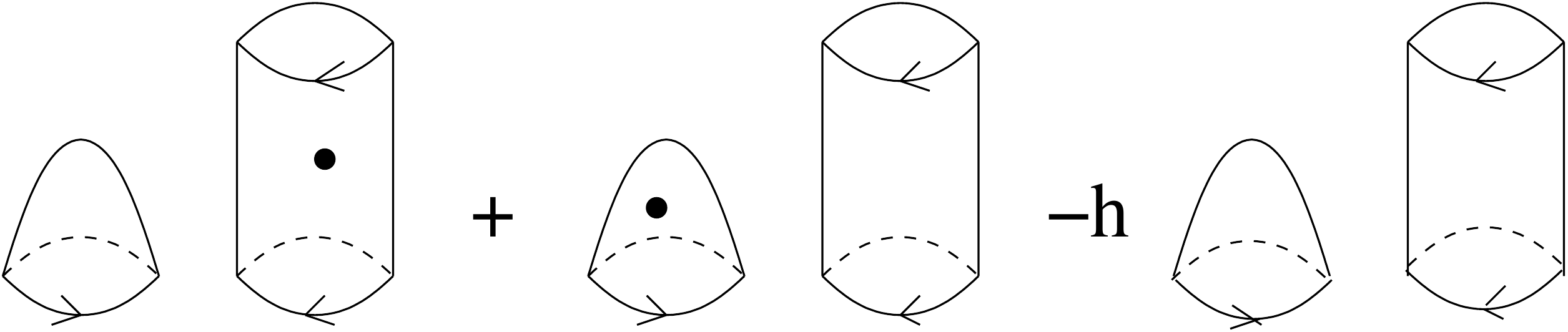}}\quad \text {and} \quad \raisebox{-5pt}{\includegraphics[height=0.4in]{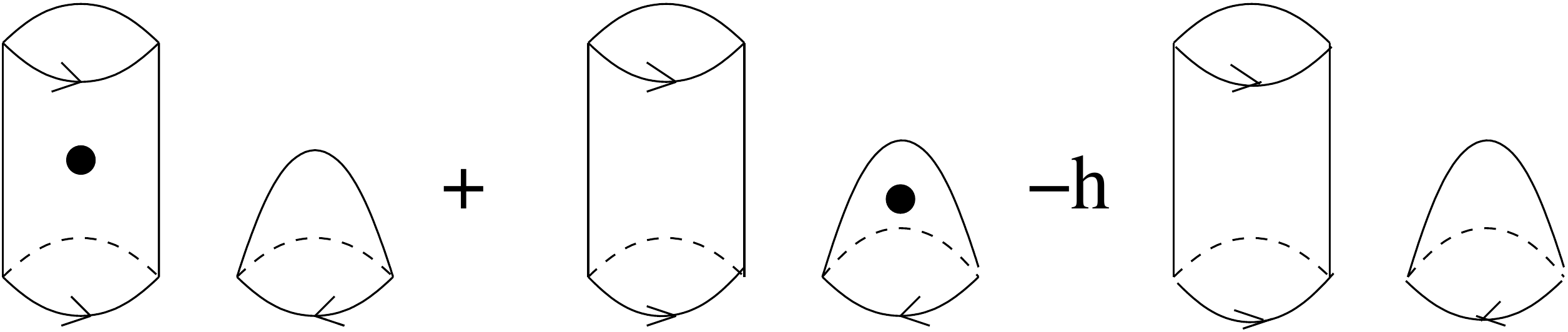}}$$
followed by $\raisebox{-5pt}{\includegraphics[height=0.18in]{caplo}}.$ Both cobordisms are isotopic to  \,\raisebox{-13pt}{\includegraphics[height=0.35in]{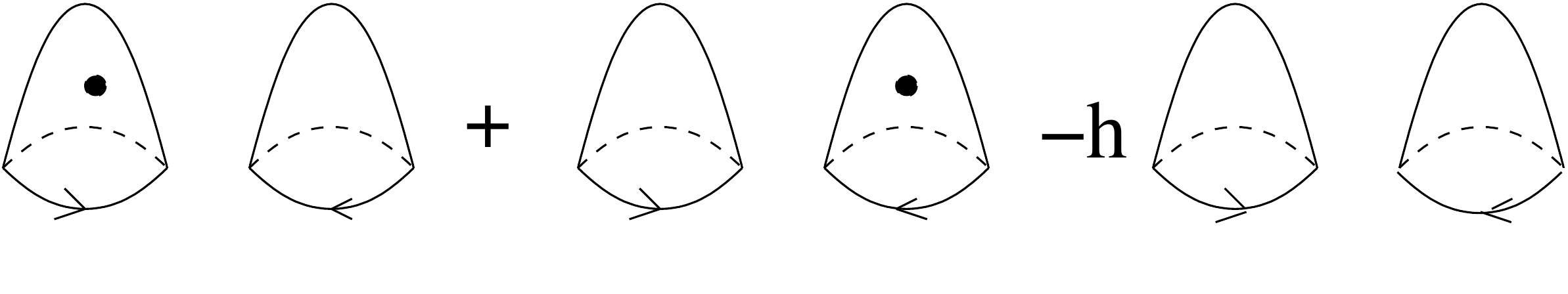}}.

The calculations for the mirror image are similar. Going up along the clip, both maps are a disjoint union of cups on the oriented resolution, and zero on the other one. Going down, we get on both sides morphisms that are isotopic to \raisebox{-8pt}{\includegraphics[height=0.25in]{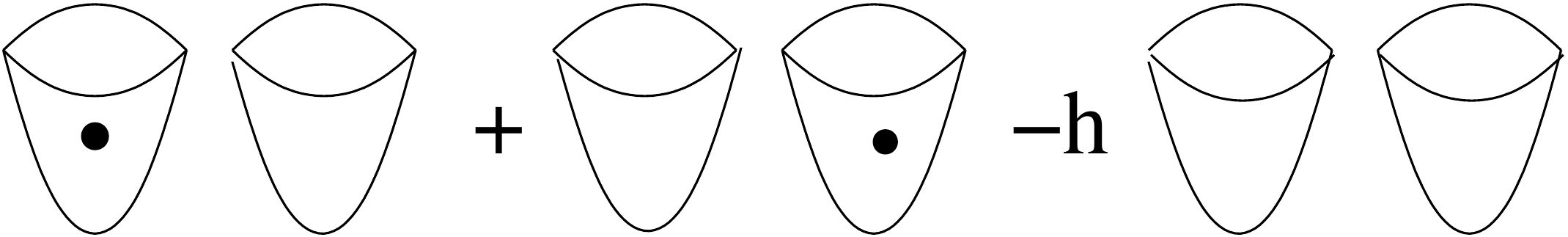}}.

\textbf{MM13.}\quad
Going down along the clip we have on the left and right, respectively:

\[ \raisebox{-8pt}{\includegraphics[height=0.5in]{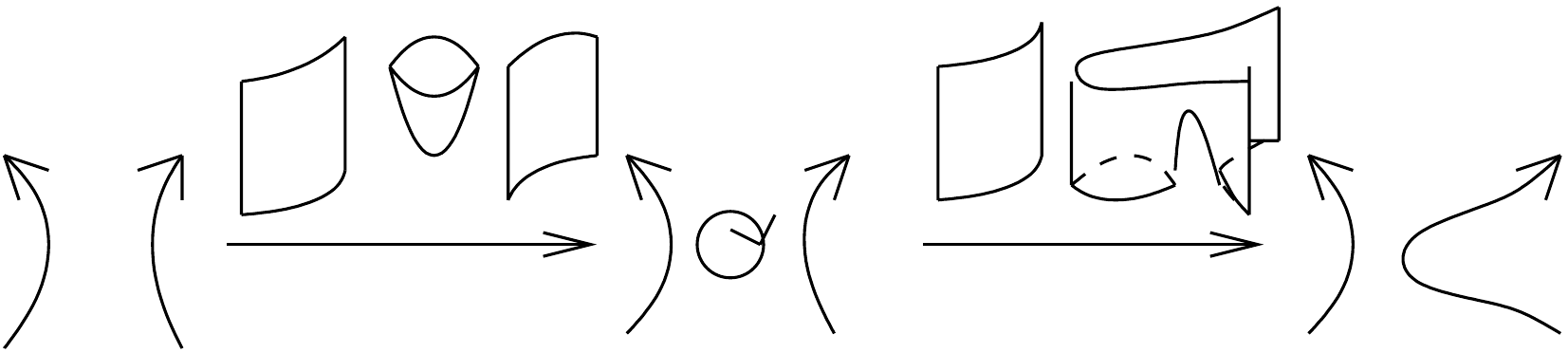}}\quad \text{and}\quad \raisebox{-8pt}{\includegraphics[height=0.5in]{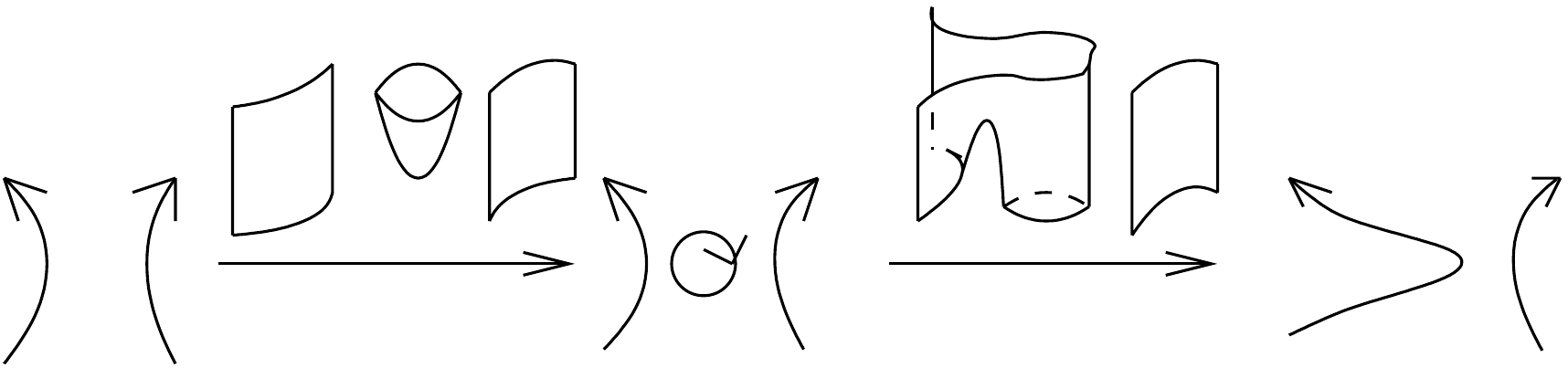}}\]

After composing these cobordisms we obtain on both sides two vertical curtains.
Going up along the clip, both maps are zero on the singular resolution, and on the oriented resolution we have $\raisebox{-15pt}{\includegraphics[height=0.4in]{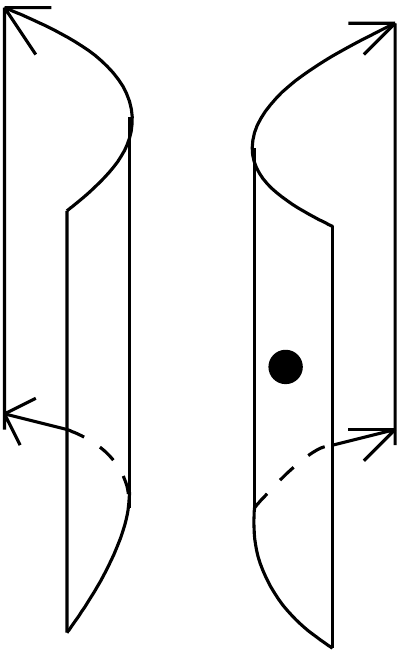}} + \raisebox{-15pt}{\includegraphics[height=0.4in]{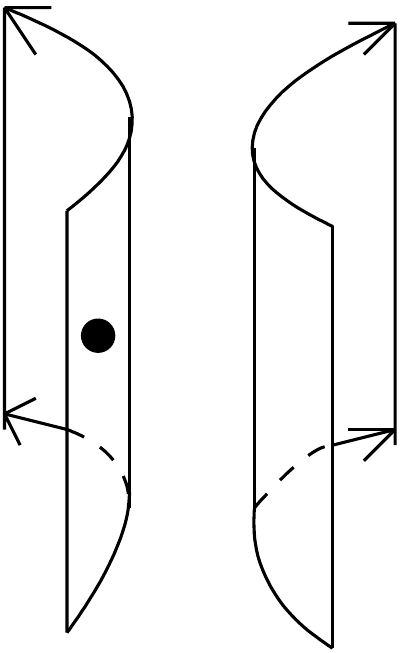}} -h \raisebox{-15pt}{\includegraphics[height=0.4in]{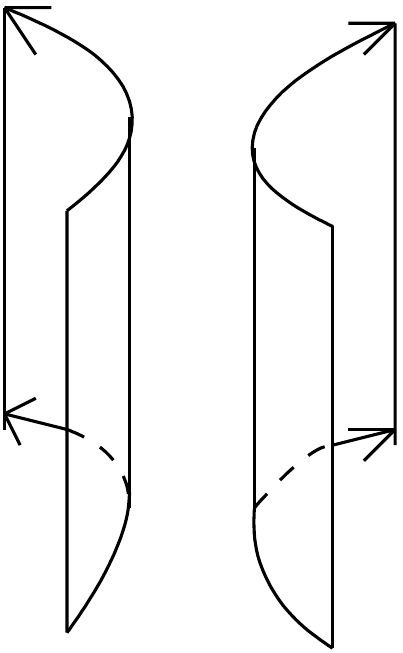}}$ on the left side of the clip, and $\raisebox{-15pt}{\includegraphics[height=0.4in]{MM13-2-new}} + \raisebox{-15pt}{\includegraphics[height=0.4in]{MM13-1-new}} -h \raisebox{-15pt}{\includegraphics[height=0.4in]{MM13-3-new}}$ on the right side. These cobordisms are the same.
		
For the mirror image we obtain similar results, with the difference that the two vertical curtains appear when we read the clip from bottom to top. 
\end{proof}


\section{\textbf{The algebraic invariant}}\label{sec:webhom}

\begin{definition}\label{def:web homology}
Let $\Gamma_0$ be an arbitrary web with boundary $B$ (if $B$ is empty, $\Gamma_0 = \emptyset$). We define a functor $ \mathcal{F}_{\Gamma_0} \co \textit{Foams}_{/\ell}(B) \rightarrow \mathbb{Z}[i][a,h]$-Mod as follows: 
 \begin{itemize}
 \item if $\Gamma \in \textit{Foams}_{/\ell}(B),\,\text{define} \,
 \mathcal{F}_{\Gamma_0}(\Gamma) = \Hom_{\textit{Foam}_{/\ell}(B)}(\Gamma_0, \Gamma)$
 \item if $S \in \Hom_{\textit{Foam}_{/\ell}(B)}(\Gamma', \Gamma'')$, define $ \mathcal{F}_{\Gamma_0}(S)$ as the $\mathbb{Z}[i][a, h]$-linear map 
 \[\Hom_{\textit{Foam}_{/l}(B)}(\Gamma_0, \Gamma') \rightarrow \Hom_{\textit{Foam}_{/\ell}(B)}(\Gamma_0, \Gamma'') \, \text{given\ by\ composition}.\] \end{itemize}
\end{definition}

Note that $ \mathcal{F}_{\Gamma_0}(\Gamma_1 \cup \Gamma_2) \cong \mathcal{F}_{\Gamma_0}(\Gamma_1) \otimes_{\mathbb{Z}[i][a,h]} \mathcal{F}_{\Gamma_0}(\Gamma_2)$ for any disjoint union of webs $\Gamma_1$ and $\Gamma_2.$

\begin{proposition}\label{prop:categorified web relations} The functor $\mathcal{F}$ mimics the web skein relations of Figure~\ref{fig:web skein relations}. \newline
Specifically, there are canonical isomorphisms of graded abelian groups:
\begin{enumerate}
\item $\xymatrix@R=2mm{
\mathcal{F}_{\emptyset}(\raisebox{-3pt}{\includegraphics[height=0.18in]{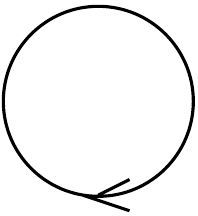}}) \cong \mathcal{A} \cong 
\mathcal{F}_{\emptyset}(\raisebox{-3pt}{\includegraphics[height=0.16in]{circle2sv.pdf}})}$
 \item $\mathcal{F}_{\Gamma_0}(\Gamma \cup  \raisebox{-3pt}{\includegraphics[height=0.18in]{unknot-clockwise.pdf}}) \cong \mathcal{F}_{\Gamma_0}(\Gamma) \otimes_{\mathbb{Z}[i][a,h]} \mathcal{A} \cong\mathcal{F}_{\Gamma_0}(\Gamma \cup  \raisebox{-3pt}{\includegraphics[height=0.16in]{circle2sv.pdf}})$
\item $\mathcal{F}_{\Gamma_0}(\raisebox{-3pt}{\includegraphics[height=0.12in]{2vertweb.pdf}}) \cong \mathcal{F}_{\Gamma_0}(\raisebox{-3pt}{\includegraphics[height=0.12in]{arcro.pdf}}) \quad \text{and} \quad
\mathcal{F}_{\Gamma_0}(\raisebox{-3pt}{\includegraphics[height=0.12in]{2vertwebleft.pdf}} ) \cong \mathcal{F}_{\Gamma_0}(\raisebox{-3pt}{\includegraphics[height=0.12in]{arclo.pdf}}).$
 \end{enumerate}
 In particular, $\mathcal{F}_{\emptyset}(\Gamma)$ is a free $\mathbb{Z}[i][a,h]$-module of graded rank $\brak{\Gamma}.$
\end{proposition}

\begin{corollary}
The functor $\mathcal{F}_{\emptyset}$ is the same as the functor $\mathsf{F}$ defined in Section~\ref{sec:TQFT}.\end{corollary}

The functor $\mathcal{F}$ extends in a straightforward manner to the category \textit{Kof}, and $\mathcal{F}([T])$ is a complex of graded free $\mathbb{Z}[i][a, h]$-modules which is an invariant of the tangle $T,$ up to cochain homotopy. Moreover, $\mathcal{F}$ is degree-preserving thus the homology $H(\mathcal{F}([T]))$ is a bigraded invariant of $T,$ denoted by $\mathcal{H}(T) = \oplus_{i,j} \mathcal{H}^{i,j}(T).$ If $T$ a link diagram $L,$ the graded Euler characteristic of $\mathcal{F}([L])$ equals the quantum $\mf{sl}(2)$-polynomial of $L.$ In other words, 
\[P_2(L) = \sum_{i,j \in \mathbb{Z}}(-1)^i q^j \,\text{rk} (\mathcal{H}^{i,j}(L)).\]

\subsection{The invariant of a surface-knot}

Given a link cobordism $C \subset \mathbb{R}^2 \times [0,1]$ between links $L_0$ and $L_1,$ there is an induced graded map $\mathcal{L}_C \co \mathcal{H}(L_0) \to \mathcal{H}(L_1)$ of degree $-\chi(C),$ well-defined under ambient isotopy of $C$ relative to its boundary. 

A \textit{surface-knot} or \textit{surface-link} $S$ is a closed surface embedded in $\mathbb{R}^4$ locally flatly, and it can be regarded as a link cobordism between empty links. The induced map $\mathcal{L}_S \co \mathcal{H}(\emptyset) \to \mathcal{H}(\emptyset)$ is a ring homomorphism $\mathbb{Z}[i][a,h] \to \mathbb{Z}[i][a,h],$ giving rise to an invariant of the surface-link $S,$ defined as $\text{Inv}(S) = \mathcal{L}_S(1) \in \mathbb{Z}[i][a,h].$ In the remaining part of this section we show that the invariant of any surface-link is determined by its genus. In doing this, we follow Tanaka's~\cite{Ta} approach to the surface-knot invariant derived from Bar-Natan's theory~\cite{BN1}. 

A surface-knot in called \textit{trivial} or \textit{unknotted} if it's obtained from some standard surfaces in $\mathbb{R}^4$ by taking a connected sum. By the results of Hosokawa and Kawauchi~\cite{HK}, any surface-knot becomes trivial by attaching a finite number of 1-handles (the \linebreak minimal number of such 1-handles is called the \textit{unknotting number}). Moreover, it is known that any 1-handle on a surface-knot is ribbon-move equivalent to a trivial 1-handle.

Consider the two movies shown in Figure~\ref{fig:ribbon-move}. From our construction one can observe that the maps of formal complexes $[\,\raisebox{-5pt}{\includegraphics[height=0.2in]{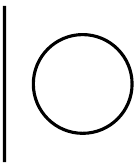}}\,] \to [\,\raisebox{-5pt}{\includegraphics[height=0.2in]{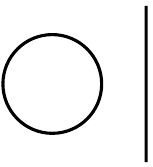}}\,]$, in particular, the corresponding homomorphisms $\mathcal{H}(\,\raisebox{-5pt}{\includegraphics[height=0.2in]{ribbon-move-left.pdf}}\,) \to \mathcal{H}(\,\raisebox{-5pt}{\includegraphics[height=0.2in]{ribbon-move-right.pdf}}\,)$ are the same for these movies. 
\begin{figure}[ht!]\begin{center}
\includegraphics[height=0.5in]{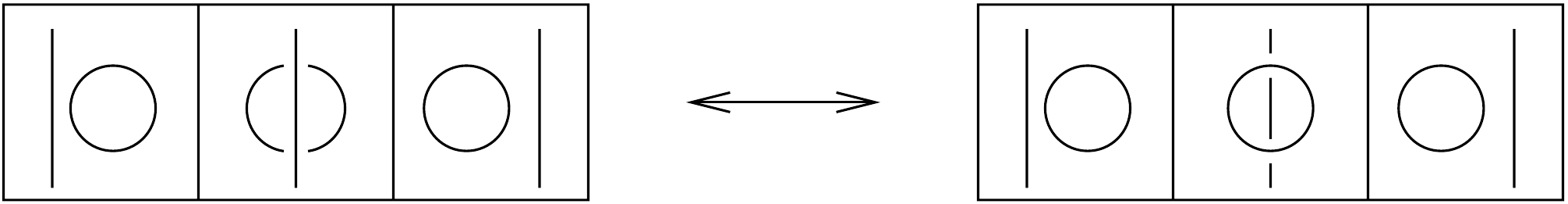}\end{center}
\caption{}\label{fig:ribbon-move}
\end{figure}

By the work of Carter, Saito and Satoh~\cite{CSS}, it is implied that two surface-knots which are related by  ribbon-moves have the same invariant obtained from our theory.

The following lemma is a direct consequences of the local relations $\ell$ of Section~\ref{sec:relations l}. 
\begin{lemma}\label{lemma:inv-trivial}
If the surface-knot $S$ of genus $g$ is trivial then

$\text{Inv}(S) = 
\begin{cases} 0, \hspace{.15in} \text{if} \hspace{.1in} g = 2k, \,\, k\geq 0 \\
2(h^2 + 4a)^k, \hspace{.15in} \text{if} \hspace{.1in} g = 2k+1,\,\, k\geq 0. \end{cases}$
\end{lemma}

Any surface-knot $S$ can be regarded as the composition between $S$ with one puncture and the ``cap'' cobordism. In particular, the surface $S$ with one puncture can be considered as a cobordism from the empty link to the trivial knot.  We adopt some notations from~\cite{Ta} and write
\[
\mathcal{L}_S = \epsilon \circ\mathcal{L}_S^{(\emptyset  \to \bigcirc)},\hspace{.15in}
\text{where} \hspace{.1in} \mathcal{L}_S^{(\emptyset  \to \bigcirc)} \co \mathcal{H} (\emptyset) \to \mathcal{H} (\bigcirc).\]
Here $\bigcirc$ is the trivial knot diagram. Similarly, it also holds the following
\[
\mathcal{L}_S = \mathcal{L}_S^{(\bigcirc \to \emptyset)} \circ \iota, \hspace{.15in} \text{where} \hspace{.1in} \mathcal{L}_S^{(\bigcirc  \to \emptyset)}: \mathcal{H} (\bigcirc) \to \mathcal{H} (\emptyset).
\]
For the connected sum $S_1 \sharp  S_2$ of two surface-knots $S_1$ and $S_2$ we write
\[
\mathcal{L}_{S_1 \sharp S_2} = \mathcal{L}_{S_2}^{(\bigcirc \to \emptyset)} \circ \mathcal{L}_{S_1}^{(\emptyset \to \bigcirc)}.
\]

\begin{lemma}\label{lemma:inv-trivial-evengenus}
If the surface-knot $S$ of genus $2k$ (where $k\geq 0$) is trivial then 
\[ \mathcal{L}_S^{(\bigcirc \to \emptyset)} (X) = (h^2 + 4a)^k\,\, \mbox{and} \,\,\, \mathcal{L}_S^{(\bigcirc \to \emptyset)} (1) = 0,\]
where $X$ and $1$ are the generators of the algebra $\mathcal{A}.$
\end{lemma}
\begin{proof}
The genus-reduction formula implies that 
\[\mathcal{L}_S^{(\bigcirc \to \emptyset)} (X) = (h^2 + 4a)^k \epsilon(X) = (h^2 + 4a)^k \,\,\mbox{and}\,\, \mathcal{L}_S^{(\bigcirc \to \emptyset)} (1) = (h^2 + 4a)^k \epsilon(1) = 0.\]
\end{proof}

\begin{theorem}
For any surface-knot $S$ of genus $g,$ the following holds.
\begin{enumerate}
\item If $g$ is even, then $\text{Inv}(S) = 0,$
\item If $g$ is odd, then $\text{Inv}(S) = 2(h^2 + 4a)^{\frac{g-1}{2}}.$
\end{enumerate}
\end{theorem} 
\begin{proof}
Assume that $\text{Inv}(S) = p(a,h),$ for some $p(a,h) \in \mathbb{Z}[i][a,h]$. Using the definition of $\epsilon,$ we obtain that 
$\mathcal{L}_S^{(\emptyset \to \bigcirc)}(1) = p(a,h)X + q(a,h),$ for some $q(a,h) \in\mathbb{Z}[i][a,h].$
Notice that if $g = 0$ then $p(a, h) = 0,$ since $\mathcal{L}_S$ is a map of degree $-\chi(S) = -2,$ and $\deg(a) = 4$ and $\deg(h) = 2.$

Let $\Sigma_{g'}$ be a trivial surface-knot of genus $g',$ and consider the connected sum $S \sharp  \Sigma_{2k'}$ for $k'\geq 0.$ Using Lemma~\ref{lemma:inv-trivial-evengenus} we obtain
\[
\mathcal{L}_{S \sharp \Sigma_{2k'}}(1) =( \mathcal{L}_{\Sigma_{2k'}}^{(\bigcirc \to \emptyset)} \circ \mathcal{L}_S^{(\emptyset \to \bigcirc)})(1) =  \mathcal{L}_{\Sigma_{2k'}}^{(\bigcirc \to \emptyset)}(p(a, h)X + q(a,h)) = p(a, h)(h^2 + 4a)^{k'}.\]
If we consider the integer $k'$ such that $2k'$ is greater than the unknotting number of $S,$ then the surface-knot $S \sharp \Sigma_{2k'}$ is ribbon-move equivalent to the trivial surface-knot $\Sigma_{g + 2k'}.$ Therefore by Lemma~\ref{lemma:inv-trivial} we conclude that
 \[\text{Inv}(S \sharp \Sigma_{2k'}) = \begin{cases}
0, \hspace{.15in} \text{if} \hspace{.1in} g = 2k \\
2(h^{2} + 4a)^{k+k'}, \hspace{.15in} \text{if} \hspace{.1in} g = 2k+1,\end{cases}\]
which implies that $p(a,h) = 0$ if $g = 2k,$ and $p(a,h) = 2(h^2 + 4a)^k$ if $g = 2k + 1.$
\end{proof}
\begin{corollary}
For any torus-knot $T^2$ we have $\text{Inv}(T^2) = 2.$ 
\end{corollary}
We remark that the above Corollary is a generalization of a similar result obtained by Tanaka~\cite[Corollary 1.2.]{Ta} and, independently, by Rasmussen~\cite{R}.

\section{\textbf{The universal link homology over $\mathbb{C}$}}

In this section we let $a$ and $h$ to be complex numbers and consider the universal $\mf{sl}(2)$-link cohomology over $\mathbb{C},$ denoted by $\mathcal{H}(L,\mathbb{C}).$ Let $f(X) = X^2 -hX -a \in \mathbb{C}[X].$ For a given choice of $a,h \in \mathbb{C},$ the isomorphism class of $\mathcal{H}(L,\mathbb{C})$ is determined by the number of distinct roots of $f(X).$

If $f(X) = (X - \alpha)^2,$ for some $\alpha \in \mathbb{C},$ there is an isomorphism between $\mathcal{H}(L,\mathbb{C})$ and the Khovanov's original $\mf{sl}(2)$-link homology over $\mathbb{C},$ induced by the isomorphism 
$$\raisebox{-3pt}{\includegraphics[height=0.15in]{plane1d.pdf}}\to \raisebox{-3pt}{\includegraphics[height=0.15in]{plane1d.pdf}} - \alpha \raisebox{-3pt}{\includegraphics[height=0.15in]{plane.pdf}}.$$

\textbf{Two distinct roots.}\quad
Let us assume that $f(X) = (X - \alpha)(X - \beta),$ for some $\alpha, \beta \in \mathbb{C}.$ Therefore $\alpha + \beta = h,$ and $\alpha \beta = -a.$

We study this case in a similar way as Mackkay and Vaz did in~\cite[Section 3.1]{MV}. Moreover, the reader will notice similarities with the work by Gornik~\cite{G}.

Given the algebra $\mathbb{C}[X]/(f(X)),$ there is an isomorphism of algebras
\[ \mathbb{C}[X]/(f(X)) \cong \mathbb{C}[X]/(X-\alpha) \oplus \mathbb{C}[X]/(X-\beta) \cong \mathbb{C}^2. \]
Let $\Gamma$ be a resolution of a link $L$ and denote by $e(\Gamma)$ the set of all edges in $\Gamma.$ 
\begin{definition}
Let $R(\Gamma)$ be the free commutative algebra generated by elements $X_j, \,j \in e(\Gamma)$ with relations $ X_j + X_k = h$ and $X_j X_k = -a$ for any pair of edges $j,k$ that meet at a bivalent vertex.
\end{definition}
Note that for each $j \in e(\Gamma)$ we have $f(X_j) = 0,$ thus there exists an algebra homomorphism $\mathbb{C}[X]/(f(X)) \to R(\Gamma)$ defined by $X \to X_j.$

\begin{definition}\label{def:coloring}
Let $S = \{\alpha, \beta \}.$ A \textit{coloring} of $\Gamma$ is a map $\phi \co e(\Gamma) \to S,$ and an \textit{admissible coloring} is a coloring that satisfies
\[ \phi(j) + \phi(k) = h, \quad \phi(j) \phi(k) = -a\] for all edges $j,k$ meeting at a bivalent vertex. Denote by $C(\Gamma)$ the set of all colorings and by $AC(\Gamma)$ the set of all admissible colorings of $\Gamma.$ \end{definition}

\begin{lemma} \label{lemma:identities for Q}
For any coloring $\phi$ we define $$Q_\phi (\Gamma) = \displaystyle \prod_{j \in e(\Gamma)} Q_{\phi(j)}(X_j) \in R(\Gamma), \quad \text{where} \quad Q_\alpha(X) = \displaystyle \frac{X - \beta}{X - \alpha}, \,\,Q_\beta(X) = \displaystyle \frac{X - \alpha}{X - \beta}\,.$$
Then the following relations hold.
\begin{enumerate}
\item $Q_\alpha(X) + Q_\beta(X) =1, \quad Q_\alpha(X) Q_\beta(X) =0$
\item  $Q_\alpha(X)^2 = Q_\alpha(X), \quad Q_\beta(X)^2 = Q_\beta(X)$
\item $\displaystyle \sum _{\phi \in C(\Gamma)} Q_\phi (\Gamma) = 1, \quad Q_\phi(\Gamma) Q_\psi (\Gamma) = \delta_{\phi \psi}(\Gamma)$, where $ \delta_{\phi \psi}$ is the Kronecker delta
 \item $X_j Q_\phi(\Gamma) = \phi(j) Q_\phi(\Gamma).$
\end{enumerate}\end{lemma}

The next result is analogous to~\cite[Theorem 3]{G} and~\cite[Lemma 3.7]{MV}. 
\begin{lemma}\label{lemma:decomposed-alg}
For any non-admissible coloring $\phi,$ we have $Q_\phi(\Gamma) = 0.$ 

For any admissible coloring $\phi,$ we have $Q_\phi(\Gamma) R(\Gamma) \cong \mathbb{C} \Longrightarrow \dim Q_\phi(\Gamma) R(\Gamma) =1.$

Therefore, the following direct sum decomposition of $\mathbb{C}$-algebras holds
\[ R(\Gamma) \cong \displaystyle \bigoplus_{\phi \in AC(\Gamma)} Q_\phi(\Gamma) \mathbb{C}. \] 
\end{lemma}
\begin{proof}
Let $\phi$ be any coloring and $j,k \in e(\Gamma)$ be the labels corresponding to a vertex. By relation (4) in the previous lemma and relations in the algebra $R(\Gamma),$ we get
\begin{align*} (\phi(j) + \phi(k)) Q_\phi(\Gamma) \,&= \,h Q_\phi(\Gamma) \\
\phi(j) \, \phi(k) Q_\phi(\Gamma) \,&= \,-a Q_\phi(\Gamma) \end{align*}
If $\phi$ is non-admissible, one of the relations $\phi(j) + \phi(k) = h,\, \phi(j) \phi(k) = -a$ does not hold, thus $Q_\phi(\Gamma) = 0.$

Now assume that $\phi$ is admissible. Then $Q_\phi(\Gamma) \neq 0.$ For, if $Q_\phi(\Gamma)$ was $0$ then its representative in $ \otimes_{j \in e(\Gamma)}\mathbb{C}[X_j] $ would be in the ideal generated by the relations that define the algebra $R(\Gamma).$ But these relations give $0$ when evaluated at $X_j = \phi(j)$, while $Q_\phi(\Gamma)$ evaluates to $1$ at $X_j = \phi(j).$ Consider $R'(\Gamma) = \otimes _{j \in e(\Gamma)} \mathbb{C}[X_j]/(f(X_j))$ and $Q'_\phi(\Gamma) \in R'(\Gamma)$ given by the same expressions as $Q_\phi(\Gamma).$ Note that Lemma~\ref{lemma:identities for Q} holds for $Q'_\phi(\Gamma)$ as well, and that they form a basis for $R'(\Gamma).$ Thus $Q'_\phi(\Gamma) R'(\Gamma) = Q'_\phi(\Gamma) \mathbb{C},$ and the same holds in $R(\Gamma)$ since it is a quotient of $R'(\Gamma).$ \end{proof}

Relations (ED) show that $R(\Gamma)$ acts on $\mathcal{F}(\Gamma)$ by the web cobordism of merging a circle into an edge of $\Gamma.$ From Lemma~\ref{lemma:identities for Q} and Lemma~\ref{lemma:decomposed-alg} we have
\[ \mathcal{F}(\Gamma) = \displaystyle \bigoplus_{\phi \in AC(\Gamma)} Q_\phi(\Gamma)\mathcal{F}(\Gamma). \]
For all $x \in \mathcal{F}(\Gamma)$ we have $\, x \in Q_\phi(\Gamma)\mathcal{F}(\Gamma) \Longleftrightarrow X_j x = \phi(j)x, \quad \forall j \in e(\Gamma).$ Moreover, using an inductive argument on the number of vertices in $\Gamma,$ and the results from Proposition~\ref{prop:categorified web relations}, we have that the $\mathbb{C}$-space $Q_\phi(\Gamma)\mathcal{F}(\Gamma)$ is one-dimensional, for any $\phi \in AC(\Gamma).$

\begin{definition} Let $L_{\phi}$ be a \textit{colored link} with its arcs colored by $\alpha$ and $\beta.$ Following Mackaay and Vaz~\cite{MV}, we say that a coloring $\phi$ of $L$ is \textit{admissible} if there exists a resolution of $L$ which admits an admissible coloring that is compatible with the coloring of $L.$ We denote by $AC(L)$ the set of all admissible colorings of $L.$ We also say that an admissible coloring of $L$ is a \textit{canonical coloring} if the arcs belonging to the same component of $L$ have the same color. We denote by $C(L)$ the set of canonical colorings of $L.$\end{definition}

We also denote by  $\mathcal{H}(L_{\phi}, \mathbb{C})$ the cohomology over $\mathbb{C}$ of the colored link $L_\phi,$ induced by the spaces $Q_{\phi}(\Gamma)\mathcal{F}(\Gamma),$ for all resolutions $\Gamma$ of the given link. 
\begin{proposition}
If $\phi \in AC(L) \diagdown C(L)$ then $\mathcal{H}(L_{\phi}, \mathbb{C}) = 0.$ Therefore $$\mathcal{H}(L, C) = \displaystyle \bigoplus _{\phi \in C(L)} \mathcal{H}(L_{\phi}, \mathbb{C}).$$

\end{proposition}
\begin{proof} We only sketch the proof, since it is similar to the proof of~\cite[Theorem 3.9]{MV}. The differences are that it uses our relations (RSC) and (CN).

Consider the diagrams $\Gamma =\raisebox{-10pt} {\includegraphics[height=0.4in]{singres.pdf}}$ and $\Gamma' = \raisebox{-10pt} {\includegraphics[height=0.4in]{orienres.pdf}}.$ Up to permutation, the admissible colorings of $\Gamma$ are $$\phi_1 = \raisebox{-19pt}{\includegraphics[height=.6in]{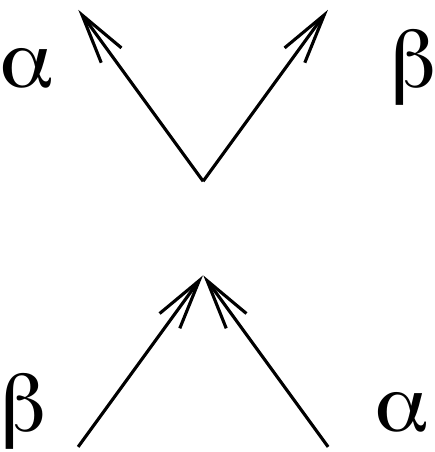}} \qquad \mbox{and} \qquad \phi_2 = \raisebox{-19pt}{\includegraphics[height=0.6in]{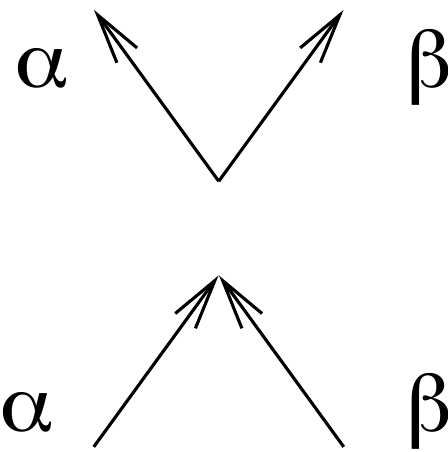}}$$
and up to permutation, the admissible colorings of $\Gamma'$ are $$\phi_1' = \raisebox{-19pt}{\includegraphics[height=.6in]{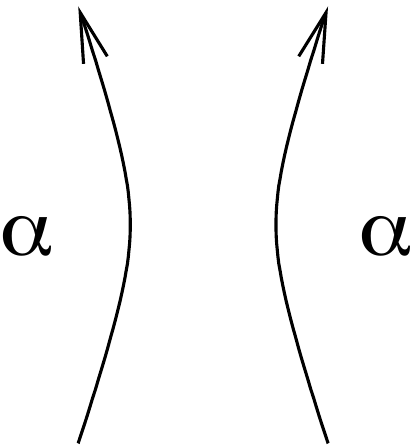}} \qquad \mbox{and} \qquad \phi_2' = \raisebox{-19pt}{\includegraphics[height=0.6in]{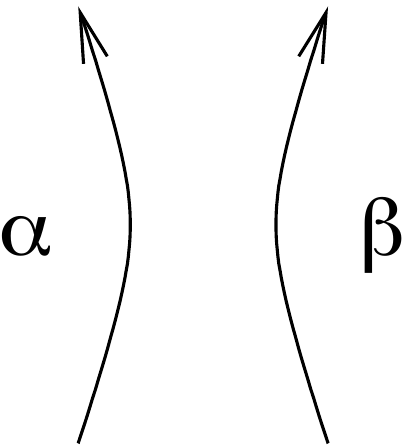}}\,.$$
The elementary cobordisms (the piecewise oriented saddles depicted in Figure~\ref{fig:saddle}) have to map colorings to compatible colorings. Therefore $$ Q_{\phi_1}(\Gamma)\mathcal{F}(\Gamma) \to 0, \qquad Q_{\phi_1'}(\Gamma')\mathcal{F}(\Gamma') \to 0
.$$
Consider the cobordisms $\Gamma \to \Gamma' \to \Gamma$ and $\Gamma' \to \Gamma \to \Gamma'$ and use relations (CN) and (RSC) respectively, to show that both maps are isomorphisms. Therefore $$Q_{\phi_2}(\Gamma)\mathcal{F}(\Gamma) \cong Q_{\phi_2'}(\Gamma')\mathcal{F}(\Gamma) .$$
In conclusion, from the boundary map behaviour explained above, the colorings that survive in the cohomology $\mathcal{H}(L, \mathbb{C})$ are (up to permutation) $\phi_1$ for $\Gamma$ and $\phi_1'$ for $\Gamma',$ which are exactly those obtained from canonical colorings of $L.$ \end{proof}

Given an $n$-component link $L,$ there are $2^n$ canonical colorings of $L,$ and each such coloring $\phi$ defines precisely one resolution, namely, the one obtained by resolving to the singular resolution all crossings at which $\phi$-values of the two strands are different, and resolving to the oriented resolution all crossings at which $\phi$-values of the two strands are equal. Note that the cohomological degree of the resolution determined by some $\phi$ is easy to compute, since only the singular resolution contributes to the cohomological degree, and this contribution is $-1$ for positive crossings and $1$ for negative crossings. Hence the following result holds.

\begin{theorem}
For any $n$-component link $L,$ the dimension of $\mathcal{H}(L, \mathbb{C})$ equals $2^n,$ and to each map $\phi \co \{ \text{components of L} \} \to S = \{ \alpha, \beta \}$ there exists a non-zero element $h_{\phi} \in \mathcal{H}(L, \mathbb{C})$ which lies in the cohomological degree
$$ -2 \sum_{\substack {(u_1,u_2)\in S \times S \\ u_1 \neq u_2}} lk(\phi^{-1}(u_1), \phi^{-1}(u_2)).$$
All $h_{\phi}$ generate $\mathcal{H}(L, \mathbb{C}).$
\end{theorem}

\textbf{Acknowledgements.} The author would like to thank the referee for reading the manuscript carefully and making valuable suggestions.

\end{document}